\documentclass[11pt]{amsart} \textwidth=14.5cm \oddsidemargin=1cm
\usepackage{amsmath,wasysym}
\usepackage{amsxtra}
\usepackage{amscd}
\usepackage{amsthm}
\usepackage{amsfonts}
\usepackage{amssymb}
\usepackage{eucal}
\usepackage{graphics}
\usepackage{hyperref,mathrsfs}
\usepackage[usenames,dvipsnames]{xcolor}
\usepackage{tikz}
\usepackage{tikz-cd}
\usepackage{stmaryrd}
\usepackage[all]{xy}
\usepackage{hyperref}
\usepackage{color}
\usepackage{bbm} % \mathbbm
\usepackage{upgreek}
\usepackage{extarrows}
\allowdisplaybreaks

\hypersetup{colorlinks,linkcolor=blue, urlcolor=blue,
	citecolor=blue}

\newtheorem{thm}{Theorem} [section]
\newtheorem{lem}[thm]{Lemma}
\newtheorem{cor}[thm]{Corollary}
\newtheorem{prop}[thm]{Proposition}

\theoremstyle{definition}

\newtheorem{example}[thm]{Example}

\newtheorem{rem}[thm]{Remark}

\numberwithin{equation}{section}

\newcommand{\A}{\mathbb{A}}

\newcommand{\Ad}{\mathrm{Ad}}

\newcommand{\sB}{{\mathscr B}}
\newcommand{\C}{{\mathbb C}}
\newcommand{\D}{\mathcal{D}}

\newcommand{\End}{{\mathrm{End}}}

\newcommand{\ff}{\mathbf{f}}
\newcommand{\sF}{{\mathscr F}}
\newcommand{\g}{{\mathbf g}}

\newcommand{\HH}{\mathbf H}
\newcommand{\Hom}{{\mathrm{Hom}}}
\newcommand{\id}{\mathbbm{1}}

\newcommand{\Ob}{\mathscr O}

\newcommand{\Q}{\mathbb {Q}}
\newcommand{\rel}{\mathrm{rel}}

\newcommand{\Sc}{\mathbf S}
\newcommand{\T}{\mathbb T}
\newcommand{\bT}{\mathbf{T}}
\newcommand{\V}{\mathbb {V}}
\newcommand{\W}{\mathbb {W}}

\newcommand{\Z}{{\mathbb Z}}

\newcommand{\nc}{\newcommand}
\nc{\browntext}[1]{\textcolor{brown}{#1}}
\nc{\greentext}[1]{\textcolor{green}{#1}}
\nc{\redtext}[1]{\textcolor{red}{#1}}
\nc{\bluetext}[1]{\textcolor{blue}{#1}}
\nc{\brown}[1]{\browntext{ #1}}
\nc{\green}[1]{\greentext{ #1}}
\nc{\red}[1]{\redtext{ #1}}
\nc{\blue}[1]{\bluetext{ #1}}

\subjclass[2020]{20G43, 22E57, 20G42}
\keywords{Schur algebra, Borel-Moore homology, equivariant K-theory}

\title[Geometric construction of Schur algebras]
{Geometric construction of Schur algebras}

\author[Li Luo]{Li Luo}
\author[Zheming Xu]{Zheming Xu}
\author[Yang Yang]{Yang Yang}
\address{School of Mathematical Sciences, Key Laboratory of MEA (Ministry of Education) \& Shanghai Key Laboratory of PMMP, East China Normal University, Shanghai 200241, China}
\email{lluo@math.ecnu.edu.cn (Luo)}
\email{zxu0@stu.ecnu.edu.cn (Xu)} 
\email{52275500011@stu.ecnu.edu.cn (Yang)}

\begin{document}

\begin{abstract} 
We provide the geometric construction of a series of generalized Schur algebras of any type via Borel-Moore homologies and equivariant K-groups of generalized Steinberg varieties. As applications, we obtain a Schur algebra analogue of the local geometric Langlands correspondence of any type, provide an equivariant K-theoretic realization of quasi-split $\imath$quantum groups of affine type AIII, and establish a geometric Howe duality for affine ($\imath$-)quantum groups.
\end{abstract}

\maketitle

\tableofcontents
%===========================================
\section{Introduction}
%===========================================
\noindent{1.1.}\quad
%-------------------------------------------
One of the origins of geometric representation theory is the geometric realization of the Hecke algebra associated to an algebraic group $G$ as the convolution algebra of the $G$-orbits on the double complete flag varieties owing to Iwahori (cf. \cite{IM65}). Such a formulation was generalized to the case of partial ($n$-step) flag variety for $G=GL_d$ by Beilinson-Lusztig-MacPherson \cite{BLM90}, in which the corresponding convolution algebra is just the $q$-Schur algebra $\Sc_{n,d}$ (of type A) introduced by Dipper-James \cite{DJ89}. Furthermore, they realized the quantized enveloping algebra $\mathbf{U}_q(\mathfrak{gl}_n)$ and its canonical basis (for the modified version) based on a stabilization property for $\Sc_{n,d}$. The affine quantum $\mathfrak{gl}_n$ counterpart of the above realization is derived in \cite{Lu99, DF15}. 

There is another geometric approach to drawing the Weyl group $\W$ and the affine Hecke algebra $\widetilde{\mathbf{H}}$ via the Steinberg variety $Z$ attached to $G$. The Lagrangian construction obtained by Kashiwara-Tanisaki \cite{KT84} (and independently by Ginzburg \cite{G86}) says that the top Borel-Moore homology $H(Z)$ with convolution product is just the group algebra of $\W$. Then by Springer's representation theory on the cohomology of Springer fibers \cite{Sp76}, one can attain a classification of irreducible modules of $H(Z)$.
For the case of affine Hecke algebras, it is initiated by Kazhdan-Lusztig \cite{Lu85,KL85} and improved by Ginzburg \cite{G87} that the equivariant K-group $K^{G\times \mathbb{C}\backslash\{0\}}(Z)$ is isomorphic to the affine Hecke algebra $\widetilde{\mathbf{H}}({}^L{G})$ assoicated with ${}^L{G}$, where ${}^L{G}$ is the Langlands dual of $G$. This is a significant evidence of the mysterious link between Iwahori's realization and the equivariant K-theoretic realization, which is the so-called local geometric Langlands correspondence for affine Hecke algebras (See also \cite{Be16} for a categorification version of this correspondence).
The resulting classification of irreducible modules of $\widetilde{\mathbf{H}}({}^L{G})$ via Springer’s representation theory, is one of the principal successes of the celebrated Langlands program.

Inspired by the aforementioned Beilinson-Lusztig-MacPherson's realization of quantum $\mathfrak{gl}_n$, Ginzburg \cite{G91} constructed the enveloping algebra $U(\mathfrak{sl}_n)$ and its irreducible modules by the top Borel-Moore homology, while Ginzburg-Vasserot \cite{GV93} realized the affine quantum $\mathfrak{sl}_n$ and its irreducible modules by equivariant K-groups (see also \cite{Va98} for the case of affine quantum $\mathfrak{gl}_n$). The Steinberg varieties used in \cite{G91,GV93,Va98} are defined by employing $n$-step flags instead of complete flags.

\vspace{0.3cm}
%----------------------------
\noindent{1.2.}\quad
%-----------------------------
Apart from Nakajima's treatment of quivers \cite{N98}, the above two geometric constructions of Schur algebras and quantum groups by partial flag varieties had not been promoted to the cases other than type A for about two decades. Thanks to Bao-Wang's influential work \cite{BW18}, in which they established a quantum Schur duality between quasi-split $\imath$quantum groups of type AIII (in the sense of Satake diagrams) and Hecke algebras of type B, experts began to recognize that the geometric construction with partial flag varieties of types BCD should not lead to a Drinfeld-Jimbo quantum group but a quasi-split $\imath$quantum group $\mathbf{U}^\imath$ of type AIII. The Beilinson-Lusztig-MacPherson type realization of $\mathbf{U}^\imath$ and its canonical basis was derived in \cite{BKLW18} (see also \cite{LL21} for the case of unequal parameters), while the equivariant K-theoretic approach was studied in \cite{FMX22}. Moreover, quasi-split affine $\imath$quantum groups of type AIII, including $\mathbf{U}^{\imath\imath}$, $\mathbf{U}^{\imath\jmath}$, $\mathbf{U}^{\jmath\imath}$ and $\mathbf{U}^{\jmath\jmath}$, were introduced in \cite{FLLLW20}. The Beilinson-Lusztig-MacPherson type realization for these affinization has been provided in \cite{FLLLW20, FLLLW23}. Recently, Su and Wang \cite{SW24} gave the equivariant K-theoretic realization of $\mathbf{U}^{\imath\imath}$ by using its Drinfeld new presentation given in \cite{LWZ24}.

 The quantum Schur duality established in \cite{BW18} was employed to reformulate the Kazhdan-Lusztig theory of types BCD, for which the Fock space in the duality is regarded as the Grothendieck group of the BGG category $\mathcal{O}$. Motivated by these connections to BGG category  $\mathcal{O}$ and Kazhdan-Lusztig theory, the first author and Wang \cite{LW22} generalized the notion of $n$-step partial flag variety $\sF_\ff$ in terms of a finite subset $Q_\ff$ consisting of $\W$-orbits on the associated (co)weight lattice for arbitrary finite type. A series of generalized $q$-Schur algebras $\Sc_\ff$ of any type were introduced therein. They admit a Beilinson-Lusztig-MacPherson type realization, saying that $\Sc_\ff$ is isomorphic to the convolution algebra of $G$-invariant $\Z[q,q^{-1}]$-valued functions on $\sF_\ff\times\sF_\ff$. 
 These new $q$-Schur algebras were studied widely by the first two authors and their collaborators, such as canonical bases and Schur dualities \cite{LW22}, Howe dualities \cite{LX22}, cells \cite{CLW24} and cellularity \cite{CLX23}. The affinization $\widetilde{\Sc}_\ff$ and its Beilinson-Lusztig-MacPherson type realization are introduced in \cite{CLW24}.
 
 For a special choice of $Q_\ff$ in type A, the algebra $\Sc_\ff$ recovers Dipper-James' $q$-Schur algebra $\Sc_{n,d}$. For some special choices of $Q_\ff$ in type B, it recovers the $q$-Schur algebras considered in \cite{Gr97, BKLW18} and proper subalgebras of the Schur algebras studied in \cite{DJM98, DS00}. Particularly, if we take $Q_\ff$ a single regular $\W$-orbit for any type, then $\Sc_\ff$ is just the Hecke algebra associated with $\W$.

 \vspace{0.3cm}
%----------------------
\noindent{1.3.}\quad
%----------------------
Since these $q$-Schur algebras $\Sc_\ff$ and their affinizations $\widetilde{\Sc}_\ff$ admit a Beilinson-Lusztig-MacPherson type geometric approach shown in \cite{LW22,CLW24}, one may ask whether the classical limit (resp. the affinization) of $\Sc_\ff$ owns a Lagrangian construction (resp. an equivariant K-theoretic realization) in light of the Langlands reciprocity. This is the main issue we shall address in the present paper. 

We shall provide an affirmative answer to the above question. For any finite $\W$-invariant subset $Q_\ff$ of weight or coweight lattice, we attach a generalized Steinberg variety $Z_\ff=\widetilde{\mathcal{N}}_\ff\times_\mathcal{N}\widetilde{\mathcal{N}}_\ff$ (see \S\ref{sec:spre}-\ref{sec:steinberg}). The Lagrangian construction obtained in Theorem~\ref{geo schur} says that the direct sum of top Borel-Moore homologies of each irreducible component of $Z_\ff$ is just the classical limit of $\Sc_\ff$. The equivariant K-groups realization of affine $q$-Schur algebras is given by Theorem~\ref{geo affine Schur-Weyl duality}, which shows that the affine $q$-Schur algebra $\widetilde{\Sc}_\ff$ is isomorphic to the equivariant K-group $K^{G\times \mathbb{C}\backslash\{0\}}(Z_\ff)$.

%Though we absorb and imitate numbers of arguments in the remarkable book \cite{CG97} to achieve our goal, there are still many fresh obstacle we have to overcome. For example, under our setup we need to define convolutions for top Borel-Moore homologies and for equivariant K-groups of locally closed varieties (instead of the case of closed varieties in \cite{CG97}); see \S\ref{subsec:conv} and \S\ref{}

We must point out that our formulation for the isomorphisms obtained in Theorems~\ref{geo schur} \& \ref{geo affine Schur-Weyl duality} is implicit, for two main reasons. The first reason is that the convolution product for Schur algebras is much more cumbersome than the one for Hecke algebras; and the other reason is that the cellular fibration lemma, (which is the essential means to obtain good bases of the equavariant K-group for Hecke algebras in \cite{CG97}), no longer works for the setup of our affine $q$-Schur algebras $\widetilde{\Sc}_\ff$. Our strategy is as follows. In the spirit of Schur duality, we have three construction parts: the $q$-Schur algebra $\widetilde{\Sc}_\ff$, the Fork space $\widetilde{\mathbf{T}}_\ff$, and the Hecke algebra $\widetilde{\mathbf{H}}$. It is well known that $\widetilde{\mathbf{H}}$ corresponds to the Steinberg variety $\widetilde{\mathcal{N}}\times_\mathcal{N}\widetilde{\mathcal{N}}$, while we potentially expect that $\widetilde{\Sc}_\ff$ and $\widetilde{\mathbf{T}}_\ff$ correspond to the Steinberg varieties $\widetilde{\mathcal{N}}_\ff\times_\mathcal{N}\widetilde{\mathcal{N}}_\ff$ and $\widetilde{\mathcal{N}}_\ff\times_\mathcal{N}\widetilde{\mathcal{N}}$, respectively. On the one hand, as well as the convolution product for $\widetilde{\mathbf{H}}$, the convolution construction for $\widetilde{\mathbf{H}}$-action
on $\widetilde{\mathbf{T}}_\ff$ also only involves the pair $(\widetilde{\mathcal{N}}, \widetilde{\mathcal{N}})$, which we can hold. On the other hand, the cellular fibration lemma still works for $K^{G\times\mathbb{C}\setminus\{0\}}(\widetilde{\mathcal{N}}_\ff\times_\mathcal{N}\widetilde{\mathcal{N}})$. Thanks to these advantages, we firstly show the isomorphism $\widetilde{\mathbf{T}}_\ff\simeq K^{G\times \mathbb{C}\backslash\{0\}}(\widetilde{\mathcal{N}}_\ff\times_\mathcal{N}\widetilde{\mathcal{N}})$ as $\widetilde{\mathbf{H}}$-modules, by which we further achieve the isomorphism $\widetilde{\Sc}_\ff\simeq K^{G\times \mathbb{C}\backslash\{0\}}(Z_\ff)$ as we desired.

\vspace{0.3cm}
%----------------------------
\noindent{1.4.}\quad
%----------------------------
There are several applications for the main results. 

The first one is a Schur algebra analogue of the local geometric Langlands correspondence, which says:  (for notations we refer to \S\ref{subsec:app1})
$$\C\left[\bigsqcup_{\gamma,\nu\in\Lambda_{\ff}}\mathcal{I}_\gamma\backslash G(\mathbb{F}_p((\epsilon)))/\mathcal{I}_\nu\right]\simeq K^{{}^LG\times \mathbb{C}\backslash\{0\}}({}^LZ_\ff)|_{p=q^{-2}}.$$
If we take $Q_\ff$ to be a single regular $\W$-orbit, then the above reciprocity degenerates to the original local geometric Langlands correspondence about affine Hecke algebras:
$$\C\left[\mathcal{I}\backslash G(\mathbb{F}_p((\epsilon)))/\mathcal{I}\right]\simeq K^{{}^LG\times \mathbb{C}\backslash\{0\}}({}^LZ)|_{p=q^{-2}},$$
where $\mathcal{I}$ is the Iwahori subgroup of $G(\mathbb{F}_p((\epsilon)))$. 
%A categorification of this original reciprocity for affine Hecke algebras was given by Bezrukavnikov \cite{Be16}. We shall try to give a similar categorification of the reciprocity for $\widetilde{\Sc}_\ff$ in a subsequent work.

The second application is an implicit equivariant K-theoretic realization of the affine $\imath$quantum groups $\mathbf{U}^{\imath\imath}$ and $\mathbf{U}^{\imath\jmath}$ in the inverse limit of affine $q$-Schur algebras by stabilization property. An explicit realization of $\mathbf{U}^{\imath\imath}$ was recently given by Su-Wang in \cite{SW24}. But there still has been no explicit formulation for $\mathbf{U}^{\imath\jmath}$, $\mathbf{U}^{\jmath\imath}$ and $\mathbf{U}^{\jmath\jmath}$ yet, because of the lack of a Drinfeld new presentation for them.    

The third application is a geometric duality between two affine $q$-Schur algebras via equivariant K-theory, which can be regarded as a mirror image of the geometric construction of Howe duality provided in \cite{LX22} in the sense of Langlands reciprocity. At a specialization of $Q_\ff$ for type A, we get a geometric Howe duality between affine quantum groups $\mathbf{U}(\widetilde{\mathfrak{gl}}_m)$ and $\mathbf{U}(\widetilde{\mathfrak{gl}}_n)$, whose Lagrangian construction counterpart for $\mathfrak{gl}_m$ and $\mathfrak{gl}_n$ was given by Wang \cite{W01}. At some specializations of $Q_\ff$ for type B, we obtain a geometric Howe duality between quasi-split affine $\imath$quantum groups of type AIII.

\vspace{0.3cm}
%------------------------
\noindent{1.5.}\quad
%------------------------
Let us discuss below some follow-up work we are considering related to this paper. 

The equivariant K-theoretic construction of affine Hecke algebras (with one parameter $q$) was generalized by Kato \cite{Ka09} to the case of affine Hecke algebras of type $C_n^{(1)}$ with three parameters in terms of the geomotry of the so-called exotic nilpotent cone of the complex symplectic group $\mathrm{Sp}_{2n}(\mathbb{C})$. In the process of writing this present paper, we also dealt with the affine quantum Schur algebras of type $C$ with three parameters in the sense of Kato's exotic setup, which will appear in our subsequent paper \cite{LXY}. 

Moreover, we are trying to obtain a Drinfeld new presentation of the quasi-split affine $\imath$quantum groups of type AIII with three parameters introduced in \cite{FLLLWW20}, so that we can further give their equivariant K-theoretic realization based on \cite{LXY} similar to Su-Wang's work \cite{SW24}.    

Recently, Bezrukavnikov, Karpov and Krylov \cite{BKK23} proved an equivariant K-theoretic description for Lusztig's asymptotic affine Hecke algebras, which was conjectured by Qiu and Xi in \cite{QX23}. Cui, Wang and the first author introduced asymptotic affine $q$-Schur algebras of any type in \cite{CLW24}, which have been further studied systematically in \cite{CLX23}. We are also working on providing an equivariant K-theoretic description for these asymptotic affine $q$-Schur algebras.

\vspace{0.3cm}
%------------------------
\noindent{1.6.}\quad
%------------------------
The paper is organized as follows. In Section 2, We provide some basic notions and definitions such as Schur algebras, flag varieties, Steinberg varieties and so on, some of which are generalizations of the usual ones. Section 3 and 4 are devoted to the Lagrangian construction of our generalized Schur algebras and their finite irreducible representations, respectively. The equivariant K-theoretic realization of affine $q$-Schur algebras are obtained in Section 5, while their irreducible representations are studied in Section 6. In Section 7, we provide several applications for our geometric construction, including an Schur algebra analogue of local geometric Langlands correspondence, geometric realization of affine $\imath$quantum groups, and geometric Howe dualities.  

\subsubsection*{Acknowledgement}
The work is partially supported by the NSF of China (No. 12371028), the Science and Technology Commission of Shanghai Municipality (No. 22DZ2229014), and Fundamental Research Funds for the Central Universities.

%===========================================
\section{Preliminary}
%===========================================
\subsection{Weyl group orbits on weight lattice}
%----------------------------------------------- 
 Fix a complex connected reductive algebraic group $G$ with complex dimension $n$ and rank $d$. Let $\mathfrak{g}=\mathrm{Lie}(G)$ be the associated Lie algebra, regarded as a $G$-module by means of the adjoint action.  Let $T$ be a maximal torus of $G$, $\W=N(T)/T$ the Weyl group of $G$ associated with $T$. Fix a simple system $\Delta=\{\alpha_1,\ldots,\alpha_d\}$ and let $\Pi$ (resp. $\Pi^+$, $\Pi^-$) be the associated (resp. positive, negative) root system. By abuse of notations, we also regard that the Weyl group $\W$ is generated by the simple reflections $s_1,\ldots,s_d$, where $s_i$ is the simple reflection associated with $\alpha_i$, $(i=1,\ldots,d)$. Let $\ell(w)$ denote the length of $w\in \W$. Let $Q$ (resp. $Q^\vee$) be the weight (resp. coweight) lattice, on which there is a natural $\W$-action. The associated (extended) affine Weyl group is defined as $\widetilde{\W}:=\W\ltimes Q$. 

Let us take a $\W$-invariant finite subset
	\begin{equation*}
		Q_\ff\subset Q\quad \mbox{(or $Q^{\vee}$ if necessary)}. 
	\end{equation*} 
 The choices of $Q_\ff$ are flexible and far from being unique.
	Denote
	\begin{align*}
		\Lambda =\{\mbox{$\W$-orbits in $Q$}\},\quad
		\Lambda_\ff =\{\mbox{$\W$-orbits in $Q_\ff$}\}.
	\end{align*}
 It is known that for any $\W$-orbit $\gamma \in \Lambda$, there exists a unique anti-dominant weight $\mathbf{i}_\gamma\in\gamma$. Let $$J_\gamma:=\{s_k~|~1\leq k\leq d, s_k\mathbf{i}_\gamma =\mathbf{i}_\gamma\},\quad \Delta_\gamma:=\{\alpha_k~|~1\leq k\leq d, s_k\mathbf{i}_\gamma =\mathbf{i}_\gamma\}\subset\Delta,$$ and $\Pi_\gamma=\Pi_\gamma^+ \sqcup\Pi_\gamma^-$ the associated root system.
 
%--------------------------------
\subsection{Parabolic subgroups}\label{subsec:parabolic}
%--------------------------------
For any $\W$-orbit $\gamma \in \Lambda$, let $\W_\gamma$ be the parabolic subgroup of $\mathbb{W}$ generated by $J_\gamma$. 
%It is clear that $\mathbb{W}_\gamma$ is the isotropy group of $\mi_\gamma$. 
An orbit $\gamma\in\Lambda$ is called {\em regular} if $\W_\gamma=\{\id\}$ is trivial. 

Let $B\subset G$ (resp. $\mathfrak{b}\subset\mathfrak{g}$) be the Borel subgroup (resp. subalgebra) associated with $\Pi$. Denote by $U$ (resp. $\mathfrak{n}$) the unipotent radical  (resp. radical) of $B$ (resp. $\mathfrak{b}$), and by $U^-$ its negative counterpart. 
Let $P_\gamma=\bigsqcup_{w\in\W_\gamma}BwB$ be the parabolic subgroup of $G$ associated with $\gamma\in\Lambda$, and $U_\gamma$ be the unipotent radical of $P_\gamma$ . The Lie algebra of $P_\gamma$ (resp. $U_\gamma$) is denoted by $\mathfrak{p}_\gamma$ (resp. $\mathfrak{n}_\gamma$). 
  We remark that it may be $P_{\gamma}=P_{\nu}$ for $\gamma\neq\nu$.
 For example, $P_\gamma=B, U_\gamma=U$ and $\mathfrak{n}_\gamma=\mathfrak{n}$ for any regular orbit $\gamma$. 
Moreover, for those $\gamma$ such that $\mathbb{W}_\gamma=\{\id, s_i\}$, we shall sometimes write $P_i=P_\gamma$ and $U_i=U_\gamma$ by some abuse of notations. 

Denote
	$$\D_\gamma=\{w \in \W  ~|~  \ell(vw)= \ell(v)+ \ell(w), \forall v \in \W_\gamma\}.$$
	Then $\D_\gamma$ (resp. $\D_\gamma^{-1}$) is the set of distinguished minimal length
	right (resp. left) coset representatives of $\W_\gamma$ in $\W$.

\begin{lem}\label{dis}
The following three conditions are equivalent:
\begin{equation*}
\mbox{(1) $w \in \D_\gamma$;\qquad (2) $w^{-1}\Delta_\gamma \in \Pi^+$; \qquad (3) $w^{-1}\Pi_\gamma^\pm \in \Pi^\pm$.}
\end{equation*}
\end{lem}
\begin{proof}
    It is obvious that $(2)\Leftrightarrow(3)$. For $(1)\Leftrightarrow(2)$, it follows from a basic result about Weyl groups that $\ell(ws_i)=\ell(w)+1$ if and only if $w(\alpha_i)\in\Pi^+$.
\end{proof}

\begin{lem}\label{lem:para}
    For $\gamma\in\Lambda_\ff$ and $w\in\D_\gamma$, it holds that $P_\gamma \cap wBw^{-1}B=B$.
\end{lem}
\begin{proof}
    Let $U_\alpha$ be the root subgroup of $G$ corresponding to a root $\alpha\in\Pi$, which satisfies $w U_\alpha w^{-1}=U_{w\alpha }$ (refer to \cite[Part II \S 1.2\&1.4]{J03}). 
    Observe that 
        \begin{align*}
            wBw^{-1}=(\prod_{\alpha \in w\Pi^+}U_\alpha) T \quad\mbox{and}\quad
            P_\gamma \cap U^-=\prod_{\alpha \in \Pi_\gamma^-}U_\alpha,
        \end{align*} 
        hence we have
        \begin{align*}
          wBw^{-1}B =(\prod_{\alpha \in (w\Pi^+)\cap\Pi^-}U_\alpha)B \subset U^-B\quad\mbox{and}\quad
          P_\gamma \cap (U^-B)=(\prod_{\alpha \in \Pi_\gamma^-}U_\alpha)B.
        \end{align*}
        Thus $P_\gamma \cap wBw^{-1}B=(\prod_{\alpha \in (w\Pi^+) \cap \Pi_\gamma^-}U_\alpha)B$.  But $(w\Pi^+) \cap \Pi_\gamma^-=w(\Pi^+ \cap (w^{-1}\Pi_\gamma^-))=\emptyset$ for $w \in \D_\gamma$ by Lemma \ref{dis}. So $P_\gamma \cap wBw^{-1}B=B$.
       \end{proof}

	Let
	$\D_{\gamma\nu}=\D_\gamma \cap \D_\nu^{-1}$
	be the set of minimal length double coset representatives of $\W_\gamma\backslash \W/\W_\nu$. For any $w\in \D_{\gamma\nu}$, we denote 
 \begin{equation}\label{Pwrv}
 \W_{\gamma\nu}^w:=\W_\gamma \cap w \W_\nu w^{-1}\quad \mbox{and}\quad P_{\gamma\nu}^w:=P_\gamma \cap w P_\nu w^{-1}.
 \end{equation} 
In fact, $\W_{\gamma\nu}^w$ is also a parabolic subgroup of $\mathbb{W}$ because it is the Weyl group of the parabolic subgroup $(P_{\gamma}\cap wP_{\nu}w^{-1})U_{\gamma}$ of $G$ (see \cite[Proposition 4.7]{BT65}) and $B\subset(P_{\gamma}\cap wP_{\nu}w^{-1})U_{\gamma}$.
 We can not expect that $P_{\gamma\nu}^w$ and $\bigsqcup_{\sigma\in\W_{\gamma\nu}^w}B\sigma B$ coincide. In fact, $P_{\gamma\nu}^w\subsetneqq\bigsqcup_{\sigma\in\W_{\gamma\nu}^w}B\sigma B$ if $w \neq \id$. In general, $P_{\gamma\nu}^w$ is not a parabolic subgroup of $G$.%But their Lie algebras share the same Levi subalgebra, which is shown below.
\begin{lem}\label{levi}
The Weyl group of the Levi part of $\mathrm{Lie}(P_{\gamma\nu}^w)$ is just $\W_{\gamma\nu}^w$.
\end{lem}
\begin{proof}It follows from that 
the Levi part of $\mathrm{Lie}(P_{\gamma\nu}^w)$ is the intersection of the Levi parts of $\mathrm{Lie}(P_\gamma)$ and $\mathrm{Lie}(wP_\nu w^{-1})$, whose Weyl group is  $\W_\gamma \cap w \W_\nu w^{-1}=\W_{\gamma\nu}^w$.
\end{proof}

 %%%%%%%
	
%--------------------------------------------------------
\subsection{Schur algebras}\label{sec:schur}
%--------------------------------------------------------
Denote by $\theta_\gamma$ the unique longest element in $\W_\gamma$, and let $$x_\gamma:=\sum_{w \in \W_\gamma} (-1)^{\ell(w)-\ell(\theta_\gamma)} w\in\Q\W,$$ where $\Q\W$ is the group algebra of $\W$ over $\Q$. For any $s_i\in J_\gamma\subset\W_\gamma$, it is obvious that 
\begin{equation}\label{eq:xsi}
    x_\gamma s_i=-x_\gamma.
\end{equation}

Set $$\T_\ff:=\bigoplus_{\gamma \in \Lambda_\ff} \T_\gamma\quad \mbox{with}\quad\T_\gamma:=x_\gamma \Q \W,$$
which is a right $\Q \W$-module. As a linear space, $\T_\gamma$ owns a $\Q$-basis $\{x_\gamma w~|~w\in\D_{\gamma}\}$.
The Schur algebra $\mathbb{S}_{\ff}$ is defined as
$$\mathbb{S}_{\ff}:=\mathrm{End}_{\Q\W}(\T_\ff).$$ For $\gamma,\nu\in\Lambda_\ff$ and $w\in\D_{\gamma\nu}$, let $\phi_{\gamma\nu}^w\in\mathbb{S}_\ff$ be the element such that $$\phi_{\gamma\nu}^w(x_{\nu'})=\delta_{\nu\nu'}(-1)^{\ell(w_{\gamma\nu}^+)}\sum_{w'\in\W_\gamma w\W_\nu}(-1)^{\ell(w')}w',$$
where $w_{\gamma\nu}^+$ denotes the unique longest element in $\W_\gamma w\W_\nu$. It is known that $\{\phi_{\gamma\nu}^w~|~\gamma,\nu\in\Lambda_\ff, w\in\D_{\gamma\nu}\}$ forms a basis of $\mathbb{S}_{\ff}$.

%----------------------------
\subsection{Affine $q$-Schur algebras}
%----------------------------
Let $q$ be an indeterminate. The Hecke algebra $\mathbf{H}$ associated to $\W$ is a $\Z[q,q^{-1}]$-algebra with an $\Z[q,q^{-1}]$-basis $\{H_w~|~w\in \W\}$ subject to the following relations: 
\begin{align*}
(H_{s_i}-q^{-1})(H_{s_i}+q)=0, &\qquad \mbox{for $i=1,2,\ldots,n$};\\
H_wH_{w'}=H_{ww'}, &\qquad\mbox{if $\ell(ww')=\ell(w)+\ell(w')$}.
\end{align*}
We shall write $H_i=H_{s_i}$ for short.

The (extended) affine Hecke algebra $\widetilde{\mathbf{H}}$ is a free $\Z[q,q^{-1}]$-module with the basis $$\{H_w e^\lambda~|~w\in \W, \lambda\in Q\}$$ such that 
\begin{itemize}
    \item The set $\{H_w ~|~ w\in \W \}$ spans a subalgebras of $\widetilde{\mathbf{H}}$ isomorphic to $\mathbf{H}$;
    \item $e^\lambda e^{\lambda'}=e^{\lambda+\lambda'}$;
    \item $e^\lambda H_i-H_i e^{s_i(\lambda)} =(q^{-1}-q)\frac{e^\lambda-e^{s_i(\lambda)}}{1-e^{-\alpha_i}}$.
\end{itemize}

Let
\begin{equation}\label{eq:sym}
\mathbf{x}_\gamma:=\sum_{w \in \W_\gamma} (-q)^{\ell(w)-\ell(\theta_\gamma)} H_{w}\in\HH\quad \mbox{and}\quad \widetilde{\mathbf{T}}_\ff:=\bigoplus_{\gamma \in \Lambda_\ff}\mathbf{x}_\gamma \widetilde{\mathbf{H}}.
\end{equation}
The affine $q$-Schur algebra (or called affine quantum Schur algebra) $\widetilde{\Sc}_\ff$ is an $\A$-algebra defined as $$\widetilde{\Sc}_\ff:=\mathrm{End}_{\widetilde{\mathbf{H}}}(\widetilde{\mathbf{T}}_\ff).$$

\begin{rem}\label{rem:schur}
\begin{itemize}
\item[(1)] We would like to emphasize that we use weight lattice (instead of coweight lattice) to define the affine Hecke algebra $\widetilde{\mathbf{H}}$ here. So our setup is always the Langlands dual of the one by using coweight lattice such as in \cite{CLW24}.
\item[(2)] Though we use the $\W$-orbits $\gamma \in \Lambda_\ff$ here to define the Fock space $\widetilde{\mathbf{T}}_\ff$ while \cite{CLW24} used the orbits of the affine Weyl group $\widetilde{W}$ instead. Comparing \eqref{eq:sym} with \cite[(2.8)]{CLW24} by a replacement $v=-q^{-1}$ therein, we see that the affine $q$-Schur algebra $\widetilde{\Sc}_\ff$ defined here coincides with the one defined \textit{loc. cit.}, up to a Langlands dual mentioned in (1). Moreover, we remark that as the same as in this present paper, the set $J_\gamma$ in \textit{loc. cit.} always excludes the extra simple reflection $s_0$ of the affine Weyl group $\widetilde{W}$ though it was not explicitly stated therein.
\end{itemize}
\end{rem}

%It is showed in \cite[Proposition~3.5]{LW22} that the affine $q$-Schur algebra $\widetilde{\Sc}_\ff$ owns an $\A$-basis $\{\widetilde{\phi}_{\gamma,\nu}^w~|~\gamma,\nu\in\Lambda_\ff, \mbox{$w$ is a representative of $\W_\gamma\backslash\widetilde{\W}/\W_\nu$}\}$\red{(need to be modified later)}, where $\widetilde{\phi}_{\gamma,\nu}^w\in \widetilde{\Sc}_\ff=\mathrm{End}_{\widetilde{\mathbf{H}}}(\widetilde{\mathbf{T}}_\ff)$ is the $\widetilde{\mathbf{H}}$-morphism determined by 
%\begin{equation}\label{def:philnw}
%\widetilde{\phi}_{\gamma,\nu}^w: \mathbf{x}_{\nu'}\mapsto %\delta_{\nu,\nu'}(-%q)^{\ell(w_{\gamma\nu}^+)}\sum_{w'\in\W_\gamma w\W_\nu}(-%q)^{\ell(w')}H_{w'}. \red{(Check!)}
%\end{equation}

%------------------------------------------------------
\subsection{Flag varieties}
%-------------------------------------------------------
 Let $\sB=G/B$ be the complete flag variety, which admits a natural $G$-action. Let $G$ act on $\sB\times \sB$ diagonally. It is well-known that there is a one-to-one correspondence between the Weyl group $\mathbb{W}$ and the $G$-orbits in $\sB\times \sB$, Precisely, it sends $w\in \mathbb{W}$ to the $G$-orbit containing $(B,wB)$, which will be denoted by $\Ob_w\in G \backslash(\sB\times\sB)$.
The closure $\overline{\Ob}_w=\bigcup_{y\leq w}\Ob_y$, where ``$\leq$'' means the Bruhat order on $\W$. For example, 
\begin{equation}\label{eq:obsi}
    \overline{\Ob}_{s_i}={\Ob}_{\id}\cup{\Ob}_{s_i}=\{(gB,ghB) \in \sB \times \sB ~|~ g \in G, h \in P_i\}\simeq G \times^B (P_i/B),
\end{equation} where $P_i$ is defined in \S\ref{subsec:parabolic}, and $G \times^B (P_i/B)$ means the variety of $B$-orbits on $G \times (P_i/B)$. Hence $\overline{\Ob}_{s_i}$ is a smooth variety. 

Denote $\sF_\gamma=G/P_\gamma$.	We shall consider the following partial flag variety introduced in \cite{LW22}:
	\begin{equation*}
		\sF_\ff= \bigsqcup_{\gamma \in \Lambda_\ff} \sF_\gamma, \quad\mbox{the disjoint union of 
			$\sF_\gamma\ (\gamma\in\Lambda_\ff)$.}
	\end{equation*}
 
\begin{rem}
    \begin{itemize}
    \item[(1)] The definition of $\sF_\ff$ depends on the choice of $\Lambda_\ff$, which is very flexible. For example, if we take $\Lambda_\ff$ to be a single regular $W$-orbit then $\sF_\ff=\sB$.
        \item[(2)]  In general, $\dim_{\mathbb{R}}\sF_\gamma\neq\dim_{\mathbb{R}}\sF_{\nu}$ if $\gamma\neq\nu$, so the dimension of the irreducible components of $\sF_\ff$ may vary considerably from component to component.
        \item[(3)] We always regard $\sF_\gamma\neq\sF_\nu$ if $\gamma\neq\nu$ though it may be  $P_\gamma=P_\nu$.
    \end{itemize}
\end{rem}
 
There is a natural $G$-action on $\sF_\gamma$ and hence on $\sF_\ff$. Let $G$ act diagonally on $\sF_\gamma\times \sF_\nu, (\gamma,\nu \in \Lambda)$,
	and on $\sF_\ff \times \sF_\ff$.
		Denote
	\begin{align*}
  \Xi_{\gamma\nu}&:=\{(\gamma,w,\nu)~|~w\in\D_{\gamma\nu}\} \quad (\forall \gamma \in \Lambda_\ff, \nu \in \Lambda_\g) \quad\mbox{and}\\
		\Xi_{\ff}&:=\bigsqcup_{\gamma,\nu \in \Lambda_\ff} \Xi_{\gamma\nu}=\bigsqcup_{\gamma ,\nu\in \Lambda_\ff} \{\gamma\} \times \D_{\gamma\nu} \times \{\nu\}.
	\end{align*}
 The sets $\Xi_{\gamma\nu}$ and $ \Xi_{\ff}$ are both finite.
 
	There is a bijection between $\D_{\gamma\nu}$ and the set of $G$-orbits $G \backslash
	(\sF_\gamma \times \sF_\nu)$, which sends $w \in\D_{\gamma\nu}$ to the $G$-orbit containing $(P_\gamma,wP_\nu)$. Hence, the $G$-orbits in $\sF_\ff \times \sF_\ff$ can be labeled by
	$\Xi_{\ff}$. The orbit related to $\xi=(\gamma,w,\nu)\in\Xi_{\ff}$ will be denoted by $\Ob_\xi$ or $\Ob_{\gamma,w,\nu}$. 
    
    Recall the subgroup $P_{\gamma\nu}^w=P_{\gamma}\cap wP_{\nu}w^{-1}$ in \eqref{Pwrv}. We have a natural isomorphism \begin{equation}\label{iso1}
    \Ob_{\gamma,w,\nu}\simeq G/P^w_{\gamma\nu},\quad (gP_\gamma, gwP_\nu)\mapsto gP_{\gamma\nu}^w.
\end{equation}

 \begin{lem}\label{partial Bruhat order}
For any $w \in \D_{\gamma\nu}$, the closure $\overline{\Ob}_{\gamma,w,\nu}=\bigcup_{y \in \D_{\gamma\nu}, y\leq w} \Ob_{\gamma,y,\nu}$.     
 \end{lem}
\begin{proof}
    Let $\pi: \sB \times \sB \rightarrow \sF_\gamma \times \sF_\nu$ be the natural projective map, which satisfies $\pi(\overline{Y})=\overline{\pi(Y)}$ for any subset $Y \subset \sB \times \sB$. Then $\overline{\Ob}_{\gamma,w,\nu}=\overline{\pi(\Ob_w)}=\pi(\overline{\Ob_w})=\pi(\bigcup_{y\leq w}\Ob_y)=\bigcup_{y\leq w}\pi(\Ob_y)=\bigcup_{y \in \D_{\gamma\nu}, y\leq w}\Ob_{\gamma,y,\nu}$.
\end{proof}
%-------------------
\subsection{Nilpotent orbit and isotropic subgroup}
%----------------------
Let $\mathcal{N}$ denote the set of all nilpotent elements of $\mathfrak{g}$ (called the nilpotent cone of $\mathfrak{g}$), which is a closed $G$-stable subvariety of $\mathfrak{g}$. Under the adjoint action of $G$, the nilpotent cone $\mathcal{N}$ can be partitioned into finite $G$-orbits  (cf. \cite[Proposition~3.2.9]{CG97}): 
\begin{equation*}
\mathcal{N}=\bigsqcup_{\alpha\in\mathbb{J}} \mathcal{N}_\alpha,\quad \mbox{where $\mathbb{J}$ is a finite set labelling the orbits}.
\end{equation*} 

We shall always fix a base point $x_\alpha\in \mathcal{N}_\alpha$ for each $\alpha\in\mathbb{J}$.
 Let $G_{x_\alpha}$ be the isotropic subgroup of $x_\alpha$ in $G$, and $G^{\circ}_{x_\alpha}$ its connected component containing the identity. Let $C(\alpha):=G_{x_\alpha}/G^{\circ}_{x_\alpha}$ be the component group of $G_{x_\alpha}$ which is independent of the choice of $x_\alpha$. 
 
 Take a $G_{x_\alpha}$-variety $\mathcal{V}$. Denote by $\mathfrak{I}(\mathcal{V})$ the set of irreducible components of $\mathcal{V}$.
 The $G_{x_\alpha}$-action on $\mathcal{V}$ induces a $C(\alpha)$-action on $\mathfrak{I}(\mathcal{V})$ by permutation. 
For  $X\in\mathfrak{I}(\mathcal{V})$, denote by $\mathbb{O}_{X}$ the $G_{x_\alpha}$-orbit in $\mathfrak{I}(\mathcal{V})$ containing $X$. Write $\mathcal{V}_{X}:=\bigcup_{Y\in\mathbb{O}_{X}}Y$.
 Then 
 \begin{equation*}
 \mu:G\times^{G_{x_\alpha}}\mathcal{V}\to \mathcal{N}_\alpha, \quad (g,v)\mapsto \mathrm{Ad}g(x_\alpha)
 \end{equation*}
is a $G$-equivariant fiber bundle, and $G\times^{G_{x_\alpha}}\mathcal{V}_X$ is a closed subbundle. 

The following lemma may be known to experts. For safety, we provide a proof below.
 \begin{lem}\label{irreducible components of Z0}
     There is a natural bijection:
\begin{align*}
    \left\{\text{Irreducible components of }G\times^{G_{x_\alpha}}\mathcal{V}\right\}\longleftrightarrow \left\{C(\alpha)\text{-orbits on }\mathfrak{I}(\mathcal{V})\right\}.
\end{align*}
More precisely, each irreducible component of $G\times^{G_{x_\alpha}}\mathcal{V}$ is in the form of $G\times^{G_{x_\alpha}}\mathcal{V}_X$ for $X\in\mathfrak{I}(\mathcal{V})$.
      \end{lem}
\begin{proof}
 Note that $\mathcal{V}^{\mathrm{reg}}_{X}$ is $G_{x_\alpha}$-stable, and the connected components of $\mathcal{V}_{X}^{\mathrm{reg}}$ are bijective to the irreducible components of $\mathcal{V}_{X}$. Clearly,  $G\times^{G_{x_\alpha}}\mathcal{V}^{\mathrm{reg}}_{X}$ is smooth.
     We claim that $G\times^{G_{x_\alpha}}\mathcal{V}_{X}^{\mathrm{reg}}$ is connected (and hence also irreducible), which is shown as follows. 
     
     Let $(g,v)$ and $(g',v')$ be arbitrary two points in $G\times^{G_{x_\alpha}}\mathcal{V}^{\mathrm{reg}}_{X}$. Since the $G_{x_\alpha}$-action on the set of irreducible components of $\mathcal{V}_{X}$ is transitively, the action on the connected component of $\mathcal{V}_{X}^{reg}$ is also transitively. We can find $h\in G_{x_\alpha}$ such that $h(v)$ and $v'$ lie in the same connected component. Then $(g',h(v))$ and $(g',v')$ lie in the same connected component. Note $(g'h,v)$ ($=(g',h(v))$ and $(g,v)$ are also in the same connected component, since they are in the same $G$-orbit and $G$ is connected. So $(g,v)$ and $(g',v')$ are in the same connected component, which shows that $G\times^{G_{x_\alpha}}\mathcal{V}_{X}^{\mathrm{reg}}$ is a connected variety.
     
    Now we have a Zariski open, dense irreducible subvariety $G\times^{G_{x_\alpha}}\mathcal{V}_{X}^{\mathrm{reg}}$ of $G\times^{G_{x_\alpha}}\mathcal{V}_{X}$, so $G\times^{G_{x_\alpha}}\mathcal{V}_{X}$ itself is irreducible. 
    Clearly, for $\mathbb{O}_{X}\neq \mathbb{O}_{X'}$, we have $G\times^{G_{x_\alpha}}\mathcal{V}_{X}\not\subset G\times^{G_{x_\alpha}}V_{X'}$. Noting that $G\times^{G_{x_\alpha}}\mathcal{V}=\bigcup_{X}G\times^{G_{x_\alpha}}\mathcal{V}_{X}$, where $X$ runs over a complete set of representatives of $\mathfrak{I}(\mathcal{V})/C(x_\alpha)$, the lemma follows.
\end{proof}

%---------------------------------
\subsection{Partial Springer resolution}\label{sec:spre}
%-----------------------------------
Let $$\widetilde{\mathcal{N}}:=T^*\sB\simeq G\times^B\mathfrak{n}$$ be the cotangent bundle of $\sB$, which is isomorphic to the variety consisting of $B$-orbits on $G\times\mathfrak{n}$. Here $B$ acts on $G$ by $b\cdot g=gb^{-1}$ ($g\in G, b\in B$), and on $\mathfrak{n}$ by the adjoint action. We shall use $(gB, x)$ to present the element in $\widetilde{\mathcal{N}}$ corresponding to the $B$-orbit containing $(g, \mathrm{Ad}_{g^{-1}}(x))\in G\times\mathfrak{n}$.

Similarly, let $$\widetilde{\mathcal{N}}_\gamma:=T^*\sF_\gamma\simeq G \times^{P_\gamma} \mathfrak{n}_\gamma$$ be the cotangent bundle of $\sF_\gamma$ (isomorphic to the variety consisting of $P_\gamma$-orbits on $G\times \mathfrak{n}_\gamma$), which is regarded as a closed $G$-subvariety of $\sF_\gamma \times \mathfrak{g}$. We shall use $(gP_\gamma, x)$ to present the element in $\widetilde{\mathcal{N}}_\gamma$ corresponding to the $P_\gamma$-orbit containing $(g, \mathrm{Ad}_{g^{-1}}(x))\in G\times\mathfrak{n}_\gamma$. Under this notation, the element $x$ can run over $\mathrm{Ad}G(\mathfrak{n}_\gamma)$.

Denote $$\widetilde{\mathcal{N}}_\ff=\bigsqcup_{\gamma \in \Lambda_\ff} \widetilde{\mathcal{N}}_\gamma.$$
The following projective maps
\begin{equation*}%\label{def:springer}
\pi_\gamma: \widetilde{\mathcal{N}}_\gamma \to \mathcal{N}, \quad (gP_\gamma, x)\mapsto x, \quad (\forall \gamma\in\Lambda_\ff)
\end{equation*}
are called partial Springer resolutions. In other words, the resolution $\pi_\gamma$ maps the $P_\gamma$-orbit containing $(g,x)$ to $Ad_g(x)$. 
Denote $$\mathcal{N}_{\gamma}:=\mathrm{Im} \pi_\gamma=\mathrm{Ad}G(\mathfrak{n}_{\gamma}),$$ which is a closed subvariety of $\mathcal{N}$.
It is known that $\mathcal{N}_{\gamma}$ contains a unique open and dense $G$-orbit called Richardson orbit.

Moreover, we denote $$\pi_\ff=\bigsqcup_{\gamma \in \Lambda_\ff}\pi_\gamma: \widetilde{\mathcal{N}}_\ff \to \mathcal{N}.$$
We shall always denote
\begin{equation}\label{notation:S}
\widetilde{\mathcal{S}}:=\pi_\ff^{-1}(\mathcal{S})\quad\mbox{and}\quad \widetilde{\mathcal{S}}_\gamma:=\pi_\gamma^{-1}(\mathcal{S}),\quad\mbox{for any subset $\mathcal{S}\subset\mathcal{N}$}.
\end{equation}

%-------------------
\subsection{Steinberg varieties}\label{sec:steinberg}
%----------------------
The original Steinberg variety $Z$ (associated with the complete flag variety $\sB$) is defined as 
\begin{align}\label{steinvar}
Z&:= \widetilde{\mathcal{N}} \times_{\mathcal{N}} \widetilde{\mathcal{N}}=\{((gB,x),(g'B,x))\in\widetilde{\mathcal{N}} \times \widetilde{\mathcal{N}}\} 
\\\nonumber &\simeq \{(x,gB,g'B)\in\mathcal{N}\times\sB\times\sB~|~x\in(\mathrm{Ad} g)(\mathfrak{b}) \cap (\mathrm{Ad}g')(\mathfrak{b})\}.
\end{align}
In the spirit of \eqref{steinvar}, we introduce the generalized Steinberg variety $Z_\ff$ associated with the partial flag variety $\sF_\ff$ as follows:
	\begin{align*}
 Z_\ff&:=\widetilde{\mathcal{N}}_\ff \times_{\mathcal{N}} \widetilde{\mathcal{N}}_\ff=\bigsqcup_{\gamma,\nu \in \Lambda_\ff} \widetilde{\mathcal{N}}_\gamma \times_{\mathcal{N}} \widetilde{\mathcal{N}}_\nu
 =\{((gP_\gamma,x),(gP_\nu,x))\in \widetilde{\mathcal{N}}_\gamma \times \widetilde{\mathcal{N}}_\nu~|~\gamma,\nu\in\Lambda_\ff\}
 \\&\simeq\{(x,gP_\gamma,g'P_\nu) \in \mathcal{N} \times \sF_\ff \times \sF_\ff ~|~ \gamma,\nu\in\Lambda_\ff, x \in (\mathrm{Ad} g)(\mathfrak{n}_\gamma) \cap (\mathrm{Ad}g')(\mathfrak{n}_\nu)\}.
 \end{align*}
We shall denote the subvariety $Z_{\gamma\nu}:=\widetilde{\mathcal{N}}_\gamma \times_{\mathcal{N}} \widetilde{\mathcal{N}}_\nu$ for any $\gamma,\nu\in\Lambda_\ff$, and hence $$Z_\ff=\bigsqcup_{\gamma,\nu \in \Lambda_\ff}Z_{\gamma\nu}.$$
Let $Z_{\gamma\nu}\to \mathcal{N}$ and $Z_\ff\to \mathcal{N}$ be the canonical projections. For any subset $\mathcal{S}\subset\mathcal{N}$, we denote its preimage in $Z_{\gamma\nu}$ (resp. $Z_\ff$) by $Z_{\gamma\nu}^{\mathcal{S}}$ (resp. $Z_\ff^{\mathcal{S}}$).
It is clear that 
\begin{equation}\label{notation:Zfs}
Z_{\gamma\nu}^{\mathcal{S}}=\widetilde{\mathcal{S}}_\gamma\times_\mathcal{S}\widetilde{\mathcal{S}}_\nu\quad \mbox{and}\quad
Z_\ff^{\mathcal{S}}=\widetilde{\mathcal{S}}\times_{\mathcal{S}}\widetilde{\mathcal{S}}.
\end{equation}

Below is an analogue of \cite[Proposition~3.3.4]{CG97}, which says that the Steinberg variety $Z$ is the union of the conormal bundles to all $G$-orbits in $\sB\times\sB$.
\begin{prop}\label{prop:uninion}
The subvariety $Z_{\gamma\nu}$ is the union of the conormal bundles to all $G$-orbits in $\sF_\gamma\times\sF_\nu$.
\end{prop}
\begin{proof}
The arguments are similar to the proof of \cite[Proposition~3.3.4]{CG97}, just replacing $T^*\sB=G\times^B\mathfrak{n}=G\times^B\mathfrak{b}^{\perp}$ therein by $T^*\sF_{\gamma}=G\times^{P_\gamma}\mathfrak{n_\gamma}=G\times^{P_\gamma}\mathfrak{p}_\gamma^{\perp}$.
\end{proof}

Note that the $G$-orbits on $\sF_\gamma\times\sF_\nu$ are parameterized by the finite set $\Xi_{\gamma\nu}$. For $\xi\in\Xi_{\gamma\nu}$, denote by $T_{\Ob_\xi}^*$ the conormal bundle of $\Ob_\xi$ in $\sF_\gamma \times \sF_\nu$ and by $T_\xi^*:=\overline{T_{\Ob_\xi}^*}$ its closure. 

We have a direct corollary of Proposition~\ref{prop:uninion} as follows, which is a partial flag variety counterpart of \cite[Corollary~3.3.5]{CG97}.
\begin{cor}\label{cor:stein}
\begin{itemize}
\item[(1)] For any $\gamma,\nu\in\Lambda_\ff$, the subvariety $$Z_{\gamma\nu}=\bigsqcup_{\xi \in \Xi_{\gamma\nu}} T_{\Ob_\xi}^*.$$
\item[(2)] Irreducible components of $Z_\ff$ are parameterized by elements of $\Xi_{\ff}$. Each irreducible component is $T_{\xi}^*$ for a unique $\xi=(\gamma,w,\nu)\in\Xi_{\ff}$.  In particular, there are finite irreducible components of $Z_\ff$. 
\end{itemize}
\end{cor}

%This corollary allows us to call $T_{\Ob_\xi}^*$ a stratum of $Z_{\gamma\nu}$.
 
Another corollary immediate from Proposition~\ref{prop:uninion} is the following ``middle dimensional property'' for $Z_{\gamma\nu}$.
\begin{cor} \label{cor:dim} It holds that, for any $\gamma,\nu\in\Lambda_\ff$,
$$\dim_\mathbb{R}Z_{\gamma\nu}=\dim_\mathbb{R}\sF_{\gamma}+\dim_\mathbb{R}\sF_{\nu}=\frac{1}{2}(\dim_\mathbb{R} \widetilde{\mathcal{N}}_{\gamma}+\dim_\mathbb{R} \widetilde{\mathcal{N}}_{\nu}).$$
\end{cor}

%------------------------
\subsection{Relevant nilpotent orbit}
%------------------------
\begin{lem}\label{dimensinal inequality}
    For any $G$-orbit $\mathcal{N}_\alpha \subset \pi_\ff(\widetilde{\mathcal{N}}_\ff)$ and $\gamma \in \Lambda_\ff$, we have 
    \begin{equation}\label{ineq:lem41}
    2\dim_{\mathbb{R}}(\pi_{\gamma}^{-1}(x_\alpha))+\dim_\mathbb{R} \mathcal{N}_\alpha \le \dim_\mathbb{R} \widetilde{\mathcal{N}}_\gamma.
    \end{equation}
    \end{lem}
    
    \begin{proof} It is a straightforward computation that
    \begin{align*}
        &2\dim_{\mathbb{R}}(\pi_\gamma^{-1}(x_\alpha))+\dim_{\mathbb{R}}\mathcal{N}_\alpha=\dim_{\mathbb{R}}(\pi_{\gamma}^{-1}(x_\alpha) \times \pi_{\gamma}^{-1}(x_\alpha))+\dim_{\mathbb{R}}\mathcal{N}_\alpha\\
        &=\dim_{\mathbb{R}}(G\times^{G_{x_\alpha}}(\pi^{-1}_{\gamma}(x_\alpha)\times\pi^{-1}_{\gamma}(x_\alpha)))=\dim_{\mathbb{R}}(\pi^{-1}_\gamma(\mathcal{N}_\alpha)\times_{\mathcal{N}_\alpha}\pi^{-1}_{\gamma}(\mathcal{N}_\alpha))\\& \le \dim_\mathbb{R} Z_{\gamma\gamma}\stackrel{\mathrm{Cor.}~\ref{cor:dim}}=\dim_\mathbb{R} \widetilde{\mathcal{N}}_\gamma.
    \end{align*} \end{proof}

We say  $\mathcal{N}_\alpha$ is \textit{relevant} for $\pi_{\gamma}$ if the equality in \eqref{ineq:lem41} holds. 
    We remark that not all nilpotent orbit contained in $\mathcal{N}_{\gamma}$ is relevant for $\pi_{\gamma}$.

Recall the notation $Z_{\gamma\nu}^{\mathcal{S}}$ in \eqref{notation:Zfs}. Take $\mathcal{S}=\mathcal{N}_\alpha$. We have the following lemma.
\begin{lem} It holds that $\dim_{\mathbb{R}}Z_{\gamma\nu}^{\mathcal{N}_\alpha}\leq\dim_{\mathbb{R}} Z_{\gamma\nu}$ for any $\gamma,\nu\in\Lambda_\ff$ and $\alpha\in\mathbb{J}$, where the equality holds if and only if $\mathcal{N}_\alpha$ is relevant for both $\pi_{\gamma}$ and $\pi_{\nu}$.
\end{lem}\label{lem:relZ}
   \begin{proof}
       Since $Z_{\gamma\nu}^{\mathcal{N}_\alpha}\simeq G\times^{G_{x_\alpha}}(\pi^{-1}_{\gamma}(x_\alpha)\times\pi^{-1}_{\nu}(x_\alpha))$, we have \begin{align*}
           \dim_\mathbb{R} Z_{\gamma\nu}^{\mathcal{N}_\alpha}&=\dim_{\mathbb{R}}(\pi_{\gamma}^{-1}(x_\alpha))+\dim_{\mathbb{R}}(\pi_{\nu}^{-1}(x_\alpha))+\dim_\mathbb{R} \mathcal{N}_\alpha\\&\leq\frac{1}{2}(\dim_\mathbb{R} \widetilde{\mathcal{N}}_\gamma+\dim_\mathbb{R} \widetilde{\mathcal{N}}_\nu)=\dim_{\mathbb{R}} Z_{\gamma\nu},
       \end{align*}
       where the equality holds if and only if $\mathcal{N}_\alpha$ is relevant for both $\pi_{\gamma}$ and $\pi_{\nu}$ by the previous lemma.
   \end{proof} 

%---------------------------------
\subsection{Transversal slices}\label{sec:transslice}
%------------------------------------
Let $S\subset\mathcal{N}$ be a transversal slice to $\mathcal{N}_\alpha$ at the point $x_\alpha$. By definition (cf. \cite[Definition~3.2.19]{CG97}), there is an open neighborhood $V$ of $x_\alpha$ such that $(\mathcal{N}_\alpha \cap V) \times S \xrightarrow{\sim} V$. By shrinking $V$ if necessary, we may assume that $\overline{\mathcal{N}_\alpha}\cap V=\mathcal{N}_\alpha\cap V$ and $V_\alpha:=\mathcal{N}_\alpha\cap V$ is a connected open neighborhood of $x_\alpha$ in $\mathcal{N}_\alpha$. 

Recall the notations in \eqref{notation:S} for $\widetilde{V}_\gamma=\pi_\gamma^{-1}(V)$ and $\widetilde{S}_\gamma=\pi_\gamma^{-1}(S)$.
\begin{lem}\label{dim2}
If $V \cap \mathcal{N}_{\gamma} \neq \emptyset$ for $\gamma\in\Lambda_\ff$, then $$\dim_{\mathbb{R}}\widetilde{V}_{\gamma}=\dim_{\mathbb{R}}\widetilde{\mathcal{N}}_{\gamma}\quad \mbox{and}\quad \dim_{\mathbb{R}}\widetilde{S}_{\gamma}=\dim_{\mathbb{R}}\widetilde{\mathcal{N}}_{\gamma}-\dim_{\mathbb{R}}\mathcal{N}_{\alpha}.$$
   \end{lem}
\begin{proof}
If $V \cap \mathcal{N}_{\gamma} \neq \emptyset$, then $\widetilde{V}_{\gamma}=\pi_{\gamma}^{-1}(V \cap \mathcal{N}_{\gamma})$ is an open submanifold of $\widetilde{\mathcal{N}}_{\gamma}$ and hence $\dim_{\mathbb{R}}\widetilde{V}_{\gamma}=\dim_{\mathbb{R}}\widetilde{\mathcal{N}}_{\gamma}$ by Corollary~\ref{cor:dim}.
 The rest part follows from $V_\alpha\times\widetilde{S}_{\gamma}\simeq\widetilde{V}_{\gamma}$.   \end{proof}

Recall the notation $Z_{\gamma\nu}^{\mathcal{S}}$ in \eqref{notation:Zfs}. We shall specialize $\mathcal{S}=S, V, V_\alpha$ or $\{x_\alpha\}$ below. Applying \cite[Corollary~3.2.21]{CG97}, we have 
\begin{equation*}
V_\alpha\times\bigsqcup\limits_{\gamma,\nu\in\Lambda_\ff}Z^{S}_{\gamma\nu}\simeq\bigsqcup\limits_{\gamma,\nu\in\Lambda_\ff}Z^{V}_{\gamma\nu}=Z^{V}_{\ff}\quad \mbox{and} \quad
V_\alpha\times\bigsqcup\limits_{\gamma\in\Lambda_\ff}\widetilde{S}_{\gamma}\simeq\bigsqcup\limits_{\gamma\in\Lambda_\ff}\widetilde{V}_{\gamma}=\widetilde{V},
\end{equation*}
and have that the canonical projection $Z^{V_\alpha}_{\gamma\nu}\to V_\alpha$ (resp. $\pi^{-1}_{\gamma}(V_\alpha)\to V_\alpha$) is a trivial fibration with fiber $Z^{x_\alpha}_{\gamma\nu}$ (resp. $\pi^{-1}_{\gamma}(x_\alpha)$) for any $\gamma,\nu\in\Lambda_\ff$, where $Z^{x_\alpha}_{\gamma\nu}$ is the abbreviation for $Z^{\{x_\alpha\}}_{\gamma\nu}$.
Moreover, $Z^{V_\alpha}_{\gamma\nu}$ (resp. $Z^{x_\alpha}_{\gamma\nu}$) is a closed subset of $Z^{V}_{\gamma\nu}$ (resp. $Z^{S}_{\gamma\nu}$). 
We denote 
\begin{equation}
\label{def:dln}
d_{\gamma\nu}:=\dim_\mathbb{R} Z_{\gamma\nu}\quad\mbox{and}\quad d'_{\gamma\nu}:=d_{\gamma\nu}-\dim_{\mathbb{R}}\mathcal{N}_\alpha.
\end{equation}
\begin{lem}\label{lem:vsdim}
For any $\gamma,\nu\in\Lambda_\ff$ and $\alpha\in\mathbb{J}$, 
it holds that 
$$\dim_{\mathbb{R}}Z_{\gamma\nu}^{V_\alpha}\leq d_{\gamma\nu}\quad \mbox{and}\quad  \dim_{\mathbb{R}}Z_{\gamma\nu}^{x_\alpha}\leq d'_{\gamma\nu},$$ where the two $``="$ hold if and only if $\mathcal{N}_\alpha$ is relevant for both $\pi_\gamma$ and $\pi_\nu$. 
\end{lem}
\begin{proof}
Since $V_\alpha$ is an open subset of $\mathcal{N}_\alpha$, we have $\dim_{\mathbb{R}}Z_{\gamma\nu}^{V_\alpha}=\dim_{\mathbb{R}}Z_{\gamma\nu}^{\mathcal{N}_\alpha}\leq d_{\gamma\nu}$. The second statement is derived by $$\dim_{\mathbb{R}}Z_{\gamma\nu}^{x_\alpha}=\dim_{\mathbb{R}}(\pi^{-1}_{\gamma}(x_\alpha)\times\pi^{-1}_{\nu}(x_\alpha))\leq\dim_{\mathbb{R}}Z_{\gamma\nu}^{V_\alpha}-\dim_{\mathbb{R}}\mathcal{N}_\alpha=d'_{\gamma\nu}.$$ The condition for the two equalities comes from the definition of relevant nilpotent orbits directly.   
\end{proof}

%--------------------------
\subsection{Convention}\label{sec:convention}
%---------------------------
Let $\mu_0\in\Lambda$ be a regular $\W$-orbit. By abuse of notations, we shall always omit $\mu_0$ if it occurs in the subscript, e.g. $\Ob_{\gamma,w}=\Ob_{\gamma,w,\mu_0}$, 
$\Ob_w=\Ob_{\mu_0,w,\mu_0}$, $T_{w,\nu}^*=T_{\mu_0,w,\nu}^*$, $T_w^*=T_{\mu_0,w,\mu_0}^*$, etc.  We do not have to worry about the notation 
$\Ob_{w}$ can represent both the orbit $\Ob_{\mu_0,w,\mu_0}$ on $\sF_{\mu_0}\times\sF_{\mu_0}$ and the orbit $\Ob_w$ on $\sB\times\sB$, because on the one hand we can distinguish them contextually; and on the other hand, they are essentially equivalent since $P_{\mu_0}=B$. Furthermore, we shall write  $\sF_{\mu_0}=\sB$ in this case without confusion.

%========================================================= 
\section{Lagrangian construction}
%==========================================================

%-------------------------------------
\subsection{Borel-Moore homology}\label{subsec:BM}
%-------------------------------------

For a complex variety $X$ (whose irreducible components are allowed to have different dimensions), let $H_\bullet(X)$ be the Borel-Moore homology groups of $X$ with coefficients in $\mathbb{Q}$, which may be replaced by any field of characteristic zero. We refer to \cite[\S2.6]{CG97} for details about the Borel-Moore homology.
%The essential feature of Borel-Moore homology is the existence and uniqueness of the fundamental class $[X]$ in the top Borel-Moore group $H_{\mathrm{top}}(X)$ if $X$ is irreducible.
For any complex closed subvariety $Y$ of $X$, there exists a fundamental class $[Y]$ in $H_\bullet(X)$ induced by the complex structure.
Let $H(X)$ denote the direct sum of top Borel-Moore homology groups of the connected components of $X$, which is clearly a subspace of $H_\bullet(X)$ spanned by the (linearly independent) fundamental classes of the irreducible components of $X$. Let $$\boxtimes: H_\bullet(X_1)\otimes H_\bullet(X_2)\rightarrow H_\bullet(X_1\times X_2),\quad h_1\otimes h_2\mapsto h_1\boxtimes h_2$$ be the isomorphism about the K\"unneth formula for Borel-Moore homology (cf. \cite[\S2.6.19]{CG97}).

%The following lemma about fundamental classes will be used in the proofs of Proposition \ref{5} and Proposition \ref{6}. Here $Y^\mathrm{reg}$ is the Zariski open dense subset consisting of the non-singular points of $Y$.
%\begin{lem}
   %Let $Y_{1}, Y_{2}$ are two irreducible complex subvarieties of a non-singular complex variety. Suppose $Y_{1}^{\mathrm{reg}}$ and $Y^{\mathrm{reg}}_{2}$ intersect transversely, and $Y_{1}\cap Y_{2}=Y^{\mathrm{reg}}_{1}\cap Y_{2}^{\mathrm{reg}}$. Then $[Y_{1}]\cap[Y_{2}]=[Y_{1}\cap Y_{2}]$. 
%\end{lem}
%\begin{proof}
    %Let $\{Z_{1},\cdots Z_{r}\}$ be the set of irreducible components of $Y_{1}\cap Y_{2}$. Then we have $[Y_{1}]\cap[Y_{2}]=\sum_{i=1}^{r}m_{i}[Z_{i}]$ for some unique integers $m_{i}$. Since $Y_{1}^{\mathrm{reg}}$ and $Y^{\mathrm{reg}}_{2}$ intersect transversely and $Y_{1}\cap Y_{2}=Y^{\mathrm{reg}}_{1}\cap Y_{2}^{\mathrm{reg}}$, we see that $Y_{1}\cap Y_{2}$ is smooth and equal dimensional. Then by \cite[Appendix B (9) and (31)]{F96}, we have each $m_{i}=1$ and hence $[Y_{1}]\cap[Y_{2}]=\sum_{k=1}^{r}[Z_{i}]=[Y_{1}\cap Y_{2}]$. 
%\end{proof}

The following lemma follows from the aforementioned basic facts about Borel-Moore homology and Corollary~\ref{cor:stein}. 
 \begin{lem}
 The set of fundamental classes $\{[T_\xi^*] ~|~ \xi \in\Xi_{\gamma\nu}\}$ forms a basis of $H(Z_{\gamma\nu})$.
 \end{lem}

%For any three sets $S_1,S_2,S_3$, let $$p_{ij}: S_1\times S_2\times %S_3\rightarrow S_i\times S_j,\quad (1\leq i<j\leq3)$$ be the natural %projection to the $(i,j)$-factor.

%-----------------------------------
\subsection{Convolution in Borel-Moore homology}\label{subsec:conv}
%-----------------------------------
Let $M_1$, $M_2$, $M_3$ be connected smooth complex varieties with real dimensions $d_1,d_2,d_3$, respectively. Let
$$Z_{12} \in M_1 \times M_2, \quad Z_{23} \in M_2 \times M_3$$
be two subsets. Define the set-theoretic composition $Z_{12} \circ Z_{23}$ as
\begin{align*}
    Z_{12}\circ Z_{23}:=
    \{(x_1,x_3)\in M_1 \times M_3~|~&\mbox{there exists}\ x_2 \in M_2 \qquad\\
    &\mbox{such that}\  (x_1,x_2,x_3) \in Z_{12} \times_{M_2} Z_{23}\}.
\end{align*}

A convolution in Borel-Moore homology 
\begin{align} \label{convolutionBM}
    *: H_i(Z_{12}) \times H_j(Z_{23}) \rightarrow H_{i+j-d_2}(Z_{12} \circ Z_{23})
\end{align}
is provided in \cite[\S 2.7]{CG97} under the assumption that $Z_{12}$ and $Z_{23}$ are closed subsets. For our purpose, we have to adapt to the case that $Z_{12}$ and $Z_{23}$ are locally closed subsets. The essential step is to generalize the intersection pairing in \cite[\S2.6]{CG97} as follows.

Let $Y$ and $Y'$ be two locally closed subsets of a connected smooth complex variety $M$. Let $V$ be an arbitrary open subset such that $Y \cap V, Y' \cap V$ are closed in $V$ and $Y \cap Y' \subset V$. we define an intersection paring 
\begin{align*}
    \cap:\ &H_i(Y) \times H_j(Y') \rightarrow H_i(V \cap Y) \times H_j(V \cap Y') \xrightarrow{\mbox{\tiny Poincar\'e duality}} \\
    &H^{m-i}(V, V\backslash Y) \times H^{m-j}(V, V\backslash Y') \xrightarrow{\cup} H^{2m-i-j}(V, V\backslash (Y\cap Y')) \xrightarrow{\mbox{\tiny Poincar\'e duality}} \\
    &H_{i+j-m}(Y \cap Y'),
\end{align*}
which is independent of the choice of $V$.

With this intersection pairing, we can introduce a convolution \eqref{convolutionBM} in the same way as in \cite[\S2.7]{CG97} under the assumption that $Z_{12}$, $Z_{23}$ and $Z_{12} \circ Z_{23}$ are locally closed subsets and the map $Z_{12} \times_{M_2} Z_{23} \to Z_{12} \circ Z_{23}$ is proper. 

Considering the convolution on each connected component of $Z_\ff$, $H_\bullet(Z_\ff)$ forms an associative algebra over $\mathbb{Q}$.
Particularly, if $\sF_\ff=\sB$, the Borel-Moore homology group $H_\bullet(Z)$ admits a convolution product and forms an associative algebra. It is known (cf. \cite[\S3.4]{CG97}) that the top Borel-Moore homology $H(Z)$ is a subalgebra of $H_\bullet(Z)$. Below is a generalization to the case of partial flag variety $\sF_\ff$.   
\begin{lem}
    The subspace $H(Z_{\ff})$ is a subalgebra of $H_\bullet(Z_\ff)$. 
\end{lem}
\begin{proof}
        It suffices to show the image of convolution $\ast:H_{\mathrm{top}}(Z_{\gamma\nu})\times H_{\mathrm{top}}(Z_{\nu\mu})\to H_{\bullet}(Z_{\gamma\mu})$ is in $H_{\mathrm{top}}(Z_{\gamma\mu})$, which follows from the fact $\dim_\mathbb{C}Z_{\gamma\nu}+\dim_\mathbb{C}Z_{\nu\mu}-\dim_\mathbb{C} \widetilde{\mathcal{N}}_{\nu}=\dim_\mathbb{C}Z_{\gamma\mu}$ by Corollary~\ref{cor:dim}.
\end{proof}

%-----------------
\subsection{A commutative diagram}
%------------------
The following lemma will be used in the computations later.
\begin{lem}\label{con dia}
    Let $Z_{12}'$ (resp. $Z_{23}'$) be an open subset of $Z_{12}$ (resp. $Z_{23}$) such that $Z_{12}' \circ Z_{23}'$ is also open in $Z_{12} \circ Z_{23}$. If the following Cartisian diagram holds:
    $$\begin{tikzcd}
        Z_{12}' \times_{M_2} Z_{23}' \ar[r] \ar[d] & Z_{12}' \circ Z_{23}' \ar[d] \\
        Z_{12} \times_{M_2} Z_{23} \ar[r] & Z_{12} \circ Z_{23},
    \end{tikzcd}$$
    then for any open subset $V\subset Z_{12}' \circ Z_{23}'$, we have the following commutative diagram:
    $$\begin{tikzcd}
        H_i(Z_{12}) \times H_j(Z_{23}) \ar[r,"*"] \ar[d] & H_{i+j-d_2}(Z_{12} \circ Z_{23}) \ar[d] \ar[dr] &\\
        H_i(Z_{12}') \times H_j(Z_{23}') \ar[r,"*"] & H_{i+j-d_2}(Z_{12}' \circ Z_{23}') \ar[r] & H_{i+j-d_2}(V).
    \end{tikzcd}$$
\end{lem}

\begin{proof}
    We have two natural commutative diagrams as follows:
    $$\begin{tikzcd}
        H_i(Z_{12}) \times H_j(Z_{23}) \ar[r] \ar[d] & H_{i+d_3}(Z_{12} \times M_3) \times H_{j+d_1}(M_1 \times Z_{23}) \ar[d]\\
        H_i(Z_{12}') \times H_j(Z_{23}') \ar[r] & H_{i+d_3}(Z_{12}' \times M_3) \times H_{j+d_1}(M_1 \times Z_{23}'),
    \end{tikzcd}$$
    and
    $$\begin{tikzcd}
        H_{i+j-d_2}(Z_{12} \circ Z_{23}) \ar[d] \ar[dr] &\\
        H_{i+j-d_2}(Z_{12}' \circ Z_{23}') \ar[r] & H_{i+j-d_2}(U).
    \end{tikzcd}$$
    The Cartisian diagram gives the following commutative diagram
    $$\begin{tikzcd}
    H_{i+j-d_2}(Z_{12} \times_{M_2} Z_{23}) \ar[r] \ar[d] & H_{i+j-d_2}(Z_{12} \circ Z_{23}) \ar[d]\\
    H_{i+j-d_2}(Z_{12}' \times_{M_2} Z_{23}') \ar[r] & H_{i+j-d_2}(Z_{12}' \circ Z_{23}').
    \end{tikzcd}$$
    Therefore, we just need to verify that the diagram 
    $$\begin{tikzcd}
        H_{i+d_3}(Z_{12} \times M_3) \times H_{j+d_1}(M_1 \times Z_{23}) \ar[r,"\cap"] \ar[d] & H_{i+j-d_2}(Z_{12} \times_{M_2} Z_{23}) \ar[d] \\
        H_{i+d_3}(Z_{12}' \times M_3) \times H_{j+d_1}(M_1 \times Z_{23}') \ar[r,"\cap"] & H_{i+j-d_2}(Z_{12}' \times_{M_2} Z_{23}')
    \end{tikzcd}$$
    is commutative. Without loss of generality, we may assume that $Z_{12}$ and $Z_{23}$ are closed. Let $V_{12}$ (resp. $V_{23}$) be an arbitrary open subset of $M_1 \times M_2$ (resp. $M_2 \times M_3$) such that $Z_{12}'$ (resp. $Z_{23}'$) is a closed subset of $V_{12}$ (resp. $V_{23}$). Then we have the commutative diagram:
    $$\begin{tikzcd}
        H_{i+d_3}(Z_{12} \times M_3) \times H_{j+d_1}(M_1 \times Z_{23}) \ar[r,"\cap"] \ar[d] & H_{i+j-d_2}(Z_{12} \times_{M_2} Z_{23}) \ar[dd] \\
        H_{i+d_3}(Z_{12}' \times M_3) \times H_{j+d_1}(M_1 \times Z_{23}') \ar[d] &\\
        H_{i+d_3}(Z_{12}' \times_{M_2} V_{23}) \times H_{j+d_1}(V_{12} \times_{M_2} Z_{23}') \ar[r,"\cap"] & H_{i+j-d_2}(Z_{12}' \times_{M_2} Z_{23}'),
    \end{tikzcd}$$
    from which we obtain the commutativity of the previous diagram. Hence the lemma is valid.
\end{proof}

%----------------------------
\subsection{Lagrangian construction of Weyl groups}
%----------------------------
For each $w \in \W$, a homology class $\Lambda_w^0 \in H(Z)$ is defined in \cite[\S 3.4]{CG97}. It is shown therein that the classes $\{\Lambda_w^0~|~w\in\W\}$ form a basis of $H(Z)$. Moreover, 
\begin{equation}\label{eq:Tww}
    \Lambda_w^0 = [T_w^*]+\sum_{y<w} n_{yw} [T_y^*], \quad \mbox{for some $n_{yw}\in \Z_{>0}$}.
\end{equation}

The following theorem \cite[Theorem~3.4.1]{CG97} gives a Lagrangian construction of Weyl groups.
	\begin{thm}[Kashiwara-Tanisaki, Ginzburg]\label{cg}
 		There is an algebra isomorphism $$\Q \W \simeq H(Z), \quad w \mapsto \Lambda_w^0.$$
	\end{thm}
    We shall identify $\Q \W$ with $H(Z)$ by the isomorphism above (i.e. $w=\Lambda_w^0$ by abuse of notations). As the unit element of $\Q\W=H(Z)$, it is clear that $\id=\Lambda_{\id}^0=[T_{\id}^*]$. Moreover, \cite[Theorem~0.4]{Sa13} tells us that
\begin{equation}\label{eq:si0}
[T_{s_i}^*]=\Lambda_{s_i}^0-[T_{\id}^*]=s_i-\id
 \quad\mbox{for any simple reflection $s_i\in\W$}.
\end{equation}

    %Set $\widetilde{\mathfrak{g}}=\{(x, gB) \in \mathfrak{g} \times \sB ~|~ x \in g(\mathfrak{b})\}$ which can be regarded as the homogeneous vector bundle $G \times^B \mathfrak{b}$.

%----------------------------------------
\subsection{Some multiplication formulas}
%---------------------------------------- 
In this subsection, we are going to get some technical multiplication formulas. First we introduce the following obvious maps: for $1\leq i<j\leq3$,
\begin{align*}%\label{def:pij}
p_{ij}&:M_1\times M_2\times M_3\rightarrow M_i \times M_j,\\ \nonumber
pr_{ij}&:T^*(M_1\times M_2\times M_3)\rightarrow T^*(M_i \times M_j),
\end{align*}
where $M_1, M_2, M_3$ are complex manifolds (perhaps disconnected manifolds with variable dimensions).

	\begin{prop}\label{2}
		For any $\gamma\in\Lambda_\ff$ and $s_i\in J_\gamma$, we have 
  \begin{equation}\label{eq:Tsi}
      [T_{\gamma,\id}^*] * [T_{s_i}^*]=-2[T_{\gamma,\id}^*]\quad\mbox{and}\quad[T_{\gamma,\id}^*] * s_i=-[T_{\gamma,\id}^*].
  \end{equation}
	\end{prop}
	\begin{proof}
 Recall in \eqref{eq:obsi} that the closure $\overline{\Ob}_{s_i}$ is a smooth variety, so $[T_{s_i}^*]=[T_{\overline{\Ob}_{s_i}}^*(\sB \times \sB)]$.
 Since $\alpha_i\in \Delta_\gamma$, we have $\Ob_{\gamma,\id} \circ \overline{\Ob}_{s_i}=\Ob_{\gamma,\id}$.
 
 We shall use \cite[Theorem~2.7.26]{CG97} (see also a correct formula in \cite[Lemma~8.5]{N98}) to derive $[T_{\gamma,\id}^*] * [T_{s_i}^*]=-2[T_{\gamma,\id}^*]$. For this purpose, we need to check the following two conditions:
 \begin{itemize}
     \item[(1)] The intersection of $p_{12}^{-1}(\Ob_{\gamma,\id})$ and $p_{23}^{-1}(\overline{\Ob}_{s_i})$ is transverse;
     \item[(2)] The map $p_{13}: \Ob_{\gamma,\id} \times_\sB \overline{\Ob}_{s_i} \rightarrow \Ob_{\gamma,\id}$ is a locally trivial fibration with smooth and compact fiber.
 \end{itemize}

 Thanks to \cite[2.7.27 (ii)]{CG97}, we know  that $p_{12}^{-1}(\Ob_{\gamma,\id})$ and $p_{23}^{-1}(\overline{\Ob}_{s_i})$ intersect transversely at any $x\in \Ob_{\gamma,\id} \times_{\sB} \Ob_{s_i}$. For any $x \in \Ob_{\gamma,\id} \times_{\sB} \Ob_\id$, by \cite[2.7.27 (ii)]{CG97} again, we have $$T_x(\sF_\gamma \times \sB \times \sB)=T_x(p_{12}^{-1}(\Ob_{\gamma,\id}))+T_x(p_{23}^{-1}(\Ob_\id)) \subset T_x(p_{12}^{-1}(\Ob_{\gamma,\id}))+T_x(p_{23}^{-1}(\overline{\Ob}_{s_i})).$$ So $p_{12}^{-1}(\Ob_{\gamma,\id})$ and $p_{23}^{-1}(\overline{\Ob}_{s_i})$ intersect transversely. The Condition (1) holds.

Explicitly, $\Ob_{\gamma,\id} \times_{\sB} \overline{\Ob}_{s_i}=\{(gP_\gamma, gB, ghB) \in \sF_\gamma \times \sB \times \sB ~|~ g \in G, h \in P_i\}$.
        Moreover, $p_{13}: \Ob_{\gamma,\id} \times_\sB \overline{\Ob}_{s_i} \rightarrow \Ob_{\gamma,\id}$ is a locally trivial fibration with fiber $P_i/B \simeq \mathbb{P}^1$, which is smooth and compact. Hence the Condition (2) is derived. 

    Note that the complex dimension $\dim_\mathbb{C} (P_i/B)=1$ and the Euler character number $e(P_i/B)=2$. We apply \cite[Lemma~8.5]{N98} to obtain
		$$[T_{\gamma,\id}^*] * [T_{s_i}^*]=(-1)^{\dim_\mathbb{C} (P_i/B)}e(P_i/B) [T_{\gamma,\id}^*]=-2[T_{\gamma,\id}^*],$$ which further implies $[T_{\gamma,\id}^*] * s_i=-[T_{\gamma,\id}^*]$ by \eqref{eq:si0}.
	\end{proof}

Since each $[T_\xi^*]$ is an algebraic cycle and $\Ob_{\gamma,\id} \circ \overline{\Ob}_w =\overline{\Ob}_{\gamma,w}$, then  
\begin{equation}\label{eq1}
    [T_{\gamma,\id}^*] * [T_w^*]\in\sum_{y \in \D_\gamma, y \leq w} \Z [T_{\gamma, y}^*],\quad(\forall\gamma\in\Lambda_\ff, w \in \D_\gamma).
\end{equation} 
Below is a refinement for the case of $w \in \D_\gamma$, which says that the leading coefficient of the product is exactly $1$.
    
    \begin{prop}\label{5}
        Let $\gamma\in\Lambda_\ff$ and $w \in \D_\gamma$. We have $$[T_{\gamma,\id}^*] * [T_w^*] \in [T_{\gamma, w}^*]+ \sum_{y \in \D_\gamma, y < w} \Z [T_{\gamma, y}^*].$$
    \end{prop}  
\begin{proof}
By \eqref{eq1},
       we have
		\begin{align*}
			[T_{\gamma,\id}^*] * [T_w^*]= k[T_{\gamma, w}^*]+ \mbox{lower terms},\quad\mbox{for some $k \in \Z$},
		\end{align*}
  and $pr_{13}^{-1}(T_{\Ob_{\gamma,w}}^*) \cap (T_{\gamma,\id}^* \times_{\widetilde{\mathcal{N}}} T_w^*) \subset T_{\gamma,\id}^* \times_{\widetilde{\mathcal{N}}} T_{\Ob_w}^*$. Take $x \in \Ob_{\gamma,w}$ and without loss of generality, we may assume $x=(P_\gamma,wB)$. Then 
    \begin{align*}
        p_{13}^{-1}(x)&=\{P_{\gamma}\}\times ((P_{\gamma}/B)\cap( wBw^{-1}B/B))\times\{wB\}.
    \end{align*} 
     By Lemma~\ref{lem:para}, $(P_{\gamma}/B)\cap (wBw^{-1}B/B)=\{B\}$, hence
     $$p_{13}:\Ob_{\gamma,\id} \times_{\sB} \Ob_w \to \Ob_{\gamma,\id} \circ \Ob_w=\Ob_{\gamma,w}$$
     is an isomorphism. Then by \cite[2.7.27.(iii)]{CG97}, $pr_{12}^{-1}(T_{\gamma,\id}^*)$ and $pr_{23}^{-1}(T_{\Ob_w}^*)$ intersect transversely and $pr_{13}: T_{\gamma,\id}^* \times_{\widetilde{\mathcal{N}}} T_{\Ob_w}^* \to T_{\Ob_{\gamma, w}}^*$ is an isomorphism. Therefore,\\ $[T_{\gamma,\id}^*] * [T_{\Ob_w}^*]=[T_{\Ob_{\gamma,w}}^*]$ and we have a Cartesian diagram
   \begin{equation*}%\label{carts1}
       \begin{tikzcd}
        T_{\gamma,\id}^* \times_{\widetilde{\mathcal{N}}} T_{\Ob_w}^* \ar[r] \ar[d] & T_{\gamma,\id}^* \times_{\widetilde{\mathcal{N}}} T_w^* \ar[d]\\
        T_{\Ob_{\gamma,w}}^* \ar[r] & \bigcup_{y \leq w} T_{\Ob_{\gamma,y}}^*.
    \end{tikzcd}
   \end{equation*}
    By Lemma \ref{con dia}, $k[T_{\Ob_{\gamma,w}}^*]=([T_{\gamma,\id}^*] * [T_w^*])|_{T_{\Ob_{\gamma,w}}^*}=[T_{\gamma,\id}^*] * [T_{\Ob_w}^*]=[T_{\Ob_{\gamma,w}}^*]$,
    i.e. $k=1$ as we desired.
\end{proof}

For $\gamma,\nu\in\Lambda_\ff$ and $w\in\D_{\gamma\nu}$, recall in \S\ref{sec:schur} the notations $\theta_\gamma$ and $w_{\gamma\nu}^+$ for the unique longest element in $\W_\gamma$ and $\W_\gamma w\W_\nu$, respectively. 
It follows from a theorem due to Howlett (cf. \cite[Theorem~4.18]{DDPW08}) that $w^+_{\gamma\nu}=\theta_\gamma w y$ for some $y\in \W_\nu$ with $w y \in\D_\gamma$. So $\theta_\gamma w_{\gamma\nu}^+=\theta_\gamma^2 w y=wy \in \D_\gamma$. Similarly, $w_{\gamma\nu}^+\theta_\nu\in\D_\nu^{-1}$.

\begin{lem}\label{lem2}
Suppose $P_\mu\subset P_\nu$. Then $\W_\gamma w\W_{\nu}\cap\mathcal{D}_{\gamma\mu}$ has a unique longest element $w'$. Furthermore,
one has $$T^*_{\gamma,w,\nu}\circ T_{\nu, \id, \mu}^* = T_{\gamma,w',\mu}^*.$$
\end{lem}
\begin{proof}
    Consider the canonical projection $\pi_{\gamma,\mu\nu}:\sF_{\gamma}\times\sF_{\mu}\longrightarrow\sF_{\gamma}\times\sF_{\nu}$, which is a fibration. We have $$\pi^{-1}_{\gamma,\mu\nu}(\Ob_{\gamma,w,\nu})=\bigsqcup\limits_{y\in\W_\gamma w\W_\nu\cap\mathcal{D}_{\gamma\mu}}\Ob_{\gamma,y,\mu},$$ which is irreducible. So it has a unique open dense $G$-orbit 
    $\Ob_{\gamma,w',\mu}$ in $\pi^{-1}_{\gamma,\mu\nu}(\Ob_{\gamma,w,\nu})$ and $$\bigsqcup\limits_{y\in\W_\gamma w\W_\nu\cap\mathcal{D}_{\gamma\mu}}\Ob_{\gamma,y,\mu}\subset\overline{\Ob}_{\gamma,w',\mu}.$$ By Lemma \ref{partial Bruhat order}, $w'$ is the unique longest element in $\W_\gamma w\W_\nu\cap\mathcal{D}_{\gamma\mu}$. 

    For the rest part, observe that $pr'_{12}: pr_{23}^{-1}(T_{\nu,\id,\mu}^*) \to \widetilde{\mathcal{N}}_\gamma \times \widetilde{\mathcal{N}}_\nu$ is a locally trivial fibration with fiber $P_\nu/P_\mu$. Therefore, $(pr'_{12})^{-1}(T^*_{\gamma,w,\nu})=T^*_{\gamma,w,\nu} \times_{\widetilde{\mathcal{N}}_\nu} T_{\nu,\id,\mu}^*$ is irreducible and further so is $T_{\gamma,w,\nu}^* \circ T_{\nu,\id}^*$.
Note that $\Ob_{\gamma,w',\mu}$ is a smooth open dense subvariety of $\Ob_{\gamma,w,\nu}\circ\Ob_{\nu,\id,\mu}$, hence $T_{\Ob_{\gamma,w',\mu}}^*\subset T_{\gamma,w,\nu}^* \circ T_{\nu,\id,\mu}^*$. Since $T_{\Ob_{\gamma,w',\mu}}^*$ is also irreducible, by taking closure on the both sides, we get $T_{\gamma,w',\mu}^*=T_{\gamma,w,\nu}^*\circ T_{\nu,\id,\mu}^*$.
\end{proof}

    \begin{prop}\label{6}
        For $\gamma,\nu\in\Lambda_\ff$ and $w \in \D_{\gamma\nu}$, we have
        \begin{equation}\label{eq:TT}
             [T_{\gamma, w, \nu}^*] * [T_{\nu,\id}^*]= [T_{\gamma, \theta_\gamma w_{\gamma\nu}^+}^*]\quad\mbox{and}\quad [T_{\id, \gamma}^*] * [T_{\gamma, w, \nu}^*]= [T_{w_{\gamma\nu}^+\theta_\nu, \nu}^*].
        \end{equation}
    As a consequence, 
    \begin{equation}\label{eq:Twrv}
        [T_{w_{\gamma\nu}^+}^*]=[T_{\id, \gamma}^*] * [T_{\gamma,w,\nu}^*] * [T_{\nu, \id}^*].
    \end{equation}
    \end{prop}
   \begin{proof} We only prove the first formula in \eqref{eq:TT}, while the other is similar.

   We see that
   $$p_{13}: p_{23}^{-1}(\Ob_{\nu,\id}) \to \sF_\gamma \times \sB,\quad (gP_\gamma,g'P_\nu,g'B) \mapsto (gP_\gamma,g'B)$$
   is an isomorphism and
   $$p_{12}: p_{23}^{-1}(\Ob_{\nu,\id}) \to \sF_\gamma \times \sF_\nu,\quad (gP_\gamma, g'P_\nu, g'B) \mapsto (gP_\gamma,g'P_\nu)$$
   is a locally trivial fibration with the fibre $P_\nu/B$. Then we have
   $p_{13}^{-1}(\Ob_{\gamma,w,\nu} \circ \Ob_{\nu,\id})=\Ob_{\gamma,w,\nu} \times_{\sF_\nu} \Ob_{\nu,\id}$ and $\Ob_{\gamma,w,\nu} \circ \Ob_{\nu,\id}$ is irreducible, smooth and open dense in $\overline{\Ob}_{\gamma,w,\nu} \circ \Ob_{\nu,\id}$. Note that $\Ob_{\gamma,\theta_\gamma w_{\gamma\nu}^+}$ is an open subvariety of $\Ob_{\gamma,w,\nu}\circ\Ob_{\nu,\id}$, so $\Ob_{\gamma,w,\nu} \circ \Ob_{\nu,\id}$ is open in $\overline{\Ob}_{\gamma,\theta_\gamma w_{\gamma\nu}^+}=\overline{\Ob}_{\gamma,w,\nu} \circ \Ob_{\nu,\id}$. Moreover, by \cite[2.7.27.(iii)]{CG97}, $pr_{12}^{-1}(T_{\Ob_{\gamma, w, \nu}}^*)$ and $pr_{23}^{-1}(T_{\nu,\id}^*)$ intersect transversely and
$$pr_{13}: T_{\Ob_{\gamma, w, \nu}}^* \times_{\widetilde{\mathcal{N}}_\nu} T_{\nu,\id}^* \to T_{\Ob_{\gamma, w, \nu} \circ \Ob_{\nu,\id}}^*(\sF_\gamma \times \sB)$$
is an isomorphism. Therefore, $T_{\Ob_{\gamma,w,\nu}}^*\circ T_{\nu,\id}^*=T_{\Ob_{\gamma, w, \nu} \circ \Ob_{\nu,\id}}^*(\sF_\gamma \times \sB)$ and $[T_{\Ob_{\gamma,w,\nu}}^*] * [T_{\nu,\id}^*]=[T_{\Ob_{\gamma,w,\nu}}^*\circ T_{\nu,\id}^*]$.

Lemma~\ref{lem2} implies $T_{\gamma,\theta_\gamma w_{\gamma\nu}^+}^*=T_{\gamma,w,\nu}^*\circ T_{\nu,\id}^*$, and hence $[T_{\gamma, w, \nu}^*] * [T_{\nu,\id}^*]= k [T_{\gamma, \theta_\gamma w_{\gamma\nu}^+}^*]$. We determine the number $k$ as follows. As mentioned above, $\Ob_{\gamma,w,\nu} \circ \Ob_{\nu,\id}$ is open in $\overline{\Ob}_{\gamma,\theta_\gamma w_{\gamma\nu}^+}$, therefore, $T_{\Ob_{\gamma,w,\nu}}^*\circ T_{\nu,\id}^*$ is open in $T_{\gamma, \theta_\gamma w_{\gamma\nu}^+}^*$. Thus we have the Cartesian diagram
    \begin{equation*}%\label{carts2}
        \begin{tikzcd}
        T_{\Ob_{\gamma,w,\nu}}^* \times_{\widetilde{\mathcal{N}}_\nu} T_{\nu,\id}^* \ar[r] \ar[d] & T^*_{\gamma,w,\nu} \times_{\widetilde{\mathcal{N}}_\nu} T_{\nu,\id}^* \ar[d]\\
        T_{\Ob_{\gamma,w,\nu}}^*\circ T_{\nu,\id}^* \ar[r] & T_{\gamma, \theta_\gamma w_{\gamma\nu}^+}^*
    \end{tikzcd}.
    \end{equation*}
    By Lemma \ref{con dia}, $$k[T_{\Ob_{\gamma,\theta_\gamma w_{\gamma\nu}^+}}^*]=([T^*_{\gamma,w,\nu}] * [T_{\nu,\id}^*])|_{T_{\Ob_{\gamma,\theta_\gamma w_{\gamma\nu}^+}}^*}=([T_{\Ob_{\gamma,w,\nu}}^*] * [T_{\nu,\id}^*])|_{T_{\Ob_{\gamma,\theta_\gamma w_{\gamma\nu}^+}}^*}=[T_{\Ob_{\gamma,\theta_\gamma w_{\gamma\nu}^+}}^*],$$
    i.e. $k=1$ as we desired.

Finally, let us verify \eqref{eq:Twrv} via \eqref{eq:TT} as follows: 
\begin{align*}
    [T_{\id, \gamma}^*] * [T_{\gamma,w,\nu}^*] * [T_{\nu, \id}^*]=[T_{\id, \gamma}^*]*[T_{\gamma,\theta_\gamma w_{\gamma\nu}^+}^*]=[T_{w_{\gamma\nu}^+}^*],
\end{align*}
where the second equality is derived by the fact that the longest element in $\W_\gamma \theta_\gamma w_{\gamma\nu}^+=\W_\gamma w_{\gamma\nu}^+$ is $w_{\gamma\nu}^+$.
\end{proof}

    \begin{cor}\label{7} It holds 
        $[T_{\theta_\gamma}^*]=x_\gamma$ for any $\gamma\in\Lambda_\ff$.
    \end{cor}
    \begin{proof}
Thanks to \eqref{eq:Tww}, we can assume 
$$[T_{\theta_\gamma}^*]=\theta_\gamma+\sum\limits_{w<\theta_\gamma}m_{w}w=\theta_\gamma+\sum\limits_{w\in\W_\gamma\backslash\{\theta_\gamma\}}m_{w}w, \quad\mbox{for some}\quad m_w\in\Z.$$
    
It follows from \eqref{eq:Tsi} and \eqref{eq:Twrv} that, for any $s_i \in J_\gamma$, $$[T_{\theta_\gamma}^*] * s_i=[T_{\id, \gamma}^*] * [T_{\gamma,\id}^*] * s_i=-[T_{\id, \gamma}^*] * [T_{\gamma,\id}^*]=-[T_{\theta_\gamma}^*],$$
        which implies $$(-1)^{\ell(\theta_\gamma)-\ell(w)}[T_{\theta_\gamma}^*]=[T_{\theta_\gamma}^*]*\theta_\gamma*w=w+\sum\limits_{y\in\W_{\gamma}\backslash\{\theta_\gamma\}}m_{y}y\theta_\gamma w$$ for any $w\in\W_{\gamma}$. Comparing the coefficients of $w$ on the both sides takes us to $m_{w}=(-1)^{\ell(\theta_\gamma)-\ell(w)}$. Hence $[T_{\theta_\gamma}^*]=\sum_{w\in\W_\gamma}(-1)^{\ell(\theta_\gamma)-\ell(w)}w=x_\gamma$.
    \end{proof}

%-------------------
\subsection{Lagrangian construction of Schur algebras}
%--------------------------
      It is clear that  $H(\widetilde{\mathcal{N}}_\ff \times_\mathcal{N} \widetilde{\mathcal{N}})$ admits a right $\W$-action by $[T_{\gamma,w}^*] \cdot w':=[T_{\gamma,w}^*] * \Lambda_{w'}^0$.
\begin{lem}
    There exists a unique right $\mathbb{QW}$-module isomorphism $$\varphi: \T_\ff \rightarrow H(\widetilde{\mathcal{N}}_\ff \times_\mathcal{N} \widetilde{\mathcal{N}})\quad \mbox{satisfying}\quad x_\gamma \mapsto [T_{\gamma,\id}^*] (\gamma\in\Lambda_\ff).$$ Moreover, for $w\in \D_\gamma$,
            \begin{align}\label{eq:varphi}
                \varphi(x_\gamma w)=[T_{\gamma, w}^*]+\sum_{y\in\D_{\gamma},y<w} m_{yw}^\gamma [T_{\gamma, y}^*]\quad\mbox{with $m_{yw}^\gamma \in \Z$.}
            \end{align}
\end{lem}
  \begin{proof}
Combining \eqref{eq:xsi}, \eqref{eq:si0} and \eqref{eq:Tsi}, we know that $x_\gamma \mapsto [T_{\gamma,\id}^*] (\gamma\in\Lambda_\ff)$ determines a $\Q\W$-module homomorphism $\varphi: \T_\ff \rightarrow H(\widetilde{\mathcal{N}}_\ff \times_\mathcal{N} \widetilde{\mathcal{N}})$. It follows from Proposition~\ref{5} that $\varphi(x_\gamma w)=[T_{\gamma,\id}^*]*\Lambda_w^0 \in [T_{\gamma, w}^*]+ \sum_{y<w} \Z [T_{\gamma, y}^*]$ admits an upper-triangular form with all leading coefficients being $1$, which implies that $\varphi$ must be an isomorphism.
\end{proof}

The above lemma (together with Theorem~\ref{cg}) implies that $$\mathbb{S}_\ff\simeq\End_{H(Z)} (H(\widetilde{\mathcal{N}}_\ff \times_\mathcal{N} \widetilde{\mathcal{N}})).$$ 
Furthermore, we shall show in the following theorem that $\mathbb{S}_\ff\simeq H(Z_\ff)$, which is called the Lagrangian construction of the Schur algebra $\mathbb{S}_\ff$.
	\begin{thm}\label{geo schur}
       There is an isomorphism $\varphi$ between $H(Z_\ff)$ and $\mathbb{S}_\ff$:      
       \begin{align}
                \psi: H(Z_\ff) &\rightarrow \mathbb{S}_\ff,\nonumber\\
                \label{iso:psi}
                [T_{\gamma,w,\nu}^*] &\mapsto \phi_{\gamma\nu}^w+ \sum_{y<w} p_{yw}^{\gamma\nu}\phi_{\gamma\nu}^y
            \end{align}
            for $w \in \D_{\gamma\nu}$ and some $p_{yw}^{\gamma\nu} \in \Z$.
            Furthermore, the coefficients satisfy $p_{yw}^{\gamma\nu}=p_{y^{-1}w^{-1}}^{\nu\gamma}$.
	\end{thm}
	
	\begin{proof}
		By the associativity of convolution, we have a natural algebra homomorphism $$\psi: H(Z_\ff) \to \End_{H(Z)} (H(\widetilde{\mathcal{N}}_\ff \times_\mathcal{N} \widetilde{\mathcal{N}})) \simeq \mathbb{S}_\ff.$$
		By Proposition~\ref{6} and \eqref{eq:varphi}, we see that $\psi$ is injective and
        \begin{align*}
            \psi([T_{\gamma,w,\nu}^*]) (x_\nu)= \varphi^{-1}([T_{\gamma, \theta_\gamma w_{\gamma\nu}^+}^*]) \in x_\gamma \theta_\gamma w_{\gamma\nu}^+ +\sum_{y<\theta_\gamma w_{\gamma\nu}^+} \Z x_\gamma y,
        \end{align*}
        which implies $\psi([T_{\gamma,w,\nu}^*]) \in \phi_{\gamma\nu}^w+ \sum_{y<w} \Z \phi_{\gamma\nu}^y$ and hence $p_{yw}^{\gamma\nu} \in \Z$.
        Comparing the dimensions, we know that $\psi$ should be an isomorphism. 

Below we shall prove $p_{yw}^{\gamma\nu}=p_{y^{-1}w^{-1}}^{\nu\gamma}$. 
        
        It follows from the associativity of convolution again that $$L_\gamma: H(\widetilde{\mathcal{N}}_\ff \times_\mathcal{N} \widetilde{\mathcal{N}}) \rightarrow \mathbb{Q}\W,\quad c \mapsto [T_{\id, \gamma}^*]*c$$ defines a right $\mathbb{Q}\W$-module homomorphism. Then $L_\gamma \circ \varphi$ defines a right $\mathbb{Q}\W$-module homomorphism from $\T_\ff$ to $\mathbb{Q}\W$. We compute directly that 
        $$(L_\gamma \circ \varphi)(x_\gamma)=[T_{\id, \gamma}^*]*[T_{\gamma,\id}^*]=[T_{\theta_\gamma}^*]=x_\gamma,$$
        where the first equality holds by definitions, and the second (resp. third) equality holds by \eqref{eq:TT} (resp. Corollary~\ref{7}).
        Thus by \eqref{eq:Twrv}, we calculate
        \begin{align*}
            [T_{w_{\gamma\nu}^+}^*]&=[T_{\id, \gamma}^*] * [T_{\gamma,w,\nu}^*] * [T_{\nu, \id}^*]=(L_\gamma \circ \varphi)\circ \psi([T_{\gamma,w,\nu}^*])(x_\nu)\\
            &=(-1)^{\ell(w_{\gamma\nu}^+)}\mathrm{sgn}(\W_\gamma w \W_\nu)+ \sum_{y<w} (-1)^{\ell(y_{\gamma\nu}^+)}p_{yw}^{\gamma\nu}\mathrm{sgn}(\W_\gamma y \W_\nu),
        \end{align*}
        where $\mathrm{sgn}(\W_\gamma y \W_\nu):=\sum_{\sigma\in\W_\gamma y \W_\nu}(-1)^{\ell(\sigma)}\sigma$.

               There is an anti-involution $c\mapsto c^{t}$ on $H(Z)$ induced by switching factors on $\widetilde{\mathcal{N}} \times \widetilde{\mathcal{N}}$ (refer to \cite[Page 179]{CG97}). One has $(T^{*}_{w})^{t}=(\overline{T^{*}_{\Ob_{w}}})^{t}=\overline{({T^{*}_{\Ob_{w}}})^t}=\overline{T^{*}_{\Ob_{w^{-1}}}}=T^{*}_{w^{-1}}$ and hence $[T^{*}_{w}]^{t}=[T^{*}_{w^{-1}}]$.
        Noting that $(w_{\nu\gamma}^+)^{-1}$ is the unique longest element of $\W_{\nu} w^{-1}\W_{\gamma}$. So
         \begin{align*}
[T_{w_{\gamma\nu}^+}^*]^{t}&=[T^{*}_{(w_{\nu\gamma}^+)^{-1}}]=(-1)^{\ell(w_{\gamma\nu}^+)}\mathrm{sgn}(\W_\nu w^{-1} \W_\gamma)+ \sum_{y<w^{-1}} (-1)^{\ell(y_{\gamma\nu}^+)}p_{yw^{-1}}^{\nu\gamma}\mathrm{sgn}(\W_\nu y \W_\gamma)\\
            &=(-1)^{\ell(w_{\gamma\nu}^+)}\mathrm{sgn}(\W_\nu w^{-1} \W_\gamma)+ \sum_{y<w} (-1)^{\ell(y_{\gamma\nu}^+)}p_{y^{-1}w^{-1}}^{\nu\gamma}\mathrm{sgn}(\W_\nu y^{-1} \W_\gamma).
        \end{align*}
        On the other hand, by \cite[Lemma~3.6.11]{CG97}, we have 
        \begin{align*}
         [T_{w_{\gamma\nu}^+}^*]^{t}&=(-1)^{\ell(w_{\gamma\nu}^+)}(\mathrm{sgn}(\W_\gamma w \W_\nu))^{t}+ \sum_{y<w} (-1)^{\ell(y_{\gamma\nu}^+)}p_{yw}^{\gamma\nu}(\mathrm{sgn}(\W_\gamma y \W_\nu))^{t}\\&=(-1)^{\ell(w_{\gamma\nu}^+)}\mathrm{sgn}(\W_\nu w^{-1} \W_\gamma)+ \sum_{y<w} (-1)^{\ell(y_{\gamma\nu}^+)}p_{yw}^{\gamma\nu}\mathrm{sgn}(\W_\nu y^{-1} \W_\gamma).
        \end{align*}
        Comparing the coefficients takes us to $p_{yw}^{\gamma\nu}=p_{y^{-1}w^{-1}}^{\nu\gamma}$.
	\end{proof}

\begin{rem}
    It follows from \eqref{iso:psi} directly that the isomorphism $\psi$ is still valid over $\Z$ instead of $\mathbb{Q}$.
\end{rem}

%========
\section{Representations of Schur algebras}
%==================================================================
%-----------------
\subsection{Partial order for nilpotent orbits}
%--------------------
There is a partial order on the set $\{\mathcal{N}_\alpha~|~\alpha\in\mathbb{J}\}$ of nilpotent orbits as follows:
\begin{equation*}
\mathcal{N}_\alpha\leq\mathcal{N}_{\alpha'}\stackrel{\mathrm{def.}}\Leftrightarrow\mathcal{N}_\alpha\subset\overline{\mathcal{N}}_{\alpha'},\qquad \mathcal{N}_\alpha <\mathcal{N}_{\alpha'}\stackrel{\mathrm{def.}}\Leftrightarrow\mathcal{N}_\alpha \subset\overline{\mathcal{N}}_{\alpha'}\backslash\mathcal{N}_{\alpha'}.
\end{equation*}
Recall in \S~\ref{sec:steinberg} the notation $Z_{\gamma\nu}^{\mathcal{S}}$ for any subset $\mathcal{S}\in\mathcal{N}$.
Set $$Z_{\gamma\nu}^{\leq \alpha}:=\bigsqcup_{\mathcal{N}_{\alpha'}\leq\mathcal{N}_\alpha}Z_{\gamma\nu}^{\mathcal{N}_{\alpha'}}=Z_{\gamma\nu}^{\overline{\mathcal{N}_\alpha}}\quad \text{and}\quad Z_{\gamma\nu}^{<\alpha}:=\bigsqcup_{\mathcal{N}_{\alpha'}<\mathcal{N}_\alpha}Z_{\gamma\nu}^{\mathcal{N}_{\alpha'}}.$$
 Recall $d_{\gamma\nu}=\dim_{\mathbb{R}}Z_{\gamma\nu}$ in \eqref{def:dln}.
 \begin{lem}\label{H_O} The homology $H_{d_{\gamma\nu}}(Z_{\gamma\nu}^{\mathcal{N}_{\alpha}})\neq 0$ if and only if $\mathcal{N}_\alpha$ is relevant for both $\pi_{\gamma}$ and $\pi_{\nu}$.
      Furthermore, there is a canonical isomorphism:$$H_{d_{\gamma\nu}}(Z^{\leq\alpha}_{\gamma\nu})/H_{d_{\gamma\nu}}(Z^{<\alpha}_{\gamma\nu})\simeq H_{d_{\gamma\nu}}(Z_{\gamma\nu}^{\mathcal{N}_{\alpha}}).$$    
       \end{lem}
 \begin{proof}
     The first statement follows from Lemma~\ref{lem:relZ} directly. Below let us prove the second statement.
     
     Note that $Z_{\gamma\nu}^{\mathcal{N}_\alpha}$ is an open subvariety of
 $Z_{\gamma\nu}^{\leq\alpha}$ such that $Z_{\gamma\nu}^{<\alpha}=Z_{\gamma\nu}^{\leq\alpha}\backslash Z_{\gamma\nu}^{\mathcal{N}_\alpha}$. Since $\dim_{\mathbb{C}}(Z_{\gamma\nu}^{\mathcal{N}_\alpha})\le\dim_{\mathbb{C}}(Z_{\gamma\nu})=\frac{1}{2}d_{\gamma\nu}$, the exact sequence of Borel-Moore homology 
 $$0\to H_{d_{\gamma\nu}}(Z^{<\alpha}_{\gamma\nu})\to H_{d_{\gamma\nu}}(Z^{\leq\alpha}_{\gamma\nu})\to H_{d_{\gamma\nu}}(Z_{\gamma\nu}^{\mathcal{N}_\alpha})\to\cdots$$
 yields an inclusion $H_{d_{\gamma\nu}}(Z^{\leq\alpha}_{\gamma\nu})/H_{d_{\gamma\nu}}(Z^{<\alpha}_{\gamma\nu})\hookrightarrow H_{d_{\gamma\nu}}(Z_{\gamma\nu}^{\mathcal{N}_\alpha})$. So we only need to consider the case  $H_{d_{\gamma\nu}}(Z_{\gamma\nu}^{\mathcal{N}_\alpha})\neq 0$, which means that $\mathcal{N}_\alpha$ are relevant for both $\pi_\gamma$ and $\pi_\nu$. In this case, $H_{d_{\gamma\nu}}(Z_{\gamma\nu}^{\mathcal{N}_\alpha})$ is a finite dimensional $\mathbb{Q}$-vector space with a basis formed by the fundamental classes of top dimensional irreducible components of $Z_{\gamma\nu}^{\mathcal{N}_\alpha}$. The quotient $H_{d_{\gamma\nu}}(Z^{\leq\alpha}_{\gamma\nu})/H_{d_{\gamma\nu}}(Z^{<\alpha}_{\gamma\nu})$ is a $\mathbb{Q}$-vector space with basis formed by the fundamental classes of top dimensional irreducible components of $Z_{\gamma\nu}^{\leq\alpha}$ that meet $Z_{\gamma\nu}^{\mathcal{N}_\alpha}$, which equal to fundamental classes of Zariski closure of top dimensional irreducible components in $Z_{\gamma\nu}^{\mathcal{N}_\alpha}$. Hence the dimensions of these two vector spaces are the same, which forces that $H_{d_{\gamma\nu}}(Z^{\leq\alpha}_{\gamma\nu})/H_{d_{\gamma\nu}}(Z^{<\alpha}_{\gamma\nu})\simeq H_{d_{\gamma\nu}}(Z_{\gamma\nu}^{\mathcal{N}_\alpha})$.
 \end{proof}

%-----------------
\subsection{Relevant homology}
%--------------------
 Denote $$Z_{\ff}^{\leq \alpha}:=\bigsqcup\limits_{\gamma,\nu\in\Lambda_{\ff}}Z^{ \leq \alpha}_{\gamma\nu}\quad \mbox{and}\quad Z_\ff^{<\alpha}:=\bigsqcup\limits_{\gamma,\nu\in\Lambda_{\ff}}Z^{<\alpha}_{\gamma\nu}.$$
We define
$$H_{\rel}(Z_{\ff}^{\leq \alpha}):=\bigoplus_{\gamma,\nu\in\Lambda_{\ff}}H_{d_{\gamma\nu}}(Z_{\gamma\nu}^{\leq \alpha})\quad \mbox{and}\quad H_{\rel}(Z_{\ff}^{<\alpha}):=\bigoplus_{\gamma,\nu\in\Lambda_{\ff}}H_{d_{\gamma\nu}}(Z_{\gamma\nu}^{<\alpha}).$$
Since $Z_{\ff}^{\leq \alpha}\circ Z_\ff=Z_\ff\circ Z_{\ff}^{\leq \alpha}=Z_{\ff}^{\leq \alpha}$ and $Z_{\ff}^{<\alpha}\circ Z_\ff=Z\circ Z_{\ff}^{<\alpha}=Z_{\ff}^{<\alpha}$, we know that $H_{\rel}(Z_{\ff}^{\leq \alpha})$ and $H_{\rel}(Z_{\ff}^{<\alpha})$ are two-sided ideals of $H(Z_{\ff})$. 
Denote $$H_{\ff,\alpha}:=H_{\rel}(Z_{\ff}^{\leq \alpha})/H_{\rel}(Z_{\ff}^{<\alpha}) \simeq \bigoplus\limits_{\gamma,\nu\in\Lambda_{\ff}} H_{d_{\gamma\nu}}(Z_{\gamma\nu}^{\le a})/H_{d_{\gamma\nu}}(Z_{\gamma\nu}^{<a}).$$ Since $H(Z_\ff)$ is a finite dimensional semisimple algebra, we have an algebra isomorphism $$H(Z_\ff)\simeq \mathrm{gr}\ H(Z_\ff)= \bigoplus_{\alpha\in\mathbb{J}} H_{\ff,\alpha}.$$

Recall the notations about transversal slices in \S\ref{sec:transslice}. By Lemma~\ref{lem:vsdim}, we have that $H_{d_{\gamma\nu}}(Z^{V_\alpha}_{\gamma\nu})=H_{d'_{\gamma\nu}}(Z^{x_\alpha}_{\gamma\nu})\neq0$ only if $\mathcal{N}_\alpha$ is relevant for $\pi_{\gamma}$ and $\pi_{\nu}$. In this case, there are embeddings $H_{d_{\gamma\nu}}(Z^{V_\alpha}_{\gamma\nu})\hookrightarrow H_{d_{\gamma\nu}}(Z^{V}_{\gamma\nu})$ and $H_{d'_{\gamma\nu}}(Z^{x_\alpha}_{\gamma\nu})\hookrightarrow H_{d'_{\gamma\nu}}(Z^{S}_{\gamma\nu})$.

We denote
\begin{align*}
&H_{\rel}(Z_{\ff}^{V_\alpha})=\bigoplus\limits_{\gamma,\nu\in\Lambda_{\ff}}H_{d_{\gamma\nu}}(Z_{\gamma\nu}^{V_\alpha}),\qquad H_{\rel}(Z_{\ff}^{V})=\bigoplus\limits_{\gamma,\nu\in\Lambda_{\ff}} H_{d_{\gamma\nu}}(Z_{\gamma\nu}^{V}),\\
&H_{\rel}(Z_{\ff}^{x_\alpha})=\bigoplus\limits_{\gamma,\nu\in\Lambda_{\ff}} H_{{d'_{\gamma\nu}}}(Z_{\gamma\nu}^{x_\alpha}),\qquad H_{\rel}(Z_{\ff}^{S})=\bigoplus\limits_{\gamma,\nu\in\Lambda_{\ff}}H_{d'_{\gamma\nu}}(Z_{\gamma\nu}^{S}),\\
&H_\rel(Z_{\ff}^{\mathcal{N}_\alpha})=\bigoplus\limits_{\gamma,\nu\in\Lambda_{\ff}}H_{d_{\gamma\nu}}(Z_{\gamma\nu}^{\mathcal{N}_\alpha}).
\end{align*}
They satisfy the following commutative diagram:
$$\begin{tikzcd}\label{diagram 2}
H(Z_\ff) \arrow[r, "\mathrm{Res}_{V}"]                                                                 & H_{\rel}(Z_{\ff}^{V}) \arrow[r, "{g}"]                              & H_{\rel}(Z_{\ff}^{S})                 \\
H_{\rel}(Z_{\ff}^{\leq \alpha}) \arrow[u, hook] \arrow[r, "i"] \arrow[d, "{j}"] & H_{\rel}(Z_{\ff}^{V_\alpha}) \arrow[u, hook] \arrow[r, "{g}'"] & H_{\rel}(Z_{\ff}^{x_\alpha}) \arrow[u, hook] \\
H_{\rel}(Z_{\ff}^{\mathcal{N}_\alpha}) \arrow[ru, "{k}"]                                                  &                                                           &                         
\end{tikzcd},$$
\noindent where $\mathrm{Res}_V$ is the restriction map $[X]\mapsto[X\cap Z_\ff^V]$; $i,j,k$ are induced by open embedding; $g$ is the Gysin pull-back of the section map $Z_{\ff}^{S}\to Z_{\ff}^{V}: y \mapsto (x_\alpha,y)$ and $g'$ is its restriction. More precisely, we have $g([V_\alpha]\boxtimes c)=c$ for $c\in H_{\rel}(Z^{S}_{\ff})$.

\begin{lem}
    The maps $g$ and $g'$ are isomorphisms of vector spaces.
\end{lem}
\begin{proof}
    Since $V_\alpha$ is a connected manifold, we have $H(V_\alpha)=\mathbb{Q}\cdot[V_\alpha]\simeq \mathbb{Q}$. Note that $Z_{\ff}^V \simeq V_\alpha \times Z_{\ff}^{S}$ and $Z_{\ff}^{V_\alpha}\simeq V_a\times Z^{x_\alpha}_{\ff}$. Then by K\"unneth formula, the Gysin pull-backs $g$ and $g'$ are isomorphisms. 
\end{proof}

Recall $V_\alpha$ is a closed subset of $V$ so that $\widetilde{V}_\alpha$ is a closed subset of $\widetilde{V}$.
Taking $\widetilde{V} \times \widetilde{V}$ as the ambient space of $Z^{V_\alpha}_{\ff}$ and $Z_{\ff}^{x_\alpha}$, one has $Z_{\ff}^{V} \circ Z^{V}_{\ff}=Z^{V}_{\ff}$, $Z^{V}_{\ff}\circ Z_{\ff}^{V_\alpha}=Z_{\ff}^{V_\alpha} \circ Z^{V}_{\ff}=Z^{V_\alpha}_{\ff}$ and $Z^{V}_{\ff}\circ Z^{x_\alpha}_{\ff}=Z^{x_\alpha}_{\ff}\circ Z^{V}_{\ff}=Z^{x_\alpha}_{\ff}$.
 By $\widetilde{V}\simeq V_\alpha \times\widetilde{S}$, $\widetilde{S}$ is also smooth, $Z_{\ff}^{x_\alpha}$ and $Z_{\ff}^{S}$ are closed subsets of $\widetilde{S}\times\widetilde{S}$. Taking $\widetilde{S}\times\widetilde{S}$ as the ambient space of $Z^{x_\alpha}_{\ff}$, one has $Z^{S}_{\ff}\circ Z^{S}_{\ff}=Z^{S}_{\ff}$ and $Z_{\ff}^{S}\circ Z_{\ff}^{x_\alpha}=Z_{\ff}^{x_\alpha}\circ Z_{\ff}^{S}=Z_{\ff}^{x_\alpha}$.   
 The next lemma is a direct consequence of (\ref{cor:dim}) and Lemma \ref{dim2} by dimension counting.

\begin{lem}\label{two sided ideals of hZx}
Keep the notations above.
\begin{itemize}
    \item[(1)] $H_{\rel}(Z^{V}_{\ff})$ is a subalgebra of $H_\bullet(Z^{V}_{\ff})$ with $H_{\rel}(Z^{V_\alpha}_{\ff})$ as a two-sided ideal.
    \item[(2)]  $H_{\rel}(Z^{S}_{\ff})$ is a subalgebra of $H_\bullet(Z^{S}_{\ff})$ with $H_{\rel}(Z_{\ff}^{x_\alpha})$ as a two-sided ideal.
    \item[(3)] $H_{\rel}(Z^{x_\alpha}_{\ff})$ is a  $H_{\rel}(Z^{V}_{\ff})$-bimodule.
\end{itemize}
\end{lem}

For a variety (or a manifold) $M$, we write $\Delta M$ for the image of diagonal embedding $\Delta: M\hookrightarrow M\times M$.

\begin{prop}\label{bimodule structure}
    The map $g$ is an algebra isomorphism. Furthermore, the $H_{\rel}(Z_{\ff}^{V})$-bimodule structure on  $H_{\rel}(Z_{\ff}^{x_\alpha})$ induced by $g$ coincides with the $H_{\rel}(Z^{V}_{\ff})$-bimodule structure induced by the convolution.
\end{prop}

\begin{proof}
  Note that $Z^{V}_{\ff}\simeq\Delta V_\alpha \times Z_{\ff}^{S}\subset V_\alpha \times V_\alpha \times \widetilde{S} \times \widetilde{S} \simeq \widetilde{V} \times \widetilde{V}$. Since $g$ is an isomorphism of vector space, to show $g$ is an algebra isomorphism, it suffices to show that the following diagram (which is valid thanks to Lemma \ref{two sided ideals of hZx}) commutes: 
 $$   \begin{tikzcd}
H_{\rel}(Z^{V}_{\ff})\times H_{\rel}(Z^{V}_{\ff}) \arrow[rr, "g\times g"] \arrow[d, "*"] &  & H_{\rel}(Z^{S}_{\ff})\times H_{\rel}(Z^{S}_{\ff}) \arrow[d, "*"] \\
H_{\rel}(Z^{V}_{\ff}) \arrow[rr, "g"]                                                   &  & H_{\rel}(Z^{S}_{\ff})                                          \end{tikzcd}.$$ 
The commutativity follows from the K\"unneth formula for convolution \cite[2.7.17]{CG97} and the fact that $[\Delta V_\alpha]$ is the identity of $H_\bullet(\Delta V_\alpha)$.

To show the rest part, it suffices to show the following diagram (which is valid thanks to Lemma \ref{two sided ideals of hZx} again) commutes:
    $$\begin{tikzcd}
H_{\rel}(Z^{V}_{\ff})\times H_{\rel}(Z^{x_\alpha}_{\ff})\times H_{\rel}(Z^{V}_{\ff}) \arrow[rr, "g\times \mathrm{id}\times g"] \arrow[d, "*"] &  & H_{\rel}(Z^{S}_{\ff})\times H_{\rel}(Z^{x_\alpha}_{\ff})\times H_{\rel}(Z^{S}_{\ff}) \arrow[d, "*"] \\
H_{\rel}(Z^{x_\alpha}_{\ff}) \arrow[rr, "\mathrm{id}"]                                                                                     &  & H_{\rel}(Z^{x_\alpha}_{\ff})                                                                     \end{tikzcd}.$$
%
%where the ambient space of $Z^{x_a}_{\ff}$ corresponding to left part of above diagram is  $\widetilde{V}\times\widetilde{V}$; the ambient space of $Z^{x_a}_{\ff}$ corresponding to right part of above diagram is $\widetilde{S}\times\widetilde{S}$. 
We use 
\begin{align*}
&Z^{x_\alpha}_{\ff}\simeq\Delta x_\alpha\times Z^{x_\alpha}_{\ff}\subset V_\alpha\times V_\alpha \times \widetilde{S} \times \widetilde{S} \simeq \widetilde{V} \times \widetilde{V},\\
&Z^{V}_{\ff} \simeq\Delta V_\alpha\times Z^{S}_{\ff}\subset V_\alpha\times V_\alpha\times\widetilde{S}\times\widetilde{S}\simeq\widetilde{V}\times\widetilde{V}\qquad \mbox{and}\\
&H_{\rel}(Z^{x_\alpha}_{\ff})\simeq H(\Delta x_\alpha)\otimes H_{\rel}(Z^{x_\alpha}_{\ff})=[\Delta x_\alpha]\otimes H_{\rel}(Z^{x_\alpha}_{\ff})
\end{align*}
when do the convolution for the left part of the diagram. Then the above diagram commutes by \cite[2.7.17]{CG97} and the fact that $[\Delta V_\alpha]$ is the identity of $H_\bullet(\Delta V_\alpha)$ (here we regard $H(\Delta x_\alpha)$ as $H_\bullet(\Delta V_\alpha)$-module via convolution in the ambient space $V_\alpha \times V_\alpha$).
\end{proof}

%----------------------------------
\subsection{Some isomorphisms}
%-------------------------------------

Set $H_{\rel}(\pi_{\ff}^{-1}(x_\alpha)):=\bigoplus_{\gamma\in\Lambda_{\ff}}H_{d'_{\gamma}}(\pi_{\gamma}^{-1}(x_\alpha))$, where $d'_{\gamma}=\dim_{\mathbb{C}}\widetilde{\mathcal{N}}_\gamma-\dim_{\mathbb{C}}\mathcal{N}_\alpha$. Since $Z^{x_\alpha}_{\ff} \simeq \pi^{-1}_{\ff}(x_\alpha)\times\pi^{-1}_{\ff}(x_\alpha)$,  
by K\"unneth formula and the equation $d'_{\gamma\nu}=d'_{\gamma}+d'_{\nu}$, we have 
\begin{equation}\label{kunneth-iso}
H_{\rel}(Z_{\ff}^{x_\alpha})\simeq H_{\rel}(\pi^{-1}_{\ff}(x_\alpha))\otimes H_{\rel}(\pi^{-1}_{\ff}(x_\alpha)).
\end{equation}
We write $H_{\rel}(\pi^{-1}_{\ff}(x_\alpha))_{L}$ and $H_{\rel}(\pi^{-1}_{\ff}(x_\alpha))_{R}$ for the corresponding left and right $H(Z_\ff)$ (and also $H_{\rel}(Z^{V}_{\ff})$) modules, respectively.
The next lemma, which is a variant of \cite[3.5.1]{CG97}, follows from the definition of convolution in homology.
    
\begin{lem}\label{Ku isomorphism}
    The K\"unneth isomorphism \eqref{kunneth-iso} yields an isomorphism of $H(Z_{\ff})$ (and also $H_{\rel}(Z^{V}_{\ff})$)-bimodules $$H_{\rel}(Z_{\ff}^{x_\alpha}) \simeq H_{\rel}(\pi^{-1}_{\ff}(x_\alpha))_{L}\otimes H_{\rel}(\pi^{-1}_{\ff}(x_\alpha))_{R}.$$  
\end{lem}

The above lemma, together with \cite[2.7.46]{CG97}, gives the following corollary immediately.
\begin{cor}\label{4.10}
   The bimodule structures on $H_{\rel}(Z_{\ff}^{x_{\alpha}})$ of $H(Z_{\ff})$ and of $H_{\rel}(Z^{V}_{\ff})$ are compatible with the restriction map $\mathrm{Res}_{V}$.
\end{cor}

Just as \cite[3.5.2]{CG97}, the $H(Z_\ff)$-convolution action on $H_{\rel}(Z_{\ff}^{x_\alpha})$ commutes with the $C(\alpha)$-action. Hence $H_{\rel}(Z_{\ff}^{x_\alpha})^{C(\alpha)}$ is naturally an $H(Z_\ff)$-bimodule.

\begin{lem}\label{lem:iota}
The inclusion $\iota: \pi_{\gamma\nu}^{-1}(V_\alpha)\to G\times^{G_{x_\alpha}}Z_{\gamma\nu}^{x_\alpha}$ yields an isomorphism of vector spaces:
$$\iota^*:H(G\times^{G_{x_\alpha}}Z_{\gamma\nu}^{x_\alpha})\stackrel{\sim}\to [V_\alpha]\boxtimes H(Z_{\gamma\nu}^{x_\alpha})^{C(\alpha)}.$$
\end{lem}
\begin{proof} Denote $\mathcal{V}=Z^{x_\alpha}_{\gamma\nu}$.
     By Lemma \ref{irreducible components of Z0}, we have $H(G\times^{G_{x_\alpha}}\mathcal{V})=\bigoplus_{X}\mathbb{Q}[G\times^{G_{x_\alpha}}\mathcal{V}_X]$, where $X$ runs over the subset of a complete set of representatives of $\mathfrak{I}(\mathcal{V})/C(\alpha)$ with maximal $\dim X$. We have $$\iota^{*}([G\times^{G_{x_\alpha}}\mathcal{V}_X])=[(G\times^{G_{x_\alpha}}\mathcal{V}_X)\cap\mu^{-1}(V_\alpha)]=[V_\alpha\times\mathcal{V}_{X}]=[V_\alpha]\boxtimes[\mathcal{V}_{X}],$$ where $[\mathcal{V}_{X}]=\sum_{Y\in\mathbb{O}_{X}}[Y]$ are $C(\alpha)$-invariant elements in $H(\mathcal{V})$ and such elements form a basis of $H(\mathcal{V})^{C(\alpha)}$. Since such $[G\times^{G_{x_\alpha}}\mathcal{V}_X]$ also form a basis of $H(G\times^{G_{x_\alpha}}\mathcal{V})$, $\iota^*$ is injective. The lemma follows. 
\end{proof}

\begin{prop}\label{4.9}
    The morphism $H_{\ff,\alpha} \to H_{\rel}(Z_{\ff}^{x_\alpha})$ induced by $g'\circ i$ yields a $H(Z_{\ff})$-bimodule isomorphism $H_{\ff,\alpha}\simeq H_{\rel}(Z_{\ff}^{x_\alpha})^{C(\alpha)}$. 
\end{prop}
\begin{proof} 
By Lemmas \ref{H_O} \& \ref{lem:iota}, we have $$H_{\ff,\alpha} \simeq H_{\rel}(Z^{\mathcal{N}_\alpha}_{\ff})=H_{\rel}(G\times^{G_{x_\alpha}}Z^{x_\alpha}_{\ff})\simeq [V_\alpha]\boxtimes H_{\rel}(Z^{x_\alpha}_{\ff})^{C(\alpha)}.$$ Then the map $g'\circ i$ yields that $H_{\ff,\alpha} \simeq H_{\rel}(Z^{x_\alpha}_{\ff})^{C(\alpha)}$ as vector spaces. 

For $a\in H(Z_\ff)$, $b \in H_{\rel}(Z_{\ff}^{\leq \alpha})$, we have
     \begin{align*}
     g'( i(a\ast b))&=g'(\mathrm{Res}_{V}(a)\ast i(b))&&\text{by \cite[2.7.46 (a)]{CG97}}\\
     &={g}'(\mathrm{Res}_{V}(a))\ast {g}'({i}(b))&&\text{by Proposition}\ \ref{bimodule structure}\\
    &=\mathrm{Res}_{V}(a)\ast{g}'({i}(b))&&\text{by Proposition}\ \ref{bimodule structure}\\
     &=a\ast{g}'({i}(b))&&\text{by Corollary }\ref{4.10}.
 \end{align*}  
 Similarly, we have $g'({i}(b\ast a))=g'(i(b))\ast a$. So $g'\circ i: H_{\rel}(Z_{\ff}^{\leq \alpha})\to H_{\rel}(Z_{\ff}^{x_\alpha})$ is a $H(Z_\ff)$-bimodule morphism, which completes the proof.
\end{proof}

%---------------------------------
\subsection{Dual modules of cellular algebras}
%---------------------------------
Let $\mathbb{K}$ be a field, and $\mathfrak{A}$ be an associative $\mathbb{K}$-algebra. For any left $\mathfrak{A}$-module $M$, there is an canonical right module structure on the dual space $\Hom_{\mathbb{K}}(M,\mathbb{K})$, which we denote by $M^{\vee}$. 
  \begin{lem}
    Let $\mathfrak{A}$ be a semisimple cellular algebra over field $\mathbb{K}$ with cell datum $(\Lambda,M,C,\Psi)$. Let $M$ be a finite dimensional left $\mathfrak{A}$-module. Equip a right $\mathfrak{A}$-module structure on $M$ with the $\mathfrak{A}$-action given by $m\cdot a:=\Psi(a)m$, and denote this module by $M^{\Psi}$. It holds that $M^{\Psi}\simeq M^{\vee}$ as right $\mathfrak{A}$-modules.  
  \end{lem}
  \begin{proof}
      Since $\mathfrak{A}$ is semisimple, we assume that $M=W(\lambda)$, $\lambda\in\Lambda$, is an finite dimensional irreducible left $\mathfrak{A}$-module. Then the $\mathbb{K}$-bilinear form $\phi_{\lambda}$ in \cite[Proposition 2.4]{GL96} is non-degenerate and satisfies that $\phi_{\lambda}(\Psi(a)x,y)=\phi_{\lambda}(x,ay)$. 
      Define $f: M^{\Psi}\rightarrow M^{\vee}$ via $x\mapsto f_{x}(:y\mapsto \phi_{\lambda}(x,y))$, which is isomorphism as $\mathbb{K}$-vector space since $\phi_{\lambda}$ is non-degenerate. For $a\in \mathfrak{A}$ and $x,y\in M^{\Psi}$, $(f_{x}\cdot a)(y)=\phi_{\lambda}(x,ay)=\phi_{\lambda}(\Psi(a)x,y)=f_{\Psi(a)x}(y)=f(x\cdot a)(y)$, hence $f$ is also a right $\mathfrak{A}$-module isomorphism.
  \end{proof}
  
  By \cite[Theorem 4.2]{CLX23} and Theorem \ref{geo schur}, the Schur algebra $H(Z_\ff)\simeq\mathbb{S}_{\ff}$ is a semisimple cellular algebra with anti-involution $\Psi=-^t$, from which we immediately have the following corollary.  
  \begin{cor}\label{isomorpism of cellular module}
      As right $H(Z_\ff)$-modules, one has $$H_{\rel}(\pi^{-1}_{\ff}(x_a))^{t}\simeq H_{\rel}(\pi^{-1}_{\ff}(x_\alpha))^{\vee}.$$
  \end{cor}

%------------------
\subsection{Classification of irreducible modules}
%----------------------
We assume in the next two subsections that $G$ is an adjoint type complex simple algebraic group. Thus the component groups relating to nilpotent elements are elementary abelian $2$-groups for classical types, and $\mathfrak{S}_{n}$ ($n\le 5$) for exceptional types. In particular, all their irreducible representations can be realized over the base field $\mathbb{Q}$.
 For a finite dimensional $C(\alpha)$-representation $M$ and an irreducible character $\chi$, let $M_{\chi}:=\Hom_{C(\alpha)}(\chi,M)$ be the $\chi$-isotypical component of $M$.
  Let $\widehat{C}(\alpha)$ be the set of all those irreducible characters $\chi$ of $C(\alpha)$ such that $H_{\rel}(\pi^{-1}_{\ff}(x_\alpha))_{\chi}\neq 0$. Each $H_{\rel}(\pi^{-1}_{\ff}(x_\alpha))_{\chi}$ is naturally an $H(Z_{\ff})$-bimodule, since the action of $C(\alpha)$ commutes with the action of $H(Z_{\ff})$ on $H_{\rel}(\pi^{-1}_{\ff}(x_\alpha))$ by a similar reason to \cite[Lemma~3.5.3]{CG97}.

  Below we shall denote by $\mathbb{J}_\mathrm{rel}\subset\mathbb{J}$ the subset consisting of such $\alpha$ that $\mathcal{N}_\alpha$ is relevant for $\pi_\gamma$ for at least one orbit $\gamma\in\Lambda_\ff$. By definition, we know that $H_{\rel}(\pi^{-1}_{\ff}(x_\alpha))\neq 0$ if and only if  $\alpha\in\mathbb{J}_{\mathrm{rel}}$.
\begin{thm}
  \begin{itemize}
        \item [(1)] For any $\alpha\in\mathbb{J}_\mathrm{rel}$ and $\chi\in\widehat{C}(\alpha)$, the $\chi$-isotypical component $H_{\rel}(\pi^{-1}_{\ff}(x_\alpha))_{\chi}$ is an irreducible representation of Schur algebra. 
        \item [(2)] The $G$-orbits of pairs $(x,\chi)$ are one to one correspondence to irreducible representations of $H(Z_\ff)$, where $x\in \mathcal{N}_{\alpha}$ and $\chi\in\widehat{C}(\alpha)$ for some $\alpha\in\mathbb{J}_\mathrm{rel}$. 
        \item[(3)] The set $\{H_{\rel}(\pi_{\ff}^{-1}(x_\alpha)_{\chi})\mid \alpha\in\mathbb{J}_\mathrm{rel},\ \chi\in\widehat{C}(\alpha)\}$ forms a complete collection of isomorphism classes of irreducible representations of $H(Z_\ff)$.
    \end{itemize}
\end{thm}
\begin{proof}
We can compute straightforward that, as $H(Z_\ff)$-bimodules, 
 \begin{align*}
        H(Z_\ff)\simeq &\bigoplus_{\alpha\in\mathbb{J}}H_{\ff,\alpha}\simeq\bigoplus_{\alpha\in\mathbb{J}} H_{\rel}(Z^{x_\alpha}_{\ff})^{C(\alpha)}\qquad\text{by Proposition \ref{4.9}}\\\simeq& \bigoplus_{\alpha\in\mathbb{J}}(H_{\rel}(\pi^{-1}_{\ff}(x_\alpha))_{L}\otimes H_{\rel}(\pi^{-1}_{\ff}(x_\alpha))_{R})^{C(\alpha)}\qquad\text{by Lemma \ref{Ku isomorphism}}\\
        \simeq&\bigoplus_{\alpha\in\mathbb{J}}(H_{\rel}(\pi^{-1}_{\ff}(x_\alpha))_{L}\otimes H_{\rel}(\pi^{-1}_{\ff}(x_\alpha))_{L}^{t})^{C(\alpha)}\\
        \simeq&\bigoplus_{\alpha\in\mathbb{J}}(H_{\rel}(\pi^{-1}_{\ff}(x_\alpha))_{L}\otimes H_{\rel}(\pi^{-1}_{\ff}(x_\alpha))_{L}^{\vee})^{C(\alpha)}\qquad\text{by Corollary \ref{isomorpism of cellular module}}\\
        \simeq&\bigoplus_{\alpha\in\mathbb{J}}\End_{\mathbb{Q}}(H_{\rel}(\pi^{-1}_{\ff}(x_\alpha)))^{C(\alpha)}\simeq\bigoplus_{\alpha\in\mathbb{J}}\End_{C(\alpha)}(H_{\rel}(\pi^{-1}_{\ff}(x_\alpha)))\\
        \simeq&\bigoplus_{\alpha\in\mathbb{J}}\End_{C(\alpha)}(\bigoplus_{\chi\in \widehat{C}(\alpha)}\chi\otimes H_{\rel}(\pi^{-1}_{\ff}(x_\alpha))_{\chi})\\    \simeq&\bigoplus_{\alpha\in\mathbb{J}, \chi,\psi\in\widehat{C}(\alpha)}\Hom_{C(\alpha)}(\chi,\psi)\otimes\Hom_{\mathbb{Q}}(H_{\rel}(\pi^{-1}_{\ff}(x_\alpha))_{\chi},H_{\rel}(\pi^{-1}_{\ff}(x_\alpha))_{\psi})\\
        \simeq&\bigoplus_{\alpha\in\mathbb{J}_\mathrm{rel}, \chi\in\widehat{C}(\alpha)}\End_{\mathbb{Q}}(H_{\rel}(\pi_{\ff}^{-1}(x_\alpha)_{\chi})).
    \end{align*} 
    Thus $\{H_{\rel}(\pi_{\ff}^{-1}(x_\alpha)_{\chi})\mid \alpha\in\mathbb{J}_\mathrm{rel},\ \chi\in\widehat{C}(\alpha)\}$ forms a complete set of irreducible representations of $H(Z_\ff)$. Finally, it is clear that $H_{\rel}(\pi^{-1}_{\ff}(x_\alpha))_{\chi}\simeq H_{\rel}(\pi^{-1}_{\ff}(x_{\alpha'}))_{\chi'}$ if and only if $(x_\alpha,\chi)$ and $(x_{\alpha'},\chi')$ are $G$-conjugacy, which completes the proof.
\end{proof}

\subsection{Classification via perverse sheaves}
    Since $G$ is connected, a $G$-equivariant local system over $\mathcal{N}_\alpha$ is irreducible if and only if the stalk at an arbitrary point is an irreducible $C(\alpha)$-representation. Therefore, an irreducible representation $\chi$ of $C(\alpha)$ gives an irreducible local system over $\mathcal{N}_\alpha$, which we denote by $\mathcal{L}_\chi$. 
    
     Let $D^b(\widetilde{\mathcal{N}}_\ff)$ (resp. $D^{b}(\mathcal{N})$) be the bounded derived category of complexes of sheaves with constructible cohomology sheaves on $\widetilde{\mathcal{N}}_\ff$ (resp. $\mathcal{N}$).   
    Set $\mathcal{C}_\ff :=\bigoplus_{\gamma\in\Lambda_\ff} \underline{\mathbb{C}}_{\widetilde{\mathcal{N}}_\gamma}[d_\gamma]\in D^b(\mathcal{N})$, where $d_\gamma=\dim_\C\widetilde{\mathcal{N}}_\gamma$ and $\underline{\mathbb{C}}_{\widetilde{\mathcal{N}}_\gamma}$ is the constant sheaf on $\widetilde{\mathcal{N}}_{\gamma}$.

     Note that the map $\pi_\ff: \widetilde{\mathcal{N}}_\ff \to \mathcal{N}, (gP_\gamma, x)\mapsto x$, is a $G$-equivariant projective morphism and $\mathcal{N}$ consists of finitely many $G$-orbits, by which we have the following theorem (cf. \cite[Theorem 8.4.12]{CG97} and \cite[Theorem 5.4]{G98}). Here we write $\C$ in the notations of homology to indicate that the coefficient field $\Q$ is replaced by $\C$. 
    \begin{thm}
    There is a direct sum decomposition
    $$R\pi_{\ff*}(\mathcal{C}_\ff)= \bigoplus_{\alpha\in\mathbb{J},\ \chi \in \widehat{C}(\alpha)} H_{\rel}(\pi_\ff^{-1}(x_\alpha),\C)_\chi \otimes IC_\mathcal{N}(\mathcal{L}_\chi),$$
    where $H_{\rel}(\pi_\ff^{-1}(x_\alpha),\mathbb{C})_\chi$ is the $\chi$-isotypical components of $H_{\rel} (\pi_\ff^{-1}(x_\alpha),\mathbb{C})$.
    \end{thm}

It is a standard result (cf. \cite[Lemma~8.6.1]{CG97} or \cite[Theorem 5.4]{G98}) that there is an algebra isomorphism $$H(Z_\ff,\mathbb{C}) \simeq \Hom_{D^b(\mathcal{N})}(R\pi_{\ff*}(\mathcal{C}_\ff),R\pi_{\ff*}(\mathcal{C}_\ff)).$$   
Thus the above theorem gives us the following corollary.
    \begin{cor}\label{rep}
   There is an algebra isomorphism
    \begin{align*}
        H(Z_\ff,\mathbb{C}) \simeq\bigoplus_{\alpha\in\mathbb{J},\ \chi \in \widehat{C}(\alpha)}\End_\mathbb{C} (H_{\rel}(\pi_\ff^{-1}(x_\alpha),\mathbb{C})_\chi).
    \end{align*}
        Moreover, $\{H_{\rel}(\pi_\ff^{-1}(x_\alpha),\mathbb{C})_\chi ~|~ \alpha\in\mathbb{J}_{\mathrm{rel}}, \chi \in \widehat{C}(\alpha)\}$ is a complete set of isomorphism classes of irreducible $H(Z_\ff,\mathbb{C})$-representations. 
    \end{cor}

    \begin{rem}
        Corollary \ref{rep} also implies that $H(\pi_\ff^{-1}(x_\alpha),\mathbb{C})$ has a decomposition $\bigoplus_\chi H_{\rel}(\pi_\ff^{-1}(x_\alpha),\mathbb{C})_\chi \otimes \chi$ as a $(H(Z_\ff, \mathbb{C}), C(\alpha))$-bimodule. However, $H(Z_\ff) \simeq \Sc_\ff$ is split semisimple, and the action of $(\mathbb{C} \otimes_\Q H(Z_\ff), C(\alpha))$ is a scalar extension of the one of $(H(Z_\ff), C(\alpha))$ on $H_{\rel}(\pi_\ff^{-1}(x_\alpha))$. Thus the decomposition is still valid for $H_{\rel}(\pi_\ff^{-1}(x_\alpha))$, which implies that $H_{\rel}(\pi_\ff^{-1}(x_\alpha),\mathbb{C})_\chi$ can be defined over $\Q$.
    \end{rem}
%========================================================

 %===================================   
\section{Equivariant K-theoretic construction}
%--------------------------------
\subsection{Equivariant K-groups}\label{sec:Kgroup}
%------------------------------
Let $G$ be a complex algebraic group and $X$ a quasi-projective $G$-variety, the $0$-th equivariant K-group $K^{G}(X)$ of $X$ is defined as the Grothendieck group of the category $\mathtt{Coh}^G(X)$ of $G$-equivariant coherent sheaves on $X$. In particular, $K^G(pt)=R(G)$ is the representation ring of $G$, and $K^G(X)$ is an $R(G)$-module canonically. We shall denote $[\mathcal{F}]\in K^{G}(X)$ for $\mathcal{F}\in\mathtt{Coh}^{G}(X)$. 
%Particularly, the structure sheaf $\mathcal{O}_X$ of regular functions on $X$ has a canonical $G$-equivariant structure (cf. \cite[Remark~5.1.7]{CG97}), and hence we have an element 

Given two G-varieties $X$ and $Y$, we refer \cite[\S 5.2.11]{CG97} for the external tensor product $\boxtimes$ on equivariant K-groups
\begin{align*}
\boxtimes: K^G(X)\otimes_{R(G)} K^G(Y)\to K^G(X\times Y).
\end{align*}
Let $f: X \to Y$ be a $G$-morphism. We will define the pull-back between equivariant K-groups when the functor $f^*$ between equivariant coherent sheaves is homologically finite, and the pushforward between equivariant K-groups when $f$ is proper. Moreover, we just denote them by $f_*$ or $f^*$.

If there exists a $G$-morphism $f: X \to M$ with $M$ smooth, then $K^G(M)$ is an $R(G)$-algebra and $K^G(X)$ is a $K^G(M)$-module by $[\mathcal{F}]\cdot[\mathcal{F}']=\sum_i (-1)^i[f^*\mathcal{E}^i \otimes_{\mathcal{O}_{X}} \mathcal{F}']$ for $\mathcal{F} \in \mathtt{Coh}^G(M)$ and $\mathcal{F}' \in \mathtt{Coh}^G(X)$ with $\mathcal{E}^\bullet \to \mathcal{F}$ a finite $G$ locally free resolution of $\mathcal{F}$. If there exists another $G$-morphism $g: Y \to M$, then the external tensor product factors through $K^G(X)\otimes_{K^G(M)} K^G(Y)$. If there exists a $G$-morphism $\varphi: X \to Y$ such that $f=g \circ \varphi$ and $\varphi$ is proper, then $\varphi_*: K^G(X) \to K^G(Y)$ is a $K^G(M)$-homomorphism by the projection formula.
%\redtext{For a closed $G$-subvariety $Y$ of $X$ and an $\mathcal{M} \in K^G(Y)$, without confusion we can also regard $\mathcal{M}$ as its extension by zero in $X$.} 

For any closed $G$-subvariety $Y$ of $X$, each element in $K^G(Y)$ can be extended trivially to a one in $K^G(X)$. By abuse of notations, we shall regard each element in $K^G(Y)$ as its extension in $K^G(X)$, though $K^G(Y)$ can not be embedded in $K^G(X)$ in general. Particularly, let $\mathcal{O}_Y$ be the structure sheaf of the regular functions on $Y$, which admits a canonical $G$-equivariant structure (cf. \cite[Remark~5.1.7]{CG97}). The class $[\mathcal{O}_Y]\in K^G(Y)$ is also regarded as an element in $K^G(X)$.

Let $H$ be a closed algebraic subgroup of $G$. For a virtual character $\chi \in R(H)$, we define the sheaf $\mathcal{O}_{G/H}(\chi)$ over $G/H$ by the sheaf of rational sections of $G \times^H \V_\chi$, where $\V_\chi$ is the $H$-module associated with $\chi$. The Observation on \cite[page 233]{CG97} implies that $\mathcal{O}_{G/H}(\chi)$ is $G$-equivariant.
Thus we have an element $[\mathcal{O}_{G/H}(\chi)]$ in $K^G(G/H)$. Since $K^G(G/H) \simeq R(H)$, any element in $K^G(G/H)$ is of the form $[\mathcal{O}_{G/H}(\chi)]$ for some $\chi \in R(H)$. For a $G$-equivariant vector bundle $E \to G/H$, there is an isomorphism $K^G(E) \simeq K^G(G/H)$. We shall write $\mathcal{O}_E(\chi)$ and $[\mathcal{O}_E(\chi)]$ for the pull-back of $\mathcal{O}_{G/H}(\chi)$ and $[\mathcal{O}_{G/H}(\chi)]$, respectively.
%------------------------------------
\subsection{Convolution in equivariant K-theory}
%------------------------------------
Similar to what we did in \S \ref{subsec:conv}, let us introduce the convolution for equivariant K-groups in the case of locally closed varieties (instead of the case of closed varieties introduced in \cite{CG97}).
Firstly, we define the tensor product with the support between locally closed subsets as follows. Let $Y$ and $Y'$ be two locally closed $G$-subvarieties of a connected smooth $G$-variety $M$. Let $V$ be an arbitrary open $G$-subvariety such that $Y \cap V, Y' \cap V$ are closed in $V$ and $Y \cap Y' \subset V$. Then we define:
\begin{align*}
    \otimes : K^G(Y) \otimes_{K^G(M)} K^G(Y') \rightarrow K^G(Y \cap V) \otimes_{K^G(V)} K^G(Y' \cap V) \xrightarrow{\Delta_{V}^* \circ \boxtimes} K^G(Y \cap Y'),
\end{align*}
which is independent of the choice of $V$.

Take three connected smooth complex $G$-varieties $M_1$, $M_2$, $M_3$. Let
$$Z_{12} \in M_1 \times M_2, \quad Z_{23} \in M_2 \times M_3$$
be $G$-subvarieties.
Assume that $Z_{12}$, $Z_{23}$ and $Z_{12} \circ Z_{23}$ are locally closed $G$-subvariety and the map $Z_{12} \times_{M_2} Z_{23} \to Z_{12} \circ Z_{23}$ is proper. We define the convolution
\begin{align*}
    \star : K^G(Z_{12}) \otimes_{K^G(M_2)} K^G(Z_{23}) \rightarrow K^G(Z_{12} \circ Z_{23})
\end{align*}
in the same way as that in \cite[\S 5.2.20]{CG97}.

Similar to Lemma~\ref{con dia}, we have the following result. 
\begin{lem}\label{Kcon dia}
    Let $Z_{12}'$ and $Z_{23}'$ be open $G$-subvarieies of $Z_{12}$ and $Z_{23}$ respectively such that $Z_{12}' \circ Z_{23}'$ is also open in $Z_{12} \circ Z_{23}$ and we have the Cartesian diagram
    $$\begin{tikzcd}
        Z_{12}' \times_{M_2} Z_{23}' \ar[r] \ar[d] & Z_{12}' \circ Z_{23}' \ar[d] \\
        Z_{12} \times_{M_2} Z_{23} \ar[r] & Z_{12} \circ Z_{23}
    \end{tikzcd}.$$
    Let $V$ be an open $G$-subvariety of $Z_{12}' \circ Z_{23}'$. Then we have a commutative diagram
    $$\begin{tikzcd}
        K^G(Z_{12}) \otimes_{K^G(M_2)} K^G(Z_{23}) \ar[r,"\star"] \ar[d] & K^G(Z_{12} \circ Z_{23}) \ar[d] \ar[dr] &\\
        K^G(Z_{12}') \otimes_{K^G(M_2)} K^G(Z_{23}') \ar[r,"\star"] & K^G(Z_{12}' \circ Z_{23}') \ar[r] & K^G(V)
    \end{tikzcd}.$$
\end{lem}
The following lemma is a $K$-theory version of \cite[Lemma~2.7.17]{CG97}. Here the notations are counterparts of those in \cite[\S2.7.16]{CG97} (with all the manifolds replacing by $G$-varieties), and we assume further that all the product varieties here satisfy the conditions stated in \cite[Theorem~5.6.1]{CG97}.

\begin{lem}\label{Ku-cov}
    The following diagram formed by convolution maps commutes:
$$\begin{tikzcd}
     K^G(Z_{12}\times\widetilde{Z}_{12})\otimes K^{G}(Z_{23}\times\widetilde{Z}_{23}) \arrow[d,equal, "{\text{K\"unneth}}"] \arrow[rd, "\star"] &\\
K^G(Z_{12})\otimes K^G(\widetilde{Z}_{12})\otimes K^G(Z_{23})\otimes K^G(\widetilde{Z}_{23}) \arrow[d, "\star"] & K^G((Z_{12}\times\widetilde{Z}_{12})\circ(Z_{23}\times\widetilde{Z}_{23})) \arrow[d,equal] \\
K^G(Z_{12}\circ Z_{23})\otimes K^G(\widetilde{Z}_{12}\circ\widetilde{Z}_{23}) \arrow[r,equal, "{\text{K\"unneth}}"]                        & K^G((Z_{12}\circ Z_{23})\times(\widetilde{Z}_{12}\circ\widetilde{Z}_{23}))    
\end{tikzcd}$$
\end{lem}
 \begin{proof}
 The statement is derived by a standard diagram chasing.
 \end{proof}

%----------------------------------
\subsection{Cellular fibration}
%----------------------------------
Let $\pi: F\to X$ be a $G$-equivariant morphism of $G$-varieties. We call $F$ a cellular fibration over $X$ if there is a finite decreasing filtration $F=F^n\supset F^{n-1}\supset\cdots\supset F^{1}\supset F^{0}=\emptyset$ satisfying the following conditions: for any $i=1,2,\ldots,n$,
\begin{itemize}
    \item[(1)] $F^i$ is a $G$-stable closed algebraic subvarieties; furthermore, the restriction $\pi:F^i\to X$ is a $G$-equivariant locally trivial fibration.
    \item[(2)] The map $\pi_i:F^i\backslash F^{i-1}\to X$, the restriction of $\pi$, is a $G$-equivariant affine fibration.  
\end{itemize}
Let $\bar{\pi}_i: \overline{F^i\backslash F^{i-1}}\longrightarrow X$ be the restriction of $\pi$ on the closure of $F^i\backslash F^{i-1}$. The following known result is referred to as the cellular fibration lemma (cf. \cite[Lemma~5.5.1]{CG97}).
\begin{lem}\label{lem. cfl}
    Let $F$ be a cellular fibration over $X$. The following statements hold.
    \begin{itemize}
        \item[(a)] For each $i=1,\cdots,n$ there is a canonical short exact sequence $$0\longrightarrow K^{G}(F^{i-1})\longrightarrow K^{G}(F^i)\longrightarrow K^{G}(F^{i}\backslash F^{i-1})\longrightarrow0$$ 
        \item[(b1)] If $X$ is smooth, then the exact sequences in (a) are split. Moreover, $K^{G}(F)$ is a free $K^G(X)$-module with a basis $$\{ [\mathcal{O}_{\overline{F^i\backslash F^{i-1}}}] ~|~ i=1,\cdots,n\};$$
        \item[(b2)] If $K^G(X)$ is a free $R(G)$-module with the basis $\{[\mathcal{F}_j] ~|~ 1 \leq j \leq m\}$, then the exact sequences in (a) are split. Moreover, $K^G(F)$ is a free $R(G)$-module with a basis $$\{[\bar{\pi}_i^*\mathcal{F}_j] ~|~ 1 \leq i \leq n, 1\leq j \leq m\}.$$
        \item[(c)] Let $H\subset G$ be a closed algebraic subgroup and suppose that the assumption of (b) holds for both $G$ and $H$. If the natural map $$R(H)\otimes_{R(G)}K^{G}(X) \to K^{H}(X)$$ is an isomorphism, then so is the map $$R(H)\otimes_{R(G)}K^{G}(F)\to K^{H}(F).$$
    \end{itemize}
\end{lem}

%---------------------
\subsection{The group $\breve{G}$}\label{sec:Aaction}
%------------------------
From now on, we assume that $G$ is 
\begin{itemize}
\item connected and reductive; and satisfies that
\item each simple factor of its derived subgroup is simply connected or of type $\mathrm{SO}_{2d+1}$.
\end{itemize}

There are natural ring isomorphisms for representation rings about $G$ and its maximal torus $T$ (see \cite[\S6.1]{CG97}):
$$R(T)=\Z[Q] \quad\mbox{and}\quad R(G)=R(T)^\W.$$ Here the second isomorphism is a ``group'' analogue of the Chevalley restriction theorem. In fact, it can be further generalized to the following result, obtained in \cite[Theorems 1.2 \& 2.2]{St75}.
\begin{lem}\label{lem:PW}
Let $P$ be a parabolic subgroup of $G$ containing $T$, and let $\W'$ be the Weyl group of the Levi part of $\mathrm{Lie}(P)$. Then $R(P)\simeq R(T)^{\W'}$ is a free $R(G)$-module with rank $\W/\W'$.
\end{lem}

For an arbitrary algebraic group $H$, denote $$\breve{H}=H \times (\mathbb{C}\backslash\{0\}).$$
It is clear that $R(\breve{H}) \simeq R(H)[q,q^{-1}]$, where $q$ means the tautological representation of $\mathbb{C}\backslash \{0\}$.

Let $\breve{G}$ act on $\sF_\gamma \times \mathfrak{g}$ by
$$(s,z)\cdot( gP_\gamma, x)=(sgP_\gamma,z^2\Ad_s(x)),\quad \forall (s,z)\in \breve{G}, (gP_\gamma,x)\in \sF_\gamma \times \mathfrak{g}.$$
Then $\widetilde{\mathcal{N}}_\gamma$ is a $\breve{G}$-subvariety of $\sF_\gamma \times \mathfrak{g}$. Let $\breve{G}$ act on $\widetilde{\mathcal{N}}_{\gamma}\times\widetilde{\mathcal{N}}_{\nu}$ diagonally for $\gamma,\nu\in\Lambda_\ff$. Then $Z_{\gamma\nu}$ (resp. $Z_\ff$) is a $\breve{G}$-subvariety of $\widetilde{\mathcal{N}}_{\gamma}\times\widetilde{\mathcal{N}}_{\nu}$ (resp. $\widetilde{\mathcal{N}}_{\ff}\times\widetilde{\mathcal{N}}_{\ff}$). Thus we get equivariant K-groups $K^{\breve{G}}(Z_{\gamma\nu})$ and $K^{\breve{G}}(Z_\ff)$. Particularly, $(K^{\breve{G}}(Z_\ff),\star)$ forms an associative $\Z[q,q^{-1}]$-algebra.

%-------------------
\subsection{Filtrations}
%-------------------
Fix a total order ``$\preceq$" on $\W$ compatible with the Bruhat order ``$\le$".
Define $$Z^{\preceq w}_{\gamma\nu}:=\bigsqcup\limits_{y\preceq w,y\in\mathcal{D}_{\gamma\nu}}T^*_{\Ob_{\gamma,y,\nu}},\quad Z^{\prec w}_{\gamma\nu}:=\bigsqcup\limits_{y\prec w,y\in\mathcal{D}_{\gamma\nu}}T^*_{\Ob_{\gamma,y,\nu}}.$$
Then we have a closed subset filtration
\begin{equation}\label{filtra:Z}
Z_{\gamma\nu} \supset \cdots\supset Z_{\gamma\nu}^{\preceq w} \supset Z_{\gamma\nu}^{\prec w} \supset\cdots \supset Z_{\gamma\nu}^{\preceq\id}= T_{\gamma,\id,\nu}^* \supset \emptyset.
\end{equation}

\begin{lem}\label{ex seq1}
Let $Y,Y'$ be two locally closed subset of $Z_{\gamma\nu}$ consisting of several conormal bundles $T^*_{\Ob_{\xi}}$ with $Y'$ closed in $Y$.  There exists an exact sequence
\begin{equation*}%\label{exactseq1} 
0 \longrightarrow K^{\breve{G}}(Y') \longrightarrow K^{\breve{G}}(Y) \longrightarrow K^{\breve{G}}(Y\backslash Y') \longrightarrow 0,
\end{equation*}
\end{lem}
\begin{proof} 
By equivariant K-theory, one has an exact sequence $$K^{\breve{G}}(Y') \longrightarrow K^{\breve{G}}(Y) \longrightarrow K^{\breve{G}}(Y\backslash Y')\longrightarrow 0.$$ It remains to show the left exactness.  
    Each $T_{\Ob_{\xi}}^*$ can be decomposed into a union of $\breve{T}$-affine spaces. Hence by the cellular fibration theorem, we have
    $$0\longrightarrow K^{\breve{T}}(Y')\longrightarrow K^{\breve{T}}(Y)\longrightarrow K^{\breve{T}}(Y\backslash Y')\longrightarrow0.$$
    Applying \cite[6.1.22(b)]{CG97}, there comes the following commutative diagram: \begin{equation*}%\label{exsq. 3.0}
        \begin{tikzcd}
K^{\breve{G}}(Y') \arrow[d] \arrow[r, "\simeq"] & K^{\breve{T}}(Y')^\W \arrow[d,hook]\\
K^{\breve{G}}(Y) \arrow[r, "\simeq"] & K^{\breve{T}}(Y)^\W,
\end{tikzcd}
    \end{equation*}
    from which the left exactness follows.
\end{proof}

Lemma~\ref{ex seq1} brings us a filtration 
\begin{equation}\label{filtra}
K^{\breve{G}}(Z_{\gamma\nu})\supset\cdots\supset K^{\breve{G}}(Z_{\gamma\nu}^{\preceq w})\supset K^{\breve{G}}(Z_{\gamma\nu}^{\prec w}) \supset\cdots \supset K^{\breve{G}}(T_{\gamma,\id,\nu}^*) \supset \{0\},
\end{equation}
in which each quotient is in the form of $K^{\breve{G}}(T_{\Ob_{\gamma,w,\nu}}^*)$, $w\in\D_{\gamma\nu}$.

\begin{cor}\label{basis}
    For $\gamma,\nu\in\Lambda_\ff$, the equivariant K-group $K^{\breve{G}}(Z_{\gamma\nu})$ is a free $R(\breve{G})$-module with rank $\# \D_\gamma\# \D_\nu$.
\end{cor}
\begin{proof}
Thanks to Lemma~\ref{lem:PW}, $K^{\breve{G}}(T_{\Ob_{\gamma,w,\nu}}^*) \simeq K^{\breve{G}}(\Ob_{\gamma,w,\nu}) \simeq R(\breve{P}_{\gamma\nu}^w)\simeq R(\breve{T})^{\W_{\gamma\nu}^w}$ is free over $R(\breve{G})$ with rank $\# \D_\gamma \# (\W_\gamma w \W_\nu/\W_\nu)$.

The filtration \eqref{filtra} implies that $K^{\breve{G}}(Z_{\gamma\nu})$ is a free $R(\breve{G})$-module. Therefore,
    \begin{align*} &\mathrm{rank}_{R(\breve{G})}K^{\breve{G}}(Z_{\gamma\nu})=\mathrm{rank}_{R(\breve{G})}(\bigoplus_{w \in \mathcal{D}_{\gamma\nu}}K^{\breve{G}}(T^{*}_{\Ob_{\gamma,w,\nu}}))=\mathrm{rank}_{R(\breve{G})}(\bigoplus_{w \in \mathcal{D}_{\gamma\nu}}K^{\breve{G}}(\Ob_{\gamma,w,\nu}))\\&=\mathrm{rank}_{R(\breve{G})}(\bigoplus_{w \in \mathcal{D}_{\gamma\nu}}K^{\breve{G}}(G/P^{w}_{\gamma\nu}))=\mathrm{rank}_{R(\breve{G})}(\bigoplus_{w \in \D_{\gamma\nu}} R(\breve{T})^{\W_{\gamma\nu}^w})\\
        &=\sum_{w \in \D_{\gamma\nu}} \# \D_\gamma \# (\W_\gamma w \W_\nu/\W_\nu)=\# \D_\gamma\# \D_\nu.
    \end{align*}
\end{proof}

%----------------------------------
\subsection{$R(\breve{T})$-modules}
%----------------------------------
Firstly, let us deal with the case for which the cellular fibration lemma works.
Recall $\mu_0\in\Lambda_\ff$ is a regular $\W$-orbit. For any $\gamma\in\Lambda$, the $G$-variety $Z_{\gamma\mu_0}$ with the finite filtration 
\eqref{filtra:Z} (taking $\nu=\mu_0$ therein)
admits a cellular fibration over $\sB$ via the canonical morphism 
\begin{equation}\label{eq:cellularpi}
\pi: Z_{\gamma\mu_0}=\bigsqcup_{w\in\mathcal{D}_\gamma}T^*_{\Ob_{\gamma,w}} \to\bigsqcup_{w\in\mathcal{D}_\gamma}\Ob_{\gamma,w} \to \sB,
\end{equation}
whose restriction $Z_{\gamma\mu_0}^{\preceq w}\backslash Z_{\gamma\mu_0}^{\prec w}=T^*_{\Ob_{\gamma,w}}\stackrel{p_1}\longrightarrow\Ob_{\gamma,w}\stackrel{p_2}\longrightarrow\sB$ is a $\breve{G}$-equivariant affine fibration since the natural map $p_1$ is an affine fibration and the projection $p_2$ is a fibration with fiber $BwP_{\gamma}/P_{\gamma}$ which is an affine space by Bruhat decomposition.

Noting $K^{\breve{G}}(\sB)\simeq R(\breve{T})$ (cf. \cite[p313]{CG97}) and applying Lemma~\ref{lem. cfl}, we have: 
\begin{cor}\label{cor:basisZr}
    For $\gamma\in\Lambda_\ff$, the equivariant K-group $K^{\breve{G}}(Z_{\gamma\mu_0})$ is a free $R(\breve{T})$-module with rank $\# \D_\gamma$ and an $R(\breve{T})$-basis $\{[\mathcal{O}_{T^{*}_{\gamma,w}}]\mid w \in \D_\gamma\}$.
\end{cor}

Similar statement is also valid for $Z_{\mu_0\nu}$, ($\forall\nu\in\Lambda$). If we
specialize $\gamma$ to be also a regular $\W$-orbit, then the above corollary recovers a known result about equivariant K-theory of the Steinberg variety $Z$ as follows.
\begin{cor}\label{basis:H}
Regard $K^{\breve{G}}(Z)$ as an $R(\breve{T})$-module by considering the projection on the second factor $Z \to \sB$. The equivariant K-group $K^{\breve{G}}(Z)$ is a free $R(\breve{T})$-module with rank $\#\W$ and an $R(\breve{T})$-basis $\{[\mathcal{O}_{T^*_w}]\mid w\in\W\}$.
\end{cor}

For each weight $\lambda\in \breve{T}^*$, let $e^\lambda\in R(\breve{T})\simeq K^{\breve{G}}(\sB)$ be the virtual character of $\breve{T}$ associated with $\lambda$. The $R(\breve{T})$-action on $K^{\breve{G}}(Z_{\gamma\mu_0})$ and $K^{\breve{G}}(Z)$, mentioned in the above two corollaries, is given by $e^\lambda\cdot-=\pi^*(e^\lambda)\otimes^{\mathbb{L}}-$, where $\pi$ is the cellular fibration \eqref{eq:cellularpi}, and $\otimes^{\mathbb{L}}$ is the derived tensor product.
We write 
\begin{equation}\label{notation:elambda}
e_{\gamma,w}^\lambda:=e^\lambda \cdot [\mathcal{O}_{T^*_{\gamma,w}}]\in K^{\breve{G}}(Z_{\gamma{\mu_0}}), \quad\mbox{(particularly, $e^\lambda_{w}:=e^\lambda\cdot[\mathcal{O}_{T^*_{w}}]\in K^{\breve{G}}(Z)$)}.
\end{equation}
We would like to emphasize that here $K^{\breve{G}}(Z)$ is regarded as an $R(\breve{T})$-module by considering the projection on the second factor $Z \to \sB$, though we write left action $e^\lambda\cdot-$ for convenience.

Let $\Delta \widetilde{\mathcal{N}}\subset \widetilde{\mathcal{N}}\times\widetilde{\mathcal{N}}$ be the diagonal copy of $\widetilde{\mathcal{N}}$. It is known (cf. \cite[7.2.4]{CG97}) that $K^{\breve{G}}(\Delta \widetilde{\mathcal{N}})\simeq R(\breve{T})$. Thus there is another (right) $R(\breve{T})$-action on $K^{\breve{G}}(\widetilde{\mathcal{N}}_\ff\times_{\mathcal{N}} \widetilde{\mathcal{N}})$ given by the convolution product  
$$\star:K^{\breve{G}}(\widetilde{\mathcal{N}}_\ff\times_{\mathcal{N}} \widetilde{\mathcal{N}})\times K^{\breve{G}}(\Delta \widetilde{\mathcal{N}})\longrightarrow K^{\breve{G}}(\widetilde{\mathcal{N}}_\ff\times_{\mathcal{N}} \widetilde{\mathcal{N}}).$$ 
We know $K^{\breve{G}}(\Delta\widetilde{\mathcal{N}})\subset K^{\breve{G}}(Z)$ is a free $R(\breve{T})$-module by the cellular fibration lemma for almost a trivial case $\Delta\widetilde{\mathcal{N}}\simeq \widetilde{\mathcal{N}}=T^*\sB\to\sB$, which implies $[\mathcal{O}_{T^*_{\id}}]\in K^{\breve{G}}(\Delta\widetilde{\mathcal{N}})$ and hence  $e^\lambda_{\id}=e^\lambda\cdot[\mathcal{O}_{T^*_{\id}}]\in K^{\breve{G}}(\Delta\widetilde{\mathcal{N}})$.  

\begin{lem}
The two actions $\cdot$ and $\star$ coincide under the identification $K^{\breve{G}}(\Delta \widetilde{\mathcal{N}})\simeq R(\breve{T})$ given by $e^{\lambda}_{\id}\leftrightarrow e^{\lambda}$, namely, 
\begin{equation}
[\mathcal{O}_{T^*_{\gamma,w}}]\star e^{\lambda}_{\id}=e^\lambda\cdot[\mathcal{O}_{T^*_{\gamma,w}}]=e^\lambda_{\gamma,w}.
\end{equation}\label{eq:stardot}
\end{lem}
\begin{proof}
It is a straightforward computation that 
\begin{align*}
    [\mathcal{O}_{T^*_{\gamma,w}}]\star e^{\lambda}_{\id}&=pr_{13*}(([\mathcal{O}_{T^*_{\gamma,w}}]\boxtimes[\mathcal{O}_{T^*\sB}])\otimes^{\mathbb{L}}([\mathcal{O}_{T^*\sF_{\gamma}}]\boxtimes[\mathcal{O}_{\Delta\widetilde{\mathcal{N}}}(\lambda)]))\\&=pr_{13*}(([\mathcal{O}_{T^*_{\gamma,w}}]\boxtimes[\mathcal{O}_{T^*\sB}])\otimes^{\mathbb{L}}(([\mathcal{O}_{T^*\sF_{\gamma}}]\boxtimes[\mathcal{O}_{\Delta\widetilde{\mathcal{N}}}])\otimes^{\mathbb{L}}pr_{2}^*(e^{\lambda}))\\&=pr_{13*}(((\pi^*(e^\lambda)\otimes^{\mathbb{L}}[\mathcal{O}_{T^*_{\gamma,w}}])\boxtimes[\mathcal{O}_{T^*\sB}])\otimes^{\mathbb{L}}([\mathcal{O}_{T^*\sF_{\gamma}}]\boxtimes[\mathcal{O}_{\Delta\widetilde{\mathcal{N}}}]))\\&=(e^{\lambda}\cdot[\mathcal{O}_{T^*_{\gamma,w}}])\star e^0_{\id}=e^{\lambda}\cdot[\mathcal{O}_{T^*_{\gamma,w}}]=e^\lambda_{\gamma,w}.
\end{align*}
\end{proof}

We would like to caution that if neither $\gamma$ nor $\nu$ is a regular $\W$-orbit then there exists no $\breve{G}$-cellular fibration for $Z_{\gamma\nu}$ (see Example~\ref{ex:nof} below). Hence the cellular fibration lemma (i.e. Lemma~\ref{lem. cfl}) is not applicable to $K^{\breve{G}}(Z_{\gamma\nu})$ to obtain its freeness and basis. We shall provide Proposition~\ref{prop:noncbasis} below, served as an analogue of the cellular fibration lemma for $Z_{\gamma\nu}$. 
\begin{example}\label{ex:nof}
    Let $G=\mathrm{GL}_{n}(\C),n\ge3$, and let $\gamma$ be a $\W$-orbit such that $P_{\gamma}=B\{1,s_1\}B$. For $w\in\mathcal{D}_{\gamma\gamma}$, consider $p:T^{*}_{\Ob_{\gamma,w,\gamma}}\to\Ob_{\gamma,w,\gamma}\to\sF_{\gamma}$. The second arrow is a fibration with fiber isomorphic to $P_{\gamma}wP_{\gamma}/P_{\gamma}=B\{w,s_{1}w\}P_{\gamma}/P_{\gamma}$, which is a disjoint union of two cells if $s_1w\neq ws_1$ (for example $w=s_2$). In this case, the variety $Z_{\gamma\gamma}$ has no $\breve{G}$-cellular fibration.
\end{example}
%

%-------------------------
\subsection{Basis of $K^{\breve{G}}(Z_{\gamma\nu})$}
%-------------------------
Though we have obtain an $R(\breve{T})$-basis of $K^{\breve{G}}(Z_{\gamma{\mu_0}})$ by the cellular fibration lemma in the previous section. As shown in Example~\ref{ex:nof}, it no longer works for $K^{\breve{G}}(Z_{\gamma\nu})$ with general $\gamma, \nu\in\Lambda$. In this section, we shall try to give an $R(\breve{G})$-basis of $K^{\breve{G}}(Z_{\gamma\nu})$ without using the cellular fibration lemma.

We have a commutative diagram: (for $\gamma,\nu\in\Lambda$ and $w\in\mathcal{D}_{\gamma\nu}$) $$\begin{tikzcd}
K^{\breve{G}}(T^*_{\gamma,w,\nu}) \arrow[r] \arrow[d] & K^{\breve{G}}(T^*_{\Ob_{\gamma,w,\nu}}) \arrow[r] \arrow[d,equal] & 0 \\
K^{\breve{G}}(Z^{\preceq w}_{\gamma\nu}) \arrow[r] & K^{\breve{G}}(T^*_{\Ob_{\gamma,w,\nu}}) \arrow[r] & 0,
\end{tikzcd}$$
by which we have a split $R(\breve{G})$-homomorphism $K^{\breve{G}}(T^*_{\Ob_{\gamma,w,\nu}}) \rightarrow K^{\breve{G}}(T^*_{\gamma,w,\nu})$ satisfying that the composition $$\iota_w: K^{\breve{G}}(T^*_{\Ob_{\gamma,w,\nu}}) \rightarrow K^{\breve{G}}(T^*_{\gamma,w,\nu})\rightarrow K^{\breve{G}}(Z^{\preceq w}_{\gamma\nu})$$ is also a split $R(\breve{G})$-homomorphism. 

Recall in \eqref{Pwrv} the subgroup $P_{\gamma\nu}^w$ and in \eqref{iso1} the isomorphism $$\Ob_{\gamma,w,\nu}\simeq G/P_{\gamma\nu}^w\simeq \breve{G}/\breve{P}_{\gamma\nu}^w.$$ We have a G-equivariant vector bundle $T^*_{\Ob_{\gamma,w,\nu}}\to \breve{G}/\breve{P}_{\gamma\nu}^w$, which induces a Thom type isomorphism 
\begin{equation}\label{iso:KKR}
K^{\breve{G}}(T^*_{\Ob_{\gamma,w,\nu}})\simeq K^{\breve{G}}(\breve{G}/\breve{P}_{\gamma\nu}^w).
\end{equation}
Recall the notation of $G$-equivariant sheaf $\mathcal{O}_{G/H}(\chi)$ in \S\ref{sec:Kgroup}. By taking $\chi\in R(\breve{P}_{\gamma\nu}^w)$, we get a $G$-equivariant sheaf $\mathcal{O}_{\breve{G}/\breve{P}_{\gamma\nu}^w}(\chi)$, and hence a class $[\mathcal{O}_{\breve{G}/\breve{P}_{\gamma\nu}^w}(\chi)]\in  K^{\breve{G}}(\breve{G}/\breve{P}_{\gamma\nu}^w)$. Let $\mathcal{O}_{T^*_{\Ob_{\gamma,w,\nu}}}(\chi)$ (resp. $[\mathcal{O}_{T^*_{\Ob_{\gamma,w,\nu}}}(\chi)]\in K^{\breve{G}}(T^*_{\gamma,w,\nu})$) denote the pull back of $\mathcal{O}_{\breve{G}/\breve{P}_{\gamma\nu}^w}(\chi)$ (resp. $[\mathcal{O}_{\breve{G}/\breve{P}_{\gamma\nu}^w}(\chi)]$) by \eqref{iso:KKR}. 
Since $K^{\breve{G}}(\breve{G}/\breve{P}_{\gamma\nu}^w)\simeq R(\breve{P}_{\gamma\nu}^w)$, each element in $K^{\breve{G}}(\breve{G}/\breve{P}_{\gamma\nu}^w)$ (resp. $K^{\breve{G}}(T^*_{\Ob_{\gamma,w,\nu}})$) is in the form $[\mathcal{O}_{\breve{G}/\breve{P}_{\gamma\nu}^w}(\chi)]$ (resp. $[\mathcal{O}_{T^*_{\Ob_{\gamma,w,\nu}}}(\chi)]$) for some $\chi\in R(\breve{P}_{\gamma\nu}^w)$.

Denote $$\chi_{\gamma\nu}^w:=\iota_w([\mathcal{O}_{T^*_{\Ob_{\gamma,w,\nu}}}(\chi)])\in  K^{\breve{G}}(Z^{\preceq w}_{\gamma\nu}).$$
Note that the notation $\chi_{\gamma\nu}^w$ is somewhat redundant, since a single $\chi\in R(\breve{P}_{\gamma\nu}^w)$ already contains the information $(\gamma,w,\nu)$. But we still keep this information in the notation for ease of understanding.
%

%by which we can define a split $R(\breve{G})$-homomorphism $\iota_w: K^{\breve{G}}(T^*_{\Ob_{\gamma,w,\nu}}) \to K^{\breve{G}}(Z^{\preceq w}_{\gamma\nu})$ such that $\iota_w$ factors through $K^{\breve{G}}(T^*_{\gamma,w,\nu})$. Denote $\chi_{\gamma\nu}^{w}:=\iota_w([\mathcal{O}_{T^*_{\Ob_{\gamma,w,\nu}}}(\chi)])$, which is supported by $T^*_{\gamma,w,\nu}$, i.e., $\chi_{\gamma\nu}^{w}|_{Z_{\gamma\nu} \backslash T^*_{\gamma,w,\nu}}=0$.

Let $\mathcal{B}(P^w_{\gamma\nu})$ be an $R(\breve{G})$-basis of $R(\breve{P}^w_{\gamma\nu})\simeq K^{\breve{G}}(T^*_{\Ob_{\gamma,w,\nu}})$.
\begin{prop}\label{prop:noncbasis}
For $w\in\mathcal{D}_{\gamma\nu}$, the set $\{{\chi}_{\gamma\nu}^y\mid y\in\mathcal{D}_{\gamma\nu}, y\preceq w,\chi\in\mathcal{B}(P^y_{\gamma\nu})\}$ forms an $R(\breve{G})$-basis of $K^{\breve{G}}(Z^{\preceq w}_{\gamma\nu})$. 
   In particular, the set $\{{\chi}_{\gamma\nu}^w\mid w\in\mathcal{D}_{\gamma\nu},\chi\in\mathcal{B}(P^w_{\gamma\nu})\}$ forms an $R(\breve{G})$-basis of $K^{\breve{G}}(Z_{\gamma\nu})$.
\end{prop}
\begin{proof}
    By our construction, $K^{\breve{G}}(Z^{\preceq w}_{\gamma\nu})=\mathrm{Span}_{R(\breve{G})}\{\chi_{\gamma\nu}^w\mid\chi\in\mathcal{B}(P^w_{\gamma\nu})\}\oplus K^{\breve{G}}(Z^{\prec w}_{\gamma\nu})$. The statements can be proved by induction on $w$.
\end{proof}
\subsection{Realization of affine Hecke algebras}
%---------------------------------

For $\lambda\in Q$ and $1\leq i\leq d$, let $e^\lambda$ and $e^{\alpha_i}$ be the virtual characters of $T$ associated with $\lambda$ and $\alpha_i$, respectively. They can also be regarded as elements in $R(\breve{T})$ via the natural embedding $R(T)\hookrightarrow R(\breve{T})=R(T)[q,q^{-1}]$.
Recall the notations in \eqref{notation:elambda} that $e_{\id}^{\lambda}=e^{\lambda}\cdot[\mathcal{O}_{T^*_{\id}}]$ and $e_{s_i}^{\alpha_i}=e^{\alpha_i}\cdot[\mathcal{O}_{T^*_{s_i}}]$.
The following is the celebrated equivariant K-theoretic construction of affine Hecke algebras due to Ginzburg \cite{G87} (see details in \cite{CG97}). 
\begin{thm}[Ginzburg]
    There exists a $\Z[q,q^{-1}]$-algebra isomorphism $\widetilde{\HH} \simeq K^{\breve{G}}(Z)$ given by
    \begin{align}\label{iso:Hecke}
        e^\lambda \mapsto e_\id^\lambda, \quad H_i \mapsto qe_{s_i}^{\alpha_i}+ q^{-1}e_\id^0,\qquad (\lambda\in Q, 1\leq i\leq d).
    \end{align}
\end{thm}
\begin{proof}
It was shown in \cite[\S 7.6]{CG97}. We reformulate the isomorphism by \eqref{iso:Hecke} (comparing with \cite[(7.6.1)]{CG97}) to adapt to our choice of generators of $\widetilde{\HH}$ and the quantum parameter $q$.
\end{proof}

%--------------------------------------
\subsection{A key convolution formulas}
%--------------------------------------
%\red{Imitating Proposition~\ref{5}, Lemma~\ref{lem2} and Proposition~\ref{6} takes us the following formula about equivariant K-groups. We omit its proof since it is similar to the one for its Borel-Moore homology counterpart.  
%\begin{prop}\label{imitforK}
%Let $\gamma\in\Lambda_\ff$ and $w \in \D_\gamma$. We have 
 %   \begin{align*}
  %       e_{\gamma,\id}^0 \star e_w^0 =[\mathcal{O}_{T^{*}_{\gamma,\id}}] \star [\mathcal{O}_{T^*_{w}}] =[\mathcal{O}_{T^*_{\gamma,w}}]+\sum_{y < w} k_y[\mathcal{O}_{T^*_{\gamma,y}}], \quad \mbox{for some $k_y\in R(\breve{T})$},
%    \end{align*} 
%\end{prop}
%} 
For $1\leq i\leq d$, let $P_i^-$ be the opposite parabolic subgroup of $P_i$ (see the definition of $P_i$ in \S\ref{subsec:parabolic}). The Levi subgroup of $P_i^-$ is denoted by $L_i$ and the unipotent radical of $P_i^-$ by $U_i^-$. Write $B_i=L_i \cap B$.The reductive part of the Levi decomposition of $B_i$ is $T$, which implies $R(B_i)\simeq R(T)$. 

Take $\gamma\in\Lambda$ and set $M_1=P_i^-P_\gamma/P_\gamma\simeq U_i^-/(U_i^-\cap P_\gamma)$ and $M_2=M_3=P_i^-B/B=P_i^- /B_i \simeq U_i^- \times L_i/B_i$, where $\breve{L}_i$ acts adjointly on $U_i$, which are open $\breve{L}_i$-subvarieties of $\sF_\gamma$ and $\sB$, respectively. We write $M_{ij}=M_i\times M_j$, $(1\leq i<j\leq 3)$, for short.  Let $Z_{12}=M_{12} \cap \Ob_{\gamma, \id}$, $Z_{13}=M_{13} \cap \Ob_{\gamma, \id}$, $Z_{23}=M_{23} \cap \overline{\Ob}_{s_i}$.
   Let $f: U_i^- \to U_i^-/(U_i^-\cap P_\gamma)$ be the canonical morphism, and denote $\mathrm{graph}(f):=\{(x,f(x))~|~x\in U_i^-\}$. 
\begin{lem} \label{lem:4eq}
Keep the notations above. We have that, as $\breve{L}_i$-varieties,
\begin{align*}
    	&Z_{12} =Z_{13} \simeq \mathrm{graph}(f) \times L_i/B_i \simeq M_2=M_3,\\
    	&Z_{23} \simeq \Delta U_i^- \times L_i/B_i \times L_i/B_i,\\
        &T_{Z_{12}}^*M_{12}=T_{Z_{13}}^*M_{13} \simeq T_{\mathrm{graph}(f)}^*(M_1 \times U_i^-) \times L_i/B_i,\\
        &T_{Z_{23}}^*M_{23} \simeq T_{\Delta U_i^-}^*(U_i^- \times U_i^-)\times L_i/B_i \times L_i/B_i,
    \end{align*}
    \end{lem}
    \begin{proof}
   We only prove $Z_{23}\simeq\Delta U_i^-\times L_i/B_i\times L_i/B_i$ and the rest is similar. 
   
   As a $\breve{L}_i$-variety,
   $M_{23}\cap\Ob_{s_i}\simeq P^-_iB(B,s_iB)\simeq \Delta U_i^-\times L_i(B_i,s_iB_i)=\Delta U_i^-\times(L_i/B_i\times L_i/B_i\backslash(\Delta L_i/B_i))$, while $M_{23}\cap\Ob_{\id}\simeq\Delta U_i^-\times\Delta(L_i/B_i)$. Hence $Z_{23}=(M_{23}\cap\Ob_{s_i})\sqcup(M_{23}\cap\Ob_\id)\simeq\Delta U_{i}^-\times L_i/B_i\times L_i/B_i.$
    \end{proof}

Let $\omega_\gamma:=\sum_{\alpha \in \Pi_\gamma^+} \alpha$ for $\gamma\in\Lambda_\ff$. We have a virtual character $e^{\omega_\gamma}\in R(T)\subset R(\breve{T})$.
Recall the notations in \eqref{notation:elambda} that $e_{\gamma,\id}^{\omega_\gamma}=e^{\omega_\gamma}\cdot[\mathcal{O}_{T^*_{\gamma,\id}}]$ and $e_{s_i}^{\alpha_i}=e^{\alpha_i}\cdot[\mathcal{O}_{T^*_{s_i}}]$. 

\begin{lem}\label{calc} For any $\gamma\in\Lambda_\ff$ and $\alpha_i\in\Delta_{\gamma}$, it holds that
    $$e_{\gamma,\id}^{\omega_\gamma} \star e_{s_i}^{\alpha_i}=-(1+q^{-2})e_{\gamma,\id}^{\omega_\gamma}.$$
\end{lem}

\begin{proof} Keep all the above notations. 
    We have a Cartisian diagram
    $$\begin{tikzcd}
        Z_{12} \times_{M_2} Z_{23} \arrow[r] \arrow[d] &  Z_{13} \arrow[d]\\
        \Ob_{\gamma, \id} \times_\sB \overline{\Ob}_{s_i} \arrow[r] & \Ob_{\gamma,\id}
    \end{tikzcd},$$
    which further induces another Cartisian diagram
    $$\begin{tikzcd}
        T_{Z_{12}}^* M_{12} \times_{T^* M_2} T_{Z_{23}}^* M_{23} \arrow[r] \arrow[d] &  T_{Z_{13}}^* M_{13} \arrow[d]\\
        T_{\gamma,\id}^* \times_\sB T_{s_i}^* \arrow[r] & T_{\gamma,\id}^*
    \end{tikzcd}.$$
    Therefore, by Lemma \ref{Kcon dia} and considering the forgetful functor from $\mathtt{Coh}^{\breve{G}}$ to $\mathtt{Coh}^{\breve{L}_i}$, we have the following commutative diagram
    \begin{equation}\label{diag:5-10}
    \begin{tikzcd}
        {K^{\breve{G}}(T_{\gamma, \id}^*) \otimes K^{\breve{G}}(T_{s_i}^*)} \arrow[r,"\star"] \arrow[d] & {K^{\breve{G}}(T_{\gamma, \id}^*)} \arrow[d]\\
        {K^{\breve{L}_i}(T_{Z_{12}}^*M_{12}) \otimes K^{\breve{L}_i}(T_{Z_{23}}^*M_{23})} \arrow[r,"\star"] & {K^{\breve{L}_i}(T_{Z_{13}}^*M_{13})}
    \end{tikzcd}.
    \end{equation}

The sheaf $\mathcal{O}_{\Ob_{\gamma,\id}}(\chi)$ corresponds to $G\times^{B}\mathbb{V}_{\chi}$, where $\Ob_{\gamma,\id}\simeq G/B$ canonically. As a $\breve{L}_i$-equivariant sheaf, $\mathcal{O}_{\Ob_{\gamma,\id}}(\chi)|_{Z_{13}}$ corresponds to $(G\times^{B}\mathbb{V}_{\chi})|_{Z_{13}}=(U_i^-\times P_i\times^{B}\mathbb{V}_\chi)\simeq\mathrm{graph}(f)\times L_i\times^{B_i}\mathbb{V}_\chi$, which is a $\breve{L}_i$-equivariant bundle over $Z_{13}$.
Hence, the map $K^{\breve{G}}(\Ob_{\gamma,\id})\to K^{\breve{L_i}}(Z_{13})$ via $\mathcal{O}_{\Ob_{\gamma,\id}}(\chi)\mapsto\mathcal{O}_{\mathrm{graph}(f)}\boxtimes\mathcal{O}_{L_i/B_i}(\chi) (=\mathcal{O}_{\Ob_{\gamma,\id}}(\chi)|_{Z_{13}})$ is an isomorphism. So the map $K^{\breve{G}}(T_{\gamma,\id}^*) \to K^{\breve{L}_i}(T_{Z_{13}}^* M_{13})$ in the left of \eqref{diag:5-10} is an isomorphism as well.

	Thanks to Lemma~\ref{lem:4eq}, there exist two canonical morphisms by external tensor product
\begin{align*}
		K^{\breve{L}_i}(T_{\mathrm{graph}(f)}^*(M_1 \times U_i^-)) \otimes K^{\breve{L}_i}(L_i/B_i) \to& K^{\breve{L}_i}(T_{Z_{12}}^*M_{12})=K^{\breve{L}_i}(T_{Z_{13}}^*M_{13}),\\
		K^{\breve{L}_i}(T_{\Delta U_i^-}^*(U_i^- \times U_i^-)) \otimes K^{\breve{L}_i}(L_i/B_i \times L_i/B_i) \to& K^{\breve{L}_i}(T_{Z_{23}}^*M_{23})
\end{align*}
	such that $(\mathcal{L}_1 \boxtimes \mathcal{L}_2)\star(\mathcal{L}_3 \boxtimes \mathcal{L}_4)=(\mathcal{L}_1 \star \mathcal{L}_3) \boxtimes (\mathcal{L}_2 \star \mathcal{L}_4)$ (Recall Lemma~\ref{Ku-cov} and don't forget that the ambient space of $L_i/B_i$ is $T^*(L_i/B_i)$).

Recall the virtual character $e^{\lambda}\in R(\breve{T})\simeq R(\breve{B})\simeq R(\breve{B}_i)$ for any weight $\lambda\in T^*$. We write $\mathcal{O}_{T^*\sB}(\lambda):=\mathcal{O}_{T^*\sB}(e^\lambda)$ for short. Its  restriction $\mathcal{O}_{T^*\sB}(\lambda)|_{T^*M_2}=\mathcal{O}_{T^*U^{-}_i}\boxtimes\mathcal{O}_{T^*(L_i/B_i)}(\lambda)$,  which further implies
\begin{align*}
    &e_{\gamma,\id}^{\lambda}|_{T^*_{Z_{12}}M_{12}}=[\mathcal{O}_{\mathrm{graph}(f)}]\boxtimes[\mathcal{O}_{L_i/B_i}(\lambda)]\\
    &e^{\lambda}_{s_i}|_{T^*_{Z_{23}}M_{23}}=[\mathcal{O}_{T_{\Delta U_i^-}^*(U_i^- \times U_i^-)}] \boxtimes [\mathcal{O}_{L_i/B_i}] \boxtimes [\mathcal{O}_{L_i/B_i}(\lambda)].
\end{align*}
    Therefore, the composition map along the left-lower corner of the diagram~\eqref{diag:5-10}, sends $e_{\gamma,\id}^{\omega_\gamma}\otimes e_{s_i}^{\alpha_i}\in {K^{\breve{G}}(T_{\gamma, \id}^*) \otimes K^{\breve{G}}(T_{s_i}^*)}$ to 
  	\begin{align*}
		&([\mathcal{O}_{\mathrm{graph}(f)}] \boxtimes [\mathcal{O}_{L_i/B_i}(\omega_\gamma)]) \star ([\mathcal{O}_{T_{\Delta U_i^-}^*(U_i^- \times U_i^-)}] \boxtimes [\mathcal{O}_{L_i/B_i}] \boxtimes [\mathcal{O}_{L_i/B_i}(\alpha_i)])\\
		=&([\mathcal{O}_{\mathrm{graph}(f)}] \star [\mathcal{O}_{T_{\Delta U_i^-}^*(U_i^- \times U_i^-)}]) \boxtimes ([\mathcal{O}_{L_i/B_i}(\omega_\gamma)] \star ([\mathcal{O}_{L_i/B_i}] \boxtimes [\mathcal{O}_{L_i/B_i}(\alpha_i)]))\\
		=&R\Gamma(\mathcal{O}_{L_i/B_i}(\omega_\gamma) \otimes_{\mathcal{O}_{T^*(L_i/B_i)}}^{\mathbb{L}} \mathcal{O}_{L_i/B_i}) [\mathcal{O}_{\mathrm{graph}(f)}] \boxtimes [\mathcal{O}_{L_i/B_i}(\alpha_i)]\in K^{\breve{L}_i}(T_{Z_{13}}^*M_{13}),
	\end{align*} where $R\Gamma$ is the derived global section functor.
    
	Since $T^*(L_i/B_i) \simeq L_i \times^{B_i} \V_{q^2 e^{\alpha_i}}$ where $\V_{q^2 e^{\alpha_i}}$ is the $\breve{B}_i$-module associated with the virtual character $q^2e^{\alpha_i}\in R(\breve{T})\simeq R(\breve{B}_i)$, we have a $\breve{G}$-equivariant Koszul complex:
	$${0} \to {q^{-2} \mathcal{O}_{T^*(L_i/B_i)}(-\alpha_i)} \to {\mathcal{O}_{T^*(L_i/B_i)}} \to {\mathcal{O}_{L_i/B_i}} \to {0}.$$
 	Therefore, $$\mathcal{O}_{L_i/B_i}(\omega_\gamma) \otimes_{\mathcal{O}_{T^*(L_i/B_i)}}^{\mathbb{L}} \mathcal{O}_{L_i/B_i}=\mathcal{O}_{L_i/B_i}(\omega_\gamma)-q^{-2}\mathcal{O}_{L_i/B_i}(\omega_\gamma-\alpha_i).$$
 	Since $R\Gamma(\mathcal{O}_{L_i/B_i}(\lambda))=\frac{e^{\lambda-\alpha_i}-e^{s_i(\lambda)}}{e^{-\alpha_i}-1}$ (cf. \cite[p17\&Ex.32]{Ku11}), %\redtext{by \cite[Corollary 6.1.17]{CG97} (but note that the choice of the positive roots in \cite{CG97} is different)}, 
    we obtain
  \begin{align*}
      &([\mathcal{O}_{\mathrm{graph}(f)}] \star [\mathcal{O}_{T_{\Delta U_i^-}^*(U_i^- \times U_i^-)}]) \boxtimes ([\mathcal{O}_{L_i/B_i}(\omega_\gamma)] \star ([\mathcal{O}_{L_i/B_i}] \boxtimes [\mathcal{O}_{L_i/B_i}(\alpha_i)]))\\
      =&R\Gamma(\mathcal{O}_{L_i/B_i}(\omega_\gamma)-q^{-2}\mathcal{O}_{L_i/B_i}(\omega_\gamma-\alpha_{i}))[\mathcal{O}_{\mathrm{graph}(f)} \boxtimes \mathcal{O}_{L_i/B_i}(\alpha_i)]\\
      =&-(1+q^{-2})e^{\omega_\gamma-\alpha_i}[\mathcal{O}_{\mathrm{graph}(f)}] \boxtimes [\mathcal{O}_{L_i/B_i}(\alpha_i)]\\
      =&-(1+q^{-2})[\mathcal{O}_{\mathrm{graph}(f)}] \boxtimes [\mathcal{O}_{L_i/B_i}(\omega_\gamma)].
  \end{align*}
  Thus we arrive at $e_{\gamma,\id}^{\omega_\gamma} \star e_{s_i}^{\alpha_i}=-(1+q^{-2})e_{\gamma,\id}^{\omega_\gamma}$ as desired.
\end{proof}

%----------------------------
\subsection{Localizations of $R(\breve{G})$-modules}
%----------------------------------
Let $R(\breve{G})_{loc}$ be the localization of $R(\breve{G})$ at $(0)$. For an $R(\breve{G})$-module $M$, denote $M \otimes_{R(\breve{G})} R(\breve{G})_{loc}$ by $M_{loc}$. Below we shall take $M=\widetilde{\HH}, \widetilde{\Sc}_{\ff}$, etc.

\begin{prop}\label{simpl:H}
    The algebra $\widetilde{\HH}_{loc}$ is split simple with the unique irreducible representation $K^{\breve{G}}(\widetilde{\mathcal{N}})_{loc}$.
\end{prop}

\begin{proof}
    By \cite[Claim 7.6.7]{CG97}, we know that the affine Hecke algebra $\widetilde{\HH}$ has a faithful representation $K^{\breve{G}}(\widetilde{\mathcal{N}})$. Furthermore, since $$\mathrm{rank}_{R(\breve{G})}(\widetilde{\HH})=(\mathrm{rank}_{R(\breve{G})}(\widetilde{\mathcal{N}}))^2=\mathrm{rank}_{R(\breve{G})}(\End_{R(\breve{G})}(K^{\breve{G}}(\widetilde{\mathcal{N}}))),$$
    the proposition follows.
\end{proof}

\begin{lem}\label{rank}
   The $\widetilde{\mathbf{H}}$-module $\widetilde{\mathbf{T}}_\ff$ is free over $R(\breve{T})$ (resp. $R(\breve{G})$) with rank $\sum_{\gamma\in\Lambda_\ff}\#\mathcal{D}_{\gamma}$ (resp. $(\sum_{\gamma\in\Lambda_\ff}\#\mathcal{D}_{\gamma})\#\W$). 
\end{lem}
\begin{proof}
    As an $R(\breve{T})$-module, $\mathbf{x}_{\gamma}\widetilde{\mathbf{H}}$ is free with a basis $\{\mathbf{x}_{\gamma}H_{w}\mid w\in\D_{\gamma}\}$, whose rank is $\#\mathcal{D}_{\gamma}$ clearly. Hence $\widetilde{\mathbf{T}}_\ff=\bigoplus_{\gamma\in\Lambda_\ff}\mathbf{x}_\gamma\widetilde{\mathbf{H}}$ is free over $R(\breve{T})$ with rank $\sum_{\gamma\in\Lambda_{\ff}}\#\D_{\gamma}$. The rest part follows from the fact that $R(\breve{T})$ is a free $R(\breve{G})$-module with rank $\#\W$.  
\end{proof}

\begin{prop}\label{rank2}
    The algebra $\widetilde{\Sc}_{\ff,loc}$ is split simple. Moreover, $$\dim_{R(\breve{G})_{loc}}(\widetilde{\Sc}_{\ff,loc})=(\sum_{\gamma \in \Lambda_\ff}\#\D_\gamma)^2.$$
\end{prop}
\begin{proof}
The split simplicity of $\widetilde{\Sc}_{\ff,loc}=\mathrm{End}_{\widetilde{\HH}_{loc}}(\widetilde{\mathbf{T}}_{\ff,loc})$ follows from that of $\widetilde{\HH}_{loc}$ shown in Proposition~\ref{simpl:H}. Since $\mathrm{rank}_{R(\breve{G})}(\widetilde{\bT}_\ff)=(\sum_{\gamma \in \Lambda_\ff}\#\D_\gamma)\#\W$ derived in Lemma~\ref{rank}, we have $$\dim_{R(\breve{G})_{loc}}(\widetilde{\Sc}_{\ff,loc})=\frac{(\mathrm{rank}_{R(\breve{G})}(\widetilde{\bT}_\ff))^2}{\mathrm{rank}_{R(\breve{G})}(\widetilde{\HH})}=(\sum_{\gamma \in \Lambda_\ff}\#\D_\gamma)^2.$$
\end{proof}

\begin{cor}\label{center of aff schur}
    The center of $\widetilde{\Sc}_\ff$ coincides with $R(\breve{G})$.
\end{cor}

\begin{proof}
    Since $\widetilde{\Sc}_\ff$ is a finite $R(\breve{G})$-subalgebra of the split simple algebra $\widetilde{\Sc}_{\ff,loc}$, the corollary follows.
\end{proof}

Consider the canonical morphism $\Ob_{\gamma,w,\nu} \to \sF_\gamma \times \sF_\nu$, whose pull back (together with the external product $\boxtimes$) induces an $R(\breve{G})$-algebra morphism $R(\breve{P}_\gamma) \otimes_{R(\breve{G})} R(\breve{P}_\nu) \to R(\breve{P}_{\gamma\nu}^w)$ by $\chi \otimes \chi' \to \chi \cdot w(\chi')$, where $\cdot$ is the product in the representation ring $R(\breve{P}_{\gamma\nu}^w)$.
The image is denoted by $\mathscr{R}_{\gamma\nu}^w:=R(\breve{P}_\gamma)R(w\breve{P}_\nu w^{-1})$.
    %Clearly, $(R(P_{\gamma})R(wP_{\nu}w^{-1}))_{loc}$ is a subring of $R(P^w_{\gamma\nu})_{loc}$.  
\begin{lem}\label{loc gal} The localization 
$\mathscr{R}_{\gamma\nu,loc}^w=R(\breve{P}^w_{\gamma\nu})_{loc}$.  
\end{lem}
\begin{proof}
For any integral domain $S$ which is a finite $R(\breve{G})$-algebra, $S_{loc}$ is nothing but the localization of $S$ at $(0)$. Therefore,
\begin{align*}
&\mathscr{R}_{\gamma\nu,loc}^w=R(\breve{P}_\gamma)_{loc}R(w\breve{P}_\nu w^{-1})_{loc}\\&=(R(\breve{T})^{\W_{\gamma}})_{loc}(R(\breve{T})^{w\W_{\nu}w^{-1}})_{loc}=(R(\breve{T})_{loc})^{\W_{\gamma}}(R(\breve{T})_{loc})^{w\W_{\gamma}w^{-1}}\\&=(R(\breve{T})_{loc})^{\W_{\gamma\nu}^w}=R(\breve{T})^{\W_{\gamma\nu}^w}_{loc}=R(\breve{P}^w_{\gamma\nu})_{loc},  
\end{align*} where the forth equality is due to a standard result of Galois theory.
\end{proof}

%------------------------------
\subsection{A restriction formula}
%------------------------------

%By the definition of the convolution, 

%we have
%\begin{equation}
%    \begin{tikzcd}
%    K^{\breve{G}}(T_{\gamma,w,\nu}^*) \otimes_{R(\breve{P}_{\nu})} K^{\breve{G}}(T_{\nu,\id}^*) \arrow[r,"\star"] \arrow[d] & K^{\breve{G}}(T_{\gamma,\theta_\gamma w_{\gamma\nu}^+}^*) \arrow[d]\\
%    K^{\breve{G}}(Z_{\gamma\nu}) \otimes_{R(\breve{P}_{\nu})} K^{\breve{G}}(Z_{\nu\mu_0}) \arrow[r,"\star"] & K^{\breve{G}}(Z_{\gamma\mu_0}).
%    \end{tikzcd}
%\end{equation}
%Since $\breve{G}/\breve{P}^w_{\gamma\mu_0}=\Ob_{\gamma,w}\to\sB$ $(\forall w\in\mathcal{D}_\gamma)$ is a $G$-equivariant affine bundle, by Thom isomorphism, we have $R(\breve{P}^{w}_{\gamma\mu_0})=K^{\breve{G}}(\breve{G}/\breve{P}^w_{\gamma\mu_0})\simeq K^{\breve{G}}(\sB)\simeq R(\breve{T})$.
%Therefore, 
Let $\chi\in R(\breve{P}^w_{\gamma\nu})$. By the definition of the convolution, the product $\chi_{\gamma\nu}^w\star e^{0}_{\nu,\id}$ lies in $K^{\breve{G}}(Z_{\gamma,\id}^{\preceq\theta_\gamma w_{\gamma\nu}^+})$. Since  $R(\breve{P}^w_{\gamma\nu})\subset R(\breve{T})\simeq R(\breve{P}_{\gamma\mu_0}^{\theta_\gamma w_{\gamma\nu}^+})$, the notation $[\mathcal{O}_{T_{\Ob_{\gamma,\theta_\gamma w_{\gamma\nu}^+}}^*}(\chi)]\in K^{\breve{G}}(T_{\Ob_{\gamma,\theta_\gamma w_{\gamma\nu}^+}}^*)$ makes sense. 
\begin{lem}\label{tec lem}
For any $(\gamma,w,\nu)\in\Xi_\ff$ and $\chi\in R(\breve{P}^w_{\gamma\nu})$, it holds that
   $$(\chi_{\gamma\nu}^{w}\star e^{0}_{\nu,\id})|_{T_{\Ob_{\gamma,\theta_\gamma w_{\gamma\nu}^+}}^*}=[\mathcal{O}_{T_{\Ob_{\gamma,\theta_\gamma w_{\gamma\nu}^+}}^*}(\chi)].$$
\end{lem}
\begin{proof}
   As shown in Proposition \ref{6}, $T_{\Ob_{\gamma,w,\nu}}^* \times \widetilde{\mathcal{N}}$ and $\widetilde{\mathcal{N}}_\gamma \times T_{\nu,\id}^*$ intersect transversely in $\widetilde{\mathcal{N}}_\gamma \times \widetilde{\mathcal{N}}_\nu \times \widetilde{\mathcal{N}}$ and we have a Cartisian diagram
    $$\begin{tikzcd}
        T_{\Ob_{\gamma,w,\nu}}^* \times_{\widetilde{\mathcal{N}}_\nu} T_{\nu,\id}^* \arrow[r, "\simeq"] \arrow[d] & T_{\Ob_{\gamma,w,\nu}}^* \circ T_{\nu,\id}^* \arrow[d] \\
        T_{\gamma,w,\nu}^* \times_{{\widetilde{\mathcal{N}}}_\nu} T_{\nu,\id}^* \arrow[r] & T_{\gamma,\theta_\gamma w_{\gamma\nu}^+}^*.
    \end{tikzcd}$$
    Then by Lemma \ref{Kcon dia} we obtain a commutative diagram
    $$\begin{tikzcd}
         K^{\breve{G}}(T_{\gamma,w,\nu}^*) \otimes_{R(\breve{P}_\nu)} K^{\breve{G}}(T_{\nu,\id}^*) \arrow[r,"\star"] \arrow[d] & K^{\breve{G}}(T_{\gamma,\theta_\gamma w_{\gamma\nu}^+}^*) \arrow[d] \ar[dr] &\\
         K^{\breve{G}}(T_{\Ob_{\gamma,w,\nu}}^*) \otimes_{R(\breve{P}_\nu)} K^{\breve{G}}(T_{\nu,\id}^*) \arrow[r,"\star"] & K^{\breve{G}}(T_{\Ob_{\gamma,w,\nu}}^* \circ T_{\nu,\id}^*) \ar[r]& K^{\breve{G}}(T_{\Ob_{\gamma,\theta_\gamma w_{\gamma\nu}^+}}^*).
    \end{tikzcd}$$
    Moreover, the proof of Proposition \ref{6} shows that $pr_{12}^{-1}(T_{\Ob_{\gamma,w,\nu}}^*)$ and $pr_{23}^{-1}(T_{\Ob_{\nu,\id}}^*)$ intersect transversely and $pr_{13}:pr_{12}^{-1}(T_{\Ob_{\gamma,w,\nu}}^*)\cap pr_{23}^{-1}(T_{\Ob_{\nu,\id}}^*)\longrightarrow T_{\Ob_{\gamma,w,\nu}}^* \circ T_{\nu,\id}^*$ is an isomorphism. Hence we have $[\mathcal{O}_{T_{\Ob_{\gamma,w,\nu}}^*}] \star e^{0}_{\nu,\id} = [\mathcal{O}_{T_{\Ob_{\gamma,w,\nu}}^* \circ T_{\nu,\id}^*}]$ and the convolution $\star$ is an $R(\breve{P}_\gamma) \otimes_{R(\breve{G})} R(\breve{T})$-morphism by considering the morphism $pr_{13}$. For brevity, denote $T_{\Ob_{\gamma,\theta_\gamma w_{\gamma\nu}^+}}^*$ by $V$.
    Firstly, we may assume that $\chi=\chi' \cdot w(\chi'') \in \mathscr{R}_{\gamma\nu}^w$ , where  $\chi'\in R(\breve{P}_\gamma)$ and $\chi''\in R(\breve{P}_\nu)$. Then
    \begin{align*}
        &(\chi_{\gamma\nu}^{w} \star e^{0}_{\nu,\id})|_V=([\mathcal{O}_{T_{\Ob_{\gamma,w,\nu}}^*}(\chi)] \star e^{0}_{\nu,\id})|_V=([\mathcal{O}_{T_{\Ob_{\gamma,w,\nu}}^*} (\chi' \cdot w(\chi''))] \star e^{0}_{\nu,\id})|_V\\
        =&([\mathcal{O}_{T_{\Ob_{\gamma,w,\nu}}^*}(\chi')]  \star ({\chi''}_{\nu\mu_0}^\id))|_V= (\chi' \otimes \chi'') \bullet ([\mathcal{O}_{T_{\Ob_{\gamma,w,\nu}}^*}]  \star e_{\nu,\id}^0)|_V=(\chi' \otimes \chi'') \bullet [\mathcal{O}_V] \\
        =&[\mathcal{O}_V (\chi' \cdot \theta_\gamma w_{\gamma\nu}^+(\chi''))] =[\mathcal{O}_V (\chi' \cdot w(\chi''))]=[\mathcal{O}_V (\chi)],
    \end{align*}
    where $(\chi'\otimes\chi'')\bullet$ means $p_\gamma^{*}(\chi')\otimes^{\mathbb{L}}p_{\mu_0}^*(\chi'')\otimes^{\mathbb{L}}-$ with the canonical projections $p_{\gamma}:V\to\sF_{\gamma}$ and $p_{\mu_0}:V\to\sB$. 
    So for any $\chi \in \mathscr{R}_{\gamma\nu}^w$, the statement holds.
    
    By Lemma \ref{loc gal}, for any $\chi \in R(\breve{P}^w_{\gamma\nu})$, there exists a nonzero $\chi' \in R(\breve{G})$ such that $\chi'\chi \in \mathscr{R}_{\gamma\nu}^w$.
    Then $$\chi'\cdot (\chi_{\gamma\nu}^{w} \star e^{0}_{\nu,\id})|_V= ((\chi'\chi)_{\gamma\nu}^w \star e^{0}_{\nu,\id})|_V=[\mathcal{O}_V (\chi'\chi)].$$ Since $R(\breve{P}^w_{\gamma\nu})$ is a torsion-free $R(\breve{G})$-module, we have $(\chi_{\gamma\nu}^{w} \star e^{0}_{\nu,\id})|_V= [\mathcal{O}_V (\chi)]$, which completes the proof.
\end{proof}    
\subsection{Realization of affine $q$-Schur algebras}
%-----------------------

\begin{thm}\label{geo affine Schur-Weyl duality}
\begin{itemize}
    \item[(1)] There exists an $\widetilde{\HH}$-module isomorphism  $$\widetilde{\varphi}: \widetilde{\bT}_\ff \to K^{\breve{G}}(\widetilde{\mathcal{N}}_\ff \times_\mathcal{N} \widetilde{\mathcal{N}})\quad\mbox{defined by}\quad \mathbf{x}_\gamma \mapsto e_{\gamma,\id}^{\omega_\gamma}.$$
    \item[(2)] This isomorphism $\widetilde{\varphi}$ induces an algebra isomorphism $K^{\breve{G}}(Z_\ff) \simeq \widetilde{\Sc}_\ff$.
\end{itemize}
\end{thm}

\begin{proof}
(1). 
 Thanks to Lemma \ref{calc}, we see that $\mathbf{x}_\gamma \mapsto e_{\gamma,\id}^{\omega_\gamma}$ determines a (right) $\widetilde{\HH}$-module homomorphism $\widetilde{\varphi}$.
 Firstly, we will show the surjectivity of $\widetilde{\varphi}$, which is equivalent to verify that as an $\widetilde{\HH}$-module, $K^{\breve{G}}(\widetilde{\mathcal{N}}_\ff \times_\mathcal{N} \widetilde{\mathcal{N}})$ can be generated by $e_{\gamma,\id}^{\omega_\gamma}$, ($\gamma\in\Lambda_\ff$).
 
    %Recall in Corollary~\ref{basis:H} the $R(\breve{T})$-basis element $[\mathcal{O}_{T^*_{w}}]\in K^{\breve{G}}(Z)\simeq\widetilde{\HH}$.
    By \eqref{eq:stardot}, we have $$e^0_{\gamma,\id}=[\mathcal{O}_{T^*_{\gamma,\id}}]=e^{\omega_\gamma}_{\gamma,\id}\cdot e^{-\omega_\gamma}=e^{\omega_\gamma}_{\gamma,\id}\star e_\id^{-\omega_\gamma},$$ which shows $e^0_{\gamma,\id}\in \mathrm{Im}\widetilde{\varphi}$. Similar to the proof of Proposition~\ref{5}, we can derive that, for any $w \in \D_\gamma$,
    \begin{align*}
         e_{\gamma,\id}^0 \star e_w^0 =[\mathcal{O}_{T^{*}_{\gamma,\id}}] \star [\mathcal{O}_{T^*_{w}}] =[\mathcal{O}_{T^*_{\gamma,w}}]+\sum_{y < w} k_y[\mathcal{O}_{T^*_{\gamma,y}}], \quad \mbox{for some $k_y\in R(\breve{T})$},
    \end{align*} 
    which shows $[\mathcal{O}_{T^*_{\gamma,w}}]\in \mathrm{Im}\widetilde{\varphi}$ by a standard argument about upper triangular matrices. 
   Furthermore, $[\mathcal{O}_{T^*_{\gamma,w}}]\star e_\id^\lambda=e^{\lambda}\cdot[\mathcal{O}_{T^*_{\gamma,w}}]$ by \eqref{eq:stardot} again, which implies the surjectivity of $\widetilde{\varphi}$ by Corollary~\ref{cor:basisZr}.
   
 By Lemma~\ref{rank} and Corollary~\ref{basis}, we get
    $$\mathrm{rank}_{R(\breve{G})}(\widetilde{\bT}_\ff)=\sum_{\gamma\in\ff}\#\D_{\gamma}\#\W=\mathrm{rank}_{R(\breve{G})}(K^{\breve{G}}(\widetilde{\mathcal{N}}_\ff \times_\mathcal{N} \widetilde{\mathcal{N}})),$$ which forces that
    $\widetilde{\varphi}$ is an isomorphism.
   
(2).
By the associativity of convolution,
     the isomorphism $\widetilde{\varphi}$ induces an algebra homomorphism $\widetilde{\psi}:K^{\breve{G}}(Z_\ff) \to \widetilde{\Sc}_\ff$ satisfying $\widetilde{\psi}(a)\cdot b=a\star b$ for $a\in K^{\breve{G}}(Z_\ff)$ and $b\in K^{\breve{G}}(\widetilde{\mathcal{N}}_\ff \times_\mathcal{N} \widetilde{\mathcal{N}})$, where we identify $K^{\breve{G}}(\widetilde{\mathcal{N}}_\ff \times_\mathcal{N} \widetilde{\mathcal{N}})$ with $\widetilde{\mathbf{T}}_\ff$ under the isomorphism $\widetilde{\varphi}$. 
    Lemma \ref{tec lem} implies $\chi_{\gamma\nu}^{w} \star e_{\nu, \id}^0 \in \chi_{\gamma\mu_0}^{\theta_\gamma w_{\gamma\nu}^+}+K^{\breve{G}}(Z_{\gamma\mu_0}^{\prec \theta_\gamma w_{\gamma\nu}^+})$, for any $w \in \D_{\gamma\nu}$ and $\chi \in R(\breve{P}_{\gamma\nu}^w)$.
Let $\mathcal{M} \in K^{\breve{G}}(Z_\ff) \backslash \{0\}$. Without loss of generality, we may assume that $\mathcal{M} \subset K^{\breve{G}}(Z^{\preceq w}_{\gamma\nu}) \backslash K^{\breve{G}}(Z^{\prec w}_{\gamma\nu})$ for some $(\gamma, w,\nu)\in\Xi_\ff$, i.e., $\mathcal{M} \in \chi_{\gamma\nu}^{w}+K^{\breve{G}}(Z^{\prec w}_{\gamma\nu})$ for some $\chi \in R(\breve{P}_{\gamma\nu}^w)\backslash \{0\}$.

By Lemma \ref{tec lem} again, we have $$(\widetilde{\psi}(\mathcal{M})\cdot e^0_{\nu,\id})|_{T_{\Ob_{\gamma,\theta_\gamma w_{\gamma\nu}^+}}^*}=(\mathcal{M} \star e^0_{\nu,\id})|_{T_{\Ob_{\gamma,\theta_\gamma w_{\gamma\nu}^+}}^*} = [\mathcal{O}_{T_{\Ob_{\gamma,\theta_\gamma w_{\gamma\nu}^+}}^*}(\chi)] \neq 0.$$ Hence $\mathrm{Ker}(\widetilde{\psi})=0$ and $\widetilde{\psi}$ is injective.

By Corollary~\ref{basis} and Proposition \ref{rank2}, $\mathrm{rank}_{R(\breve{G})}(K^{\breve{G}}(Z_\ff))=\dim_{R(\breve{G})_{loc}} \widetilde{\Sc}_{\ff,loc}$. So we obtain an isomorphism $\widetilde{\psi}_{loc}:K^{\breve{G}}(Z_\ff)_{loc}\simeq\widetilde{\Sc}_{\ff,loc}$. Consider the following commutative diagram:\begin{equation*}
  \begin{tikzcd}
K^{\breve{G}}(Z_\ff) \arrow[rr, "\widetilde{\psi}", hook] \arrow[d, hook]             &  & \widetilde{\Sc}_\ff \arrow[d, hook] \arrow[r, hook]         & \End_{R(\breve{G})}(\widetilde{\bT}_\ff) \arrow[d, hook]         \\
K^{\breve{G}}(Z_\ff)_{loc} \arrow[rr, phantom] \arrow[rr, "\simeq"] &  & \widetilde{\Sc}_{\ff,loc} \arrow[r, hook] & \End_{R(\breve{G})}(\widetilde{\bT}_\ff)_{loc}
\end{tikzcd}
\end{equation*}
where the right square is a Cartesian diagram clearly.
We will show the biggest square is also Cartesian, which forces the left square to be also Cartesian, and hence $\widetilde{\psi}$ is an isomorphism.
Indeed, suppose that there exists $\mathcal{M} \in (\End_{R(\breve{G})}(\widetilde{\bT}_\ff)\cap K^{\breve{G}}(Z_\ff)_{loc})\backslash K^{\breve{G}}(Z_{\ff})$. Without loss of generality, we may assume that $$\mathcal{M} \in K^{\breve{G}}(Z^{\preceq w}_{\gamma\nu})_{loc} \backslash K^{\breve{G}}(Z^{\prec w}_{\gamma\nu})_{loc},$$
then $\mathcal{M} \in \chi_{\gamma\nu}^{w}/\chi'+K^{\breve{G}}(Z^{\prec w}_{\gamma\nu})_{loc}$ for some $\chi \in R(\breve{P}_{\gamma\nu}^w)$ and $\chi' \in R(\breve{G})$ such that $\chi/\chi' \notin R(\breve{P}_{\gamma\nu}^w)$. We have the following commutative diagram whose right arrow is the natural algebra homomorphism
$$\begin{tikzcd}
    K^{\breve{G}}(T_{\Ob_{\gamma,w,\nu}}^*) \ar[r, "\simeq"] \ar[d, "(- \star e_{\nu,\id}^0)|_{T_{\Ob_{\gamma,\theta_\gamma w_{\gamma\nu}^+}}^*}"'] & R(\breve{P}_{\gamma\nu}^w) \ar[d]\\
    K^{\breve{G}}(T_{\Ob_{\gamma,\theta_\gamma w_{\gamma\nu}^+}}^*) \ar[r, "\simeq"] & R(\breve{T}).
\end{tikzcd}$$
By \cite[Theorem 2.2]{St75}, $R(\breve{T})$ is a free $R(\breve{P}_{\gamma\nu}^w)$-module with a basis $\{f_y ~|~ y \in \W_{\gamma\nu}^w\}$ such that $f_\id=e^0$. If $\chi/\chi' \in R(\breve{T})$, then $\chi/\chi'=\sum_y \chi_y f_y$ with $\chi_y \in R(\breve{P}_{\gamma\nu}^w)$ and $\chi_y \neq 0$ for some $y \neq \id$. So $\chi-\chi'\chi_\id= \sum_{y \neq \id} \chi'\chi_y f_y\in R(\breve{P^w_{\gamma\nu}})$, which is contradictory. Therefore, $\chi/\chi' \notin R(\breve{T})$. Moreover, Lemma \ref{tec lem} implies
$$(\mathcal{M} \star e^0_{\nu,\id})|_{T_{\Ob_{\gamma,\theta_\gamma w_{\gamma\nu}^+}}^*}=\frac{1}{\chi'}\mathcal{O}_{T_{\Ob_{\gamma,\theta_\gamma w_{\gamma\nu}^+}}^*}(\chi) \notin K^{\breve{G}}(T_{\Ob_{\gamma,\theta_\gamma w_{\gamma\nu}^+}}^*),$$
which leads to $\mathcal{M} \star e^{0}_{\nu,\id}\notin\widetilde{\bT}_\ff$ and hence $\mathcal{M} \notin\End_{R(\breve{G})}(\widetilde{\bT}_\ff)$, a contradiction. So $\End_{R(\breve{G})}(\widetilde{\bT}_\ff)\cap K^{\breve{G}}(Z_\ff)_{loc}=K^{\breve{G}}(Z_{\ff})$, which implies that the left square is Cartesian and hence $\widetilde{\psi}$ is an isomorphism.
%Therefore, $$\End_{R(\breve{G})}(\widetilde{\bT}_\ff) \cap (\mathbb{K} \otimes_{R(\breve{G})} K^{\breve{G}}(Z_\ff))=K^{\breve{G}}(Z_\ff).$$ Since $\mathrm{rank}_{R(\breve{G})}(K^{\breve{G}}(Z_\ff))=\dim_\mathbb{K} \mathbb{K} \otimes_{R(\breve{G})} \widetilde{\Sc}_\ff$, we have $\mathbb{K} \otimes_{R(\breve{G})} K^{\breve{G}}(Z_\ff)=\mathbb{K} \otimes_{R(\breve{G})} \widetilde{\Sc}_\ff$. Hence $\widetilde{\Sc}_\ff=\End_{R(\breve{G})}(\widetilde{\bT}_\ff) \cap (\mathbb{K} \otimes_{R(\breve{G})} \widetilde{\Sc}_\ff)=K^{\breve{G}}(Z_\ff)$.
\end{proof}

%=============================================================
\section{Representations of affine $q$-Schur algebras}
%=================================================================
\subsection{Specialized affine $q$-Schur algebras}
Let us take a semisimple element $a=(s,z) \in \breve{G}$, where ``semisimple'' means $s\in G$ is semisimple. Recall in \S\ref{sec:Aaction} the $\breve{G}$-actions on $\mathcal{N}_\gamma$ (hence on $\mathcal{N}$), $\widetilde{\mathcal{N}}_\gamma$, $\widetilde{\mathcal{N}}_\ff$ and $Z_\ff$. We write $\mathcal{N}_\gamma^a$, $\mathcal{N}^a$, $\widetilde{\mathcal{N}}_\gamma^a$, $\widetilde{\mathcal{N}}_\ff^a$ and $Z_\ff^a$ the corresponding $a$-fixed point subvarieties. 

Let $\mathbb{C}_a$ be the complex field regarded as an $R(\breve{G})$-algebra by means of the homomorphism
\begin{align*}
    R(\breve{G}) \to \mathbb{C}_a, \quad \chi \mapsto \chi(a).
\end{align*}
Denote $\widetilde{\Sc}_{\ff,a}= \mathbb{C}_a \otimes_{R(\breve{G})}K^{\breve{G}}(Z_\ff)\simeq\mathbb{C}_a \otimes_{R(\breve{G})} \widetilde{\Sc}_\ff$, which is called the affine $q$-Schur algebra specialized at $a$. 

\begin{lem}\label{6.2}
    For any simple $\widetilde{\Sc}_{\ff}$-module, there exists a semisimple element $a \in \breve{G}$ such that the action of $\widetilde{\Sc}_{\ff}$ factors through an action of $\widetilde{\Sc}_{\ff,a}$. Conversely, any simple module of $\widetilde{\Sc}_{\ff,a}$ is naturally a simple module of $\widetilde{\Sc}_{\ff}$. In particular, each simple $\widetilde{\Sc}_{\ff}$-module is finite dimensional.
\end{lem}
\begin{proof}
    Since $R(\breve{G})$ is the center of $\widetilde{\Sc}_{\ff}$ by Corollary \ref{center of aff schur}, the lemma can be verified by similar arguments to those in \cite[\S 8.1]{CG97}.
\end{proof}

%------------------------
\subsection{Specialized isomorphism}
%---------------------------
Let $\mathcal{A}$ be the closed subgroup of $\breve{G}$ generated by $a$, i.e. the closure of the cyclic group $\langle a\rangle$ in $\breve{G}$. This subgroup $\mathcal{A}$ is an abelian reductive subgroup of $\breve{G}$, and admits $X^a=X^\mathcal{A}$ for any $\mathcal{A}$-variety $X$ such as $X=Z_\ff$, $\widetilde{\mathcal{N}}_\ff$, etc.
%\red{\begin{lem}\label{Kunneth for partial flag variety}
    %For any parabolic subgroup $P\subset G$, the partial flag variety $G/P$ satisfies \cite[5.6.1]{CG97}, that is, for any $G$-variety $X$, we have K\"unneth formula $$K^{G}((G/P)\times X)\simeq K^{G}(G/P)\otimes_{R(G)}K^{G}(X)\simeq R(P)\times_{R(G)}K^{G}(X).$$
%\end{lem}
%\begin{proof}
    %Without loss of generality, we may assume $T\subset P$. Then $\W_P$ the Weyl group of $P$ become a subgroup of $\W$ canonically.
    %By \cite[6.1.22(a)]{CG97}, there is a $\W$-module isomorphism $$R(T)\otimes_{R(G)}K^{G}(X)\simeq K^{T}(X),$$ where the $\W$-action on $R(T)\otimes_{R(G)}K^{G}(X)$ is induced by the $\W$-action on the first factor, while the $\W$-action on right is induced by $N(T)$-action on $K^{T}(X)$. By \cite[6.1.5]{CG97}, $R(T)\simeq \mathbb{Z}[\W]\otimes_{\mathbb{Z}}R(G)$, we have $$R(T)\otimes_{R(G)}K^{G}(X)\simeq\bigoplus\limits_{w\in\W}wK^{G}(X)$$ as left $\W$-modules. Hence $$(K^{T}(X))^{\W_P}\simeq(\mathbb{Z}[\W])^{\W_P}\otimes_{\mathbb{Z}}K^{G}(X)\simeq (R(T))^{\W_P}\otimes_{R(G)}K^{G}(X)\simeq R(P)\otimes_{R(G)}K^{G}(X).$$
    %Let $L_P\supset T$ be a Levi factor of $P$, which is also a reductive group with simply connected derived subgroup.
    %Replace $G$ by $L_P$ in \cite[6.1.22(b)]{CG97}, we have $$K^{L_P}(X)\simeq(K^{T}(X))^{\W_P}\simeq R(P)\otimes_{R(G)}K^{G}(X).$$
    %By induction property,
    %$K^{L_P}(X)\simeq K^{P}(X)\simeq K^{G}(G\times^{P}X)=K^{G}((G/P)\times X)$.
    %Hence we have $K^{G}((G/P)\times X)\simeq R(P)\otimes_{R(G)}K^{G}(X)\simeq K^{G}(G/P)\otimes_{R(G)}K^{G}(X)$.
    %\end{proof}}

\begin{lem}\label{lem:isoraka}
    The natural morphism $R(A)\otimes_{R(\breve{G})}K^{\breve{G}}(Z_\ff)\to K^{\mathcal{A}}(Z_\ff)$ is an isomorphism.
\end{lem}
\begin{proof}
    Without loss of generality, we may assume that $\mathcal{A} \subset \breve{T}$. Then $$R(\mathcal{A})\otimes_{R(\breve{G})}K^{\breve{G}}(Z_\ff) \simeq R(\mathcal{A})\otimes_{R(\breve{T})}K^{\breve{T}}(Z_\ff) \simeq K^{\mathcal{A}}(Z_\ff),$$
    where the first isomorphism is obtained by \cite[Theorem 6.1.22]{CG97} and the second isomorphism is due to the cellular fibration lemma \ref{lem. cfl}.
\end{proof}

Lemma~\ref{lem:isoraka} brings us the following algebra isomorphism
\begin{equation}\label{eq:AA}
\widetilde{\Sc}_{\ff,a}=\mathbb{C}_a \otimes_{R(\breve{G})}K^{\breve{G}}(Z_\ff)\xrightarrow{\sim}\C_a\otimes_{R(\mathcal{A})} K^{\mathcal{A}}(Z_\ff).
\end{equation}
A ``bivariant version'' of the Localization Theorem (cf. \cite[Theorem~5.11.10]{CG97}) and \cite[(8.1.6)]{CG97} shows that there is an algebra isomorphism 
\begin{equation}\label{eq:ra}
\mathrm{r}_a: \C_a\otimes_{R(\mathcal{A})} K^{\mathcal{A}}(Z_\ff) \xrightarrow{\sim} \C_a\otimes_{R(\mathcal{A})} K^{\mathcal{A}}(Z_\ff^\mathcal{A}) 
%\xrightarrow{\sim} K_\mathbb{C}(Z_\ff^{\mathcal{A}}).
\end{equation}
since $\mathcal{A}$ acts trivially on $Z_\ff^{\mathcal{A}}$, there is a canonical identity (cf. \cite[(5.2.4)]{CG97}):
$$K^{\mathcal{A}}(Z_\ff^{\mathcal{A}})=R(\mathcal{A})\otimes_\Z K(Z_\ff^{\mathcal{A}}),$$
by which we have an evaluation isomorphism 
\begin{align}
\label{eq:ev}
\mathrm{ev}: \C_a\otimes_{R(\mathcal{A})} K^{\mathcal{A}}(Z_\ff^\mathcal{A})= \C_a\otimes_{R(\mathcal{A})} (R(\mathcal{A})\otimes_\Z K(Z_\ff^{\mathcal{A}}))&\xrightarrow{\sim} \C\otimes_\Z K(Z_\ff^{\mathcal{A}}),\\\nonumber
1\otimes f\otimes \mathcal{F} &\mapsto f(a)\otimes \mathcal{F}.
\end{align}
The bivariant Riemann-Roch theorem \cite[Theorem~5.11.11]{CG97} implies an algebra isomorphism (cf. \cite[(8.1.6)]{CG97})
\begin{equation}\label{eq:RR}
\mathrm{RR}:\C\otimes_\Z K(Z_\ff^{\mathcal{A}})\xrightarrow{\sim} H_\bullet(Z^{\mathcal{A}}_{\ff},\mathbb{C})=H_\bullet(Z^a_{\ff},\mathbb{C}).
\end{equation}

Combining \eqref{eq:AA},\eqref{eq:ra}, \eqref{eq:ev} and \eqref{eq:RR}, we arrive at 
\begin{prop}\label{localization mor}
    There exists an algebra isomorphism
    $$\widetilde{\Sc}_{\ff,a}\simeq H_\bullet(Z^{a}_{\ff},\mathbb{C}).$$
\end{prop}

%------------------------
\subsection{Restricted map}
%------------------------
The map $\pi_\ff: \widetilde{\mathcal{N}}_{\ff}\to\mathcal{N}$ commutes with the action of $a$, by which we get a restricted map 
$$\pi_\ff^a:\widetilde{\mathcal{N}}_{\ff}^{a}\to\mathcal{N}^a, \quad (gP_\gamma, x)\mapsto x.$$

The  $a$-fixed point subvariety $\widetilde{\mathcal{N}}_{\ff}^{a}\subset\widetilde{\mathcal{N}}_{\ff}$ is smooth (cf. \cite[Lemma~5.11.1]{CG97}). Moreover, it is also a closed subvariety because of $\widetilde{\mathcal{N}}_{\ff}^{a}\simeq \mathrm{graph}(a)\cap\Delta\widetilde{\mathcal{N}}_{\ff}$. Let $G(s)$ be the centralizer of $s$ in $G$ (recall $a=(s,z)$ with $s$ being semisimple). Then $\pi_\ff^a$ becomes a $G(s)$-equivariant projective morphism.  Let $D^b(\widetilde{\mathcal{N}}^a_\ff)$ (resp. $D^{b}(\mathcal{N}^a)$) be the bounded derived category of complexes of sheaves with constructible cohomology sheaves on $\widetilde{\mathcal{N}}^a_\ff$ (resp. $\mathcal{N}^a$). Define a complex on $\widetilde{\mathcal{N}}_{\ff}$ by $$\mathcal{C}_{\ff,a}|_{\widetilde{\mathcal{N}}^{a}_{\gamma}}=\underline{\mathbb{C}}_{\widetilde{\mathcal{N}}^{a}_{\gamma}}[\dim_{\mathbb{C}}(\widetilde{\mathcal{N}}^{a}_{\gamma})]\quad\mbox{for}\ \gamma\in\Lambda_{\ff}.$$ 
By \cite[\S 5.4]{KL87}, the variety $\mathcal{N}^a$ has finite many $G(s)$-orbits. Hence the orbits forms the stratification $\mathcal{N}^a=\bigsqcup \mathcal{O}$. Then by the equivariant version of BBDG decomposition theorem (cf. \cite[Theorem~8.4.12]{CG97}), we have $$\pi_{\ff*}^a\mathcal{C}_{\ff,a}=\bigoplus\limits_{i\in\mathbb{Z},\phi=(\mathcal{O}_{\phi},\chi_\phi)}L_{\phi}(i)\otimes IC_{\phi}[i],$$
where $\phi$ runs over all pairs $(\mathcal{O}_{\phi},\chi_\phi)$ consists of a $G(s)$-orbit $\mathcal{O}_{\phi}\subset\mathcal{N}^{a}$ and an irreducible equivariant local system $\chi_\phi$ on $\mathcal{O}_{\phi}$; $IC_{\phi}$ is the intersection cohomology complex associated with $\phi$, $[i]$ stands for the shift in the derived category and $L_{\phi}(i)$ are certain finite dimensional vector spaces.
Let $L_{\phi}:=\bigoplus_{i}L_{\phi}(i)$. Then we have algebra isomorphisms (not necessarily preserving the grading):
\begin{align*}
H_\bullet(Z^{a}_{\ff},\C)&\simeq\mathrm{Ext}^{\bullet}_{{D}^{b}(\mathcal{N}^a)}(\mu_*\mathcal{C}_{\ff,a},\mu_*\mathcal{C}_{\ff,a})\\&\simeq(\bigoplus\limits_{\phi}\End L_{\phi})\oplus(\bigoplus\limits_{\phi,\psi,i>0}\Hom_{\mathbb{C}}(L_\phi,L_\psi)\otimes\mathrm{Ext}^{i}_{{D}^{b}(\mathcal{N}^a)}(IC_{\phi},IC_{\psi})) .
\end{align*}

   \begin{lem}\label{lem 6.3}
      For any $x\in\mathcal{O}_{\phi}\subset\mathcal{N}^{a}\cap\mathrm{Im}(\pi_{\ff})$, the preimage $(\pi_\ff^a)^{-1}(x)$ is a non-empty projective variety.
   \end{lem}
   \begin{proof}
      The preimage is projective since $\pi_\ff^a$ is a projective morphism.
      Since $x\in\mathrm{Im}(\pi_{\ff})$, we may assume that $\pi^{-1}_{\gamma}(x)$ is non-empty for some $\gamma\in\Lambda_\ff$. By \cite[8.1.7]{CG97}, there is a Borel subgroup of $G$ fixed by $s$ (and hence containing $s$) such that $x\in\mathrm{Lie}(B)$. So there is a unique parabolic subgroup $P_{\gamma}\in\sF_{\gamma}$ which contains $B$ such that $x\in\mathfrak{n}_{\gamma}$, where $\mathfrak{n}_{\gamma}$ is the Lie algebra of the nilpotent radical of $P_{\gamma}$. Then $P_{\gamma}$ is fixed by $s$, since $s$ is semisimple and hence lies in a Levi factor of $P_{\gamma}$. The lemma follows.
   \end{proof}

%-------------------------------
\subsection{Classification of irreducible modules}
%-------------------------------- 

\begin{thm}[weak version]\label{weak version of rep} Each
$L_{\phi}$ is a finite dimensional irreducible representation of $\widetilde{\Sc}_{\ff}$, provided it is non-zero. Furthermore, such $L_{\phi}$ exhausts all finite dimensional irreducible representations of $\widetilde{\Sc}_{\ff}$. 
   \end{thm}
   \begin{proof}
       This result follows by Lemma \ref{6.2}, Lemma \ref{localization mor} and a general construction of representations of convolution algebras (see \cite[\S 8.6-8.7]{CG97}).
   \end{proof}

    Fix a local transversal slice $S$ to $\mathcal{O}_{\phi}$ at $x$. Let $\psi: H_{\bullet}((\pi_\ff^a)^{-1}(x))\to H_\bullet((\pi_\ff^a)^{-1}(S))$ be the direct image map on Borel-Moore homology that induced by the natural closed embedding. By Lemma \ref{lem 6.3}, $H_{\bullet}((\pi_\ff^a)^{-1}(x))\neq 0$. Denote $L_{a,x}:=\mathrm{Im}(\psi)$ the image of $\psi$. Let $G(s)_x$ be the isotropic subgroup of $x$ in $G(s)$, $C(s,x)$ the component group of $G(s)_x$. Let $L_{a,x,\chi}$ be the $\chi$-isotypical component of $L_{a,x}$ for $\chi\in \mathrm{Irr}(C(s,x))$.
   The next proposition is due to \cite[Propositions~8.5.16 \& 8.6.21]{CG97}.
   \begin{prop}\label{6.5}
       Let $x\in\mathcal{O}_{\phi}\subset\mathcal{N}^a\cap\mathrm{Im}(\pi_{\ff})$ and $\chi$ an irreducible representation of $C(s,x)$. Then there is an natural isomorphism of $H_{\bullet}(Z^{a}_{\ff})$-modules $L_\phi\simeq L_{a,x,\chi}$.
   \end{prop}

    \begin{prop}\label{6.6}
       Suppose  $\Lambda_{\ff}$ contains a regular $\W$-orbit. If $z$ is not a root of unity, then the vector space $L_{a,x,\chi}$ are non-zero for every $x\in \mathcal{N}^a$ and any irreducible representation $\chi$ of $C(s,x)$ that occurs in the isotypical decomposition of $H_\bullet(\mu^{-1}(x))$.
    \end{prop}
    \begin{proof}
        Since $\Lambda_{\ff}$ contains a regular $\W$-orbit, a result similar to \cite[Theorem~8.8.1]{CG97} holds after replacing $\widetilde{\mathcal{N}}^a$ by $\widetilde{\mathcal{N}}_{\ff}^{a}$. Then by similar arguments as in \cite[\S 8.8]{CG97}, the proposition follows.
    \end{proof}
    Write $\mathbf{M}$ for the set of $G$-conjugacy classes $[a,x,\chi]$ of the triple data$$\left\{(a,x,\chi)\middle| \begin{array}{c}a=(s,z)\in \breve{G},x\in\mathcal{N}^a,\chi\in\mathrm{Irr}(C(s,x)), \\
    \text{$s$ is semisimple and $z$ is not a root of unity}
    \end{array}\right\}{\big/}\mathrm{Ad}G.$$
   \begin{thm}\label{rep aff}
   If $\Lambda_{\ff}$ contains a regular $\W$-orbit, then the set $\{L_{a,x,\chi}~|~[a,x,\chi]\in\mathrm{M}\}$ is a complete collection of simple $\widetilde{\Sc}_{\ff}$-modules such that $q$ acts by means of multiplication by $z$.
   \end{thm}
   \begin{proof}
       It follows from Theorem \ref{weak version of rep}, Proposition \ref{6.5} and Proposition~\ref{6.6}.
   \end{proof}

%%=====================================
\section{Applications}
%%======================================

%----------------------
\subsection{Specializations}
%-----------------------
Though the choices of $Q_\ff$ are flexible and far from being unique,
here we list some natural choices from which we can rediscover some known results in the literature.  

\subsubsection{Type $A_{d-1}$} Let $G=\mathrm{GL}_d(\C)$.
We take $Q=\sum_{i=1}^d\Z\epsilon_i$ the weight lattice, whose $\W$-action is given by the permutation of $\epsilon_i$ ($i=1,2,\ldots,d$). For any positive integer $n$, we can choose $Q_\ff$ to be 
\begin{equation}\label{def:Qn}
Q_{n}:=\{\sum_{i=1}^d a_i\epsilon_i~|~a_i\in\Z, 0\leq a_i\leq n-1, \forall i\}.
\end{equation}
The set $\Lambda_\ff$ of $\W$-orbits in $Q_n$ can be regarded as the set of weak compositions of $d$ into $n$ parts 
$$\Lambda_{n,d}=\{\gamma=(\gamma_0,\gamma_1,\ldots,\gamma_{n-1})~|~\sum_{i=0}^{n-1}\gamma_i=d, \gamma_i\in\mathbb{Z}_{\geq0}\},$$ where an orbit $\gamma=(\gamma_0,\gamma_1,\ldots,\gamma_{n-1})\in\Lambda_{n,d}$ consists of all weights $\sum_{i=1}^d a_i\epsilon_i$ such that $\gamma_k=\#\{1\leq i\leq d~|~a_i=k\}$, $(k=0,1,\ldots,n-1)$. 
Then the Fock space $\widetilde{\mathbf{T}}_\ff=\bigoplus_{\gamma\in\Lambda_{n,d} }\mathbf{x}_\gamma\widetilde{\mathbf{H}}$ (rewritten by $\widetilde{\mathbf{T}}_{n,d}$) coincides with the one introduced in \cite{VV99}, where they used an anti-dominant weight to label an orbit $\gamma\in\Lambda_{n,d}$ instead. Hence in this special case, 
\begin{equation}\label{realize:A}
\widetilde{\mathbf{S}}_\ff=\widetilde{\mathbf{S}}_{n,d}
\end{equation}
is the usual $q$-Schur algebra of affine type A studied in \cite{GV93, VV99, Lu99, DF15}, etc. 

For $\gamma=(\gamma_0,\gamma_1,\cdots,\gamma_{n-1})\in\Lambda_{n,d}$, the associated parabolic subgroup $\W_\gamma$ is generated by $$J_\gamma=\{s_i~|~1\leq i\leq d-1, i\neq\gamma_0, \gamma_0+\gamma_1,\ldots,\gamma_0+\cdots+\gamma_{n-1}\}.$$
Fix a basis $\{v_1,\ldots,v_d\}$ of the standard $G$-module $\C^d$, and set $W_i=\langle v_1,\ldots,v_i\rangle$. The parabolic subgroup $P_\gamma$ consists of the elements stabilizing the flag
$$\mathfrak{f}_\gamma=(0\subset W_{\gamma_0}\subset W_{\gamma_0+\gamma_1}\subset\ldots\subset W_d=\C^d).$$
In this case, the partial flag variety $\sF_\ff=\bigsqcup_{\gamma\in\Lambda_{n,d}}G/P_\gamma$ is isomorphic to the $n$-step flag variety   
$$\sF_{n,d}=\{\mathfrak{f}=(0=V_0\subset V_1\subset \ldots\subset V_n=\C^d)\}.$$ The isomorphism is given by $gP_\gamma\mapsto g\mathfrak{f}_\gamma$ (cf. \cite[Lemma~3.1]{LW22}).

\subsubsection{Type $B_d$}\label{sub:BC}
Let $G=\mathrm{SO}_{2d+1}(\C)$ whose standard module $\C^{2d+1}$ is equipped with a non-degenerate symmetric bilinear form $$(v_i,v_j)=\delta_{i,-j}, \quad (\forall -d\leq i,j\leq d)$$ for a given basis $\{v_{-d},\ldots,v_d\}$.
We rewrite the simple reflections of $\W$ by $s_0,s_1,\ldots,s_{d-1}$. Take $Q=Q^0\oplus Q^{\frac{1}{2}}$ with $Q^0=\sum_{i=1}^d\Z\epsilon_i$ and $Q^{\frac{1}{2}}=\sum_{i=1}^d(\frac{1}{2}+\Z)\epsilon_i$, whose $\W$-action is given by that $s_0$ switches $\epsilon_k$ and $\epsilon_{-k}$ ($\forall k=1,\ldots d$) while $s_i$ $(i\neq0)$ switches $\epsilon_{\pm i}$ and $\epsilon_{\pm (i+1)}$. For any positive integer $n$, we can choose $Q_\ff$ to be
\begin{align*}
  Q_n^\jmath&:=\{\sum_{i=1}^da_i\epsilon_i~|~a_i\in\Z,-n\leq a_i\leq n, \forall i\},\quad\mbox{or}\\
 Q_n^{\imath}&:=\{\sum_{i=1}^da_i\epsilon_i~|~a_i\in\frac{1}{2}+\Z,-n\leq a_i\leq n, \forall i\}.
\end{align*}

 The sets $\Lambda_\ff$ of $\W$-orbits in $Q_n^\jmath$ and $ Q_n^{\imath}$ can be regarded as 
 \begin{align*}\Lambda_{n+1,d}&=\{\gamma=(\gamma_0,\gamma_1,\ldots,\gamma_n)~|~\sum_{i=0}^n\gamma_i=d,\gamma_i\in\mathbb{Z}_{\geq0}\}\quad\mbox{and}\\
 \Lambda_{n,d}&=\{\gamma=(\gamma_1,\gamma_2,\ldots,\gamma_n)~|~\sum_{i=1}^n\gamma_i=d, \gamma_i\in\mathbb{Z}_{\geq0}\},\quad\mbox{respectively}.
 \end{align*}
Here an orbit $\gamma=(\gamma_0,\gamma_1,\gamma_2,\ldots,\gamma_n)\in\Lambda_{n+1,d}$ consists of all weights $\sum_{i=1}^d a_i\epsilon_i\in Q_n^{\jmath}$ such that $\gamma_k=\#\{1\leq i\leq d~|~|a_i|=k\}$, $(k=0,1,2,\ldots,n)$, while an orbit $\gamma=(\gamma_1,\gamma_2,\ldots,\gamma_n)\in\Lambda_{n,d}$ consists of all weights $\sum_{i=1}^d a_i\epsilon_i\in Q_n^{\imath}$ such that $\gamma_k=\#\{1\leq i\leq d~|~|a_i|=k-\frac{1}{2}\}$, $(k=1,2,\ldots,n)$. They coincide with $\Lambda^{\imath\jmath}$ and $\Lambda^{\imath\imath}$ in \cite[\S8.4 \& \S8.5]{FLLLW23}, respectively. Thus in this special case, the affine Schur algebra 
\begin{equation}\label{realize:B}
\widetilde{\mathbf{S}}_\ff=\mathbf{S}_{2n+1,d}^{\imath\jmath}\quad\mbox{or}\quad \mathbf{S}_{2n,d}^{\imath\imath}
\end{equation}
is just the affine $\imath$Schur algebra introduced in \cite{FLLLW23}. 

Consider the case of $Q_\ff=Q_n^{\jmath}$. 
For $\gamma=(\gamma_0,\gamma_1,\cdots,\gamma_{n})\in\Lambda_{n+1,d}$, the associated parabolic subgroup $\W_\gamma$ is generated by $$J_\gamma=\{s_i~|~0\leq i\leq d-1, i\neq\gamma_0, \gamma_0+\gamma_1,\ldots,\gamma_0+\cdots+\gamma_{n}\}.$$
Set $W_{i+\frac{1}{2}}=\langle v_{-d},\ldots,v_i\rangle$. The parabolic subgroup $P_\gamma$ consists of the elements stabilizing the flag
$$\mathfrak{f}_\gamma=(0= W_{-d-\frac{1}{2}}\subset W_{-d+\gamma_{n}-\frac{1}{2}}\subset W_{-d+\gamma_n+\gamma_{n-1}-\frac{1}{2}}\subset\ldots\subset W_{d-\gamma_{n}+\frac{1}{2}}\subset W_{d+\frac{1}{2}}=\C^{2d+1})$$ with $W_i=W_j^{\perp}$ if $i+j=0$. 
In this case, the partial flag variety $\sF_\ff$ is isomorphic to the flag variety 
$$\sF_{n,d}^{B,\jmath}=\{\mathfrak{f}=(0=V_{-n-\frac{1}{2}}\subset V_{-n+\frac{1}{2}}\subset\ldots\subset V_{n+\frac{1}{2}}=\C^{2d+1})~|~V_i=V_j^{\perp}\ \mbox{if}\ i+j=0\},$$ also given by $gP_\gamma\mapsto g\mathfrak{f}_\gamma$ (cf. \cite[Lemma~4.1]{LW22}).  

Consider the case of $Q_\ff=Q_n^{\imath}$.
For $\gamma=(\gamma_1,\cdots,\gamma_{n})\in\Lambda_{n,d}$, we regard it as $(\gamma_0, \gamma_1,\cdots,\gamma_{n})\in\Lambda_{n+1,d}$ with $\gamma_0=0$. The corresponding $\W_\gamma$, $P_\gamma$ and $\mathfrak{f}_\gamma$ are all as the same as the ones given previously. In this case, the partial flag variety $\sF_\ff$ is isomorphic to the flag variety 
$$\sF_{n,d}^{B,\imath}=\{\mathfrak{f}\in \sF_{n,d}^{B,\jmath}~|~\dim V_{-\frac{1}{2}}=\dim V_{\frac{1}{2}}-1\}\subset \sF_{n,d}^{B,\jmath}.$$

%---------------------
\subsubsection{Type $C_d$} 
Let $G=\mathrm{Sp}_{2d}(\C)$ be the symplectic group, whose standard module $\C^{2d}$ is equipped with a non-degenerate skew-symmetric bilinear form $$(v_i,v_j)=\mathrm{sign}(i)\delta_{i,-j}, \quad (\forall -d<i,j<d)$$ for a given basis $\{v_{-d+\frac{1}{2}},\ldots,v_{d+\frac{1}{2}}\}$. We also take $Q_\ff=Q_n^{\jmath}$ or $Q_n^{\imath}$ and hence the same associated parabolic subgroup $\W_\gamma$ for $\gamma\in \Lambda_{n+1,d}$ or $\Lambda_{n,d}$.  

Set $W_{i}=\langle v_{-d+\frac{1}{2}},\ldots,v_{i-\frac{1}{2}}\rangle$. The parabolic subgroup $P_\gamma$ consists of the elements stabilizing the flag
$$\mathfrak{f}_\gamma=(0= W_{-d}\subset W_{-d+\gamma_{n}}\subset W_{-d+\gamma_n+\gamma_{n-1}}\subset\ldots\subset W_{d-\gamma_{n}}\subset W_{d}=\C^{2d})$$ with $W_i=W_j^{\perp}$ if $i+j=0$.

In the case of $Q_\ff=Q_n^{\jmath}$, the partial flag variety 
$$\sF_\ff\simeq\sF_{n,d}^{C,\jmath}=\{\mathfrak{f}=(0=V_{-n-\frac{1}{2}}\subset V_{-n+\frac{1}{2}}\subset\ldots\subset V_{n+\frac{1}{2}}=\C^{2d})~|~V_i=V_j^{\perp}\ \mbox{if}\ i+j=0\},$$ while in the case of $Q_\ff=Q_n^{\imath}$,  
$$\sF_\ff\simeq\sF_{n,d}^{C,\imath}=\{\mathfrak{f}\in \sF_{n,d}^{C,\jmath}~|~ V_{-\frac{1}{2}}= V_{\frac{1}{2}}\}\subset \sF_{n,d}^{C,\jmath},$$
both of which are given by $g P_\gamma\mapsto g\mathfrak{f}_\gamma$ (cf. \cite[Lemma~4.4]{LX22}).

%-------------------------------------
\subsection{Application I: Local geometric Langlands correspondence}\label{subsec:app1}
%--------------------------------------
 Let $\mathbb{F}_p$ be a finite field of $p$ elements, and $\mathbb{F}_p((\epsilon))$ the field of formal Laurent series over $\mathbb{F}_p$. Let $G(\mathbb{F}_p((\epsilon)))$ be the associated algebraic group defined over $\mathbb{F}_p((\epsilon))$.

Recall in Remark~\ref{rem:schur} that the affine $q$-Schur algebras $\widetilde{\Sc}_\ff$ defined in this paper coincide with the ones in \cite{CLW24}, up to a Langlands dual. 
A Beilinson-Lusztig-MacPherson type realization of affine $q$-Schur algebras was achieved in \cite[Theorem~3.2]{CLW24}, saying that the affine $q$-Schur algebras can be realized as the convolution algebra $$\C\left[\bigsqcup_{\gamma,\nu\in\Lambda_\ff}\mathcal{I}_\gamma\backslash G(\mathbb{F}_p((\epsilon)))/\mathcal{I}_\nu\right],$$
where $\mathcal{I}_\gamma=\bigsqcup_{w\in \W_\gamma}\mathcal{I}w\mathcal{I}$ (resp. $\mathcal{I}_\nu=\bigsqcup_{w\in \W_\nu}\mathcal{I}w\mathcal{I}$) is the parahoric subgroup of $G(\mathbb{F}_p((\epsilon)))$ associated with $\gamma$ (resp. $\nu$). Here $\mathcal{I}$ is the Iwahori subgroup of $G(\mathbb{F}_p((\epsilon)))$.
Thus using affine $q$-Schur algebras as intermediaries, Theorem~\ref{geo affine Schur-Weyl duality} brings us a Schur algebra analogue of the local geometric Langlands correspondence as follows.
\begin{thm}\label{thm:langlands}
Let $G$ be a connected split reductive algebraic group whose Langlands dual group is as in \S\ref{sec:Aaction}. Let $\W$ be the Weyl group of $G$. For any $\W$-invariant finite subset $Q_\ff$ of the coweight lattice of $G$, there always exists an isomorphism of algebras: 
\begin{equation}\label{LanglandsforSchur}
\C\left[\bigsqcup_{\gamma,\nu\in\Lambda_\ff}\mathcal{I}_\gamma\backslash G(\mathbb{F}_p((\epsilon)))/\mathcal{I}_\nu\right]\simeq K^{{}^LG\times\mathbb{C}\backslash\{0\}}({}^LZ_\ff)|_{p=q^{-2}},
\end{equation}
where $^{L}G$ is the Langlands dual of $G$, and $^{L}Z_\ff$ is the associated Steinberg variety of ${}^LG$ defined by the same $\W$-invariant finite subset $Q_\ff$ of the weight lattice of $^{L}G$.
\end{thm}

A specialization of this reciprocity, by taking $G=\mathrm{GL}_d(\C)$ (which is self dual) and $Q_\ff=Q_n$ as in \eqref{def:Qn}, seems known to experts. That is, the affine $q$-Schur algebra constructed in \cite{Lu99} via a Beilinson-Lusztig-MacPherson type realization and the one constructed in \cite{GV93} via equivariant K-theory are the same. The reciprocity for $G=\mathrm{Sp}_{2d}(\C)/\{\pm I_{2d}\}$, whose Langlands dual group is $\mathrm{Spin}_{2d+1}(\C)$, is new.

Moreover, if we take $Q_\ff$ to be a single regular $\W$-orbit, then the above reciprocity degenerates to the original local geometric Langlands correspondence about affine Hecke algebras:
$$\C\left[\mathcal{I}\backslash G(\mathbb{F}_p((\epsilon)))/\mathcal{I}\right]\simeq K^{{}^LG\times\mathbb{C}\backslash\{0\}}({}^LZ)|_{p=q^{-2}},$$
which has a categorification version provided by Bezrukavnikov \cite{Be16}. We shall try to give a similar categorification of the reciprocity \eqref{LanglandsforSchur} for $\widetilde{\Sc}_\ff$ in a subsequent work.

%-----------------------------
\subsection{Application II: Realization of affine $\imath$quantum groups}
%-----------------------------
\subsubsection{Affine quantum groups $\mathbf{U}(\widetilde{\mathfrak{sl}}_n)$ and $\mathbf{U}(\widetilde{\mathfrak{gl}}_n)$}

By a Belinson-Lusztig-MacPherson type construction, multiplication formulas for simple generators of the affine $q$-Schur algebra $\widetilde{\mathbf{S}}_{n,d}$ were shown in \cite{Lu99}, by which a stabilization property is derived to construct $\mathbf{U}(\widetilde{\mathfrak{sl}}_n)$ in the inverse limit $\lim\limits_{\overleftarrow{d}}\widetilde{\mathbf{S}}_{n,d}$. We caution that when $n\leq d$, the simple generators are not enough to generate the whole $\widetilde{\mathbf{S}}_{n,d}$. See \cite{DF15} for the construction of $\mathbf{U}(\widetilde{\mathfrak{gl}}_n)$ based on multiplication formulas of semisimple generators of $\widetilde{\mathbf{S}}_{n,d}$. 

Employing the multiplication formulas in \cite{DF15}, we can also further realize the affine quantum groups $\mathbf{U}(\widetilde{\mathfrak{gl}}_n)$ and $\mathbf{U}(\widetilde{\mathfrak{sl}}_n)$ based on our equivariant K-theoretic realization for affine Schur algebra $\widetilde{\mathbf{S}}_{n,d}$ explained in \eqref{realize:A}. We remark here that this special (but most important) case had been achieved in \cite{GV93, Va98} over the base field $\mathbb{C}(q)$, while our construction of $\widetilde{\mathbf{S}}_{n,d}$ in this present paper is over $\mathbb{Z}[q,q^{-1}]$.

The above two different approaches to $\mathbf{U}(\widetilde{\mathfrak{gl}}_n)$ and $\mathbf{U}(\widetilde{\mathfrak{sl}}_n)$ are collectively referred to as the Langlands reciprocity for affine quantum groups of type A.

%By this realization, it is clear that $\widetilde{\mathbf{S}}_{n,d}$ is a quotient of $\mathbf{U}(\widetilde{\mathfrak{gl}}_n)$.
 
 %Set $$\mathscr{E}_{i}(k)=\sum\limits_{\gamma\in\Lambda_{n,d-1}}(-q)^{\gamma_i}[\mathscr{E}_{i,\gamma}(k)],\quad\mathscr{F}_{i}(k)=\sum\limits_{\gamma\in\Lambda_{n,d-1}}(-q)^{\gamma_{i+1}}[\mathscr{F}_{i,\gamma}(k)]$$}
 %Define $$\mathscr{E}_{i}(z)=\sum\limits_{k\in\mathbb{Z}}\mathscr{E}_{i}(k)z^{-k},\quad\mathscr{F}_{i}(z)=\sum\limits_{k\in\mathbb{Z}}\mathscr{F}_{i}(k)z^{-k}\in K^{\breve{G}}(Z_\ff)[[z,z^{-1}]]$$}

%------------------
\subsubsection{Affine $\imath$quantum groups of type AIII}
Denote $_{\mathbb{Q}}\widetilde{\Sc}_{\ff}:=\mathbb{Q}(q)\otimes_{\mathbb{Z}[q,q^{-1}]}\widetilde{\Sc}_{\ff}$.
We have a similar story for quasi-split $\imath$quantum groups of affine type AIII (over $\mathbb{Q}(q)$) based on the realization of $_{\mathbb{Q}}{\mathbf{S}}_{2n+1,d}^{\imath\jmath}$ and $_{\mathbb{Q}}{\mathbf{S}}_{2n,d}^{\imath\imath}$ as shown in \eqref{realize:B}. Here we describe the results on $\mathbb{Q}(q)$ (instead of $\mathbb{Z}[q,q^{-1}]$) for safety, since there is no integral form for affine $\imath$quantum groups yet.

Although we have not derived multiplication formulas for generators of $_{\mathbb{Q}}{\mathbf{S}}_{2n+1,d}^{\imath\jmath}$ and $_{\mathbb{Q}}{\mathbf{S}}_{2n,d}^{\imath\imath}$ from the equivariant K-theoretic construction; thanks to the Langlands reciprocity displayed in Theorem~\ref{thm:langlands}, we can employed the multiplication formulas obtained in \cite{FLLLW23} for $_{\mathbb{Q}}{\mathbf{S}}_{2n+1,d}^{\imath\jmath}$ and $_{\mathbb{Q}}{\mathbf{S}}_{2n,d}^{\imath\imath}$ to further realize affine $\imath$quantum groups $\mathbf{K}_{2n+1}^{\imath\jmath}$ and $\mathbf{K}_{2n}^{\imath\imath}$ introduced in \cite{FLLLW20} via a stabilization construction originated from \cite{BLM90}. These affine $\imath$quantum groups form affine quantum symmetric pairs $(\mathbf{U}(\widetilde{\mathfrak{gl}}_{2n+1}),\mathbf{K}^{\imath\jmath}_{2n+1})$ and $(\mathbf{U}(\widetilde{\mathfrak{gl}}_{2n}),\mathbf{K}^{\imath\imath}_{2n})$ in the sense of \cite{Ko14}. A direct verification of the equivariant K-theoretic construction of $\mathbf{U}^{\imath\imath}_{2n}$ has been given in \cite{SW24}. %\red{(they use type C geometry, but it is more appropriate to use type B geometry)}.

%-----------------------------------------------------
\subsubsection{Representations of $\imath$quantum groups}
%-----------------------------------------------------
Each irreducible representation of $\widetilde{\Sc}_{\ff}$ can be extended naturally to an irreducible representation of $_{\mathbb{Q}}\widetilde{\Sc}_{\ff}$. The representations given in Theorem \ref{rep aff} exhausts all irreducible $_{\mathbb{Q}}\widetilde{\Sc}_{\ff}$-representations on which $q$ acts as not a root of unity.

For quasi-split $\imath$quantum groups of affine type AIII, one has surjective morphisms:
$$\Psi^{\imath\jmath}_{d}:\mathbf{K}^{\imath\jmath}_{2n+1}\twoheadrightarrow{_{\mathbb{Q}}\mathbf{S}}_{2n+1,d}^{\imath\jmath}\quad\text{and}\quad\Psi^{\imath\imath}_{d}:\mathbf{K}^{\imath\imath}_{2n}\twoheadrightarrow{_{\mathbb{Q}}\mathbf{S}}_{2n,d}^{\imath\imath}$$ by \cite[\S~8.3]{FLLLW23}. Hence the irreducible representations of $\mathbf{S}_{2n+1,d}^{\imath\jmath}$ and $\mathbf{S}_{2n,d}^{\imath\imath}$ obtained in Theorem \ref{rep aff} can be regarded as irreducible representations of $\mathbf{K}^{\imath\jmath}_{2n+1}$ and $\mathbf{K}^{\imath\imath}_{2n}$, respectively.

%---------------------------------------
\subsection{Application III: Geometric Howe duality}
%----------------------------------------
\subsubsection{General case}
%-----------------------------
Let $Q_\g\subset Q$ be another finite $\W$-invariant subset. We shall use the subscript $\g$ to indicate the notions associated to $Q_\g$, such as $\widetilde{\Sc}_\g$, $\widetilde{\mathcal{N}}_\g$, etc. Denote $\widetilde{\bT}_{\ff\g}:=K^{\breve{G}}(\widetilde{\mathcal{N}}_\ff\times_\mathcal{N} \widetilde{\mathcal{N}}_\g)$, which admits a left $\widetilde{\Sc}_\ff$-action and a right $\widetilde{\Sc}_\g$-action under the convolution product. 
%It follows from Corollary \ref{lem:chi} immediately that $\widetilde{\mathbf{T}}_{\ff\g}$ is a free $R(\breve{G})$-module with a basis $$\{\chi_{\gamma\nu}^{w}\mid(\gamma,\nu)\in\Lambda_\ff\times\Lambda_\g,w\in\mathcal{D}_{\gamma\nu},\chi\in\mathcal{B}(P^{w}_{\gamma\nu})\}.$$

\begin{lem}\label{lem:howe}
    If for each $\nu \in \Lambda_\ff$, there exists an orbit $\mu \in \Lambda_\g$ such that $P_\mu \subset P_\nu$, then $$\widetilde{\Sc}_\ff \cong \End_{\widetilde{\Sc}_\g}(\widetilde{\bT}_{\ff\g}).$$
\end{lem}

\begin{proof}
    To be not tedious, similar to the proof of Theorem \ref{geo affine Schur-Weyl duality} and by Lemma \ref{lem2}, one can see that if $P_\mu \subset P_\nu$, then
    \begin{align*}
        \chi_{\gamma\nu}^{w} \star e_{\nu, \id, \mu}^0\in\chi_{\gamma, w', \mu}+K^{\breve{G}}(Z^{\prec w}_{\gamma\mu}), \quad \forall w \in \D_{\gamma\nu} \ \mbox{and} \ \chi \in R(P_{\gamma\nu}^w),
    \end{align*}
    where $w'$ is the longest element in $\W_\gamma w \W_\nu \cap \D_{\gamma\mu}$. Therefore,  $\widetilde{\Sc}_\ff \to \End_{\widetilde{\Sc}_\g}(\widetilde{\bT}_{\ff\g})$ is injective. Similar to the argument of Theorem \ref{geo affine Schur-Weyl duality}, we have $$\End_{\widetilde{\Sc}_\g}(\widetilde{\bT}_{\ff\g})=\End_{R(\breve{G})}(\widetilde{\bT}_{\ff\g}) \cap \widetilde{\Sc}_{\ff,loc}=\widetilde{\Sc}_\ff.$$
\end{proof}

\begin{rem}
    If the above condition does not hold, then neither may the conclusion. For example, take $G=\mathrm{SL}_2(\C)$ and let $\Lambda_\ff$ (resp. $\Lambda_\g$) consist of a nonzero orbit (resp. the $0$ orbit). Then $$\widetilde{\Sc}_\ff \cong \widetilde{\HH}=R(\breve{G}) \oplus R(\breve{G})e^{-\varpi} \oplus R(\breve{G})H \oplus R(\breve{G})e^{-\varpi}H,$$ $$\widetilde{\Sc}_\g \cong R(\breve{G})\quad\mbox{and}\quad \widetilde{\bT}_{\ff\g} \cong R(\breve{T})=R(\breve{G}) \oplus R(\breve{G})e^{-\varpi},$$ 
    where $\varpi$ is the fundamental weight.
    Thus $\End_{\widetilde{\Sc}_\g}(\widetilde{\bT}_{\ff\g}) \cong \mathrm{Mat}_2(R(\breve{G}))$. But the image of $\widetilde{\HH}$ in $\mathrm{Mat}_2(R(\breve{G}))$ is generated by 
    $$\begin{pmatrix}
    1 & 0 \\
    0 & 1
    \end{pmatrix},\quad
    \begin{pmatrix}
    0 & -1 \\
    1 & e^\varpi+e^{-\varpi}
    \end{pmatrix},\quad
    \begin{pmatrix}
    -q & -q^{-1}(e^\varpi+e^{-\varpi}) \\
    0 & q^{-1}
    \end{pmatrix},\quad
    \begin{pmatrix}
    0 & -q^{-1} \\
    -q & 0
    \end{pmatrix}$$
    as an $R(\breve{G})$-module, which is a proper subset of $\mathrm{Mat}_2(R(\breve{G}))$.
\end{rem}

Thanks to Lemma~\ref{lem:howe}, we immediately obtain the following double centralizer property, which can be regarded as a Howe duality of affine type at the Schur algebra level.
\begin{thm}\label{thm:howe}
If the subsets of minimal parabolic subgroups of $\{P_\mu~|~\mu\in\Lambda_\ff\}$ and of $\{P_\mu~|~\mu\in\Lambda_\g\}$ coincide, then $\widetilde{\Sc}_\ff$ and $\widetilde{\Sc}_\g$ admit a double centralizer property on $\widetilde{\bT}_{\ff\g}$.
\end{thm}
 We refer to \cite{LW22} for the geometric Howe duality of any finite type via a Beilinson-Lusztig-MacPherson type construction. 

%------------------------------
\subsubsection{Howe $(\mathbf{U}(\widetilde{\mathfrak{gl}}_m), \mathbf{U}(\widetilde{\mathfrak{gl}}_n))$-duality}
%-----------------------------------
Let us consider the special case of type A by taking $Q_\ff=Q_m$ and $Q_\g=Q_n$ as in \eqref{def:Qn}, and rewrite the Fork space $\widetilde{\bT}_{\ff\g}$ by $\widetilde{\bT}_{m,n}^d$ in this case. We assume $m,n\geq d$ so that the Borel subgroup $B$ is the minimal parabolic subgroup determined by $\Lambda_{m,d}$, which is also the same by $\Lambda_{n,d}$. Then Theorem~\ref{thm:howe} shows a duality between affine $q$-Schur algebras $\widetilde{\Sc}_{m,d}$ and $\widetilde{\Sc}_{n,d}$. By the surjectivity of $\mathbf{U}(\widetilde{\mathfrak{gl}}_n)\twoheadrightarrow \widetilde{\Sc}_{n,d}$, we further obtain a Howe duality between affine quantum groups $\mathbf{U}(\widetilde{\mathfrak{gl}}_m)$ and $\mathbf{U}(\widetilde{\mathfrak{gl}}_n)$ as follows.  
\begin{thm} For any integers $m,n\geq d$, the affine quantum groups $\mathbf{U}(\widetilde{\mathfrak{gl}}_m)$ and $\mathbf{U}(\widetilde{\mathfrak{gl}}_n)$ admit a double centralizer property on $\widetilde{\bT}_{m,n}^d$:
$$\mathbf{U}(\widetilde{\mathfrak{gl}}_m)\twoheadrightarrow \widetilde{\Sc}_{m,d}\curvearrowright \widetilde{\bT}_{m,n}^d \curvearrowleft\widetilde{\Sc}_{n,d}\twoheadleftarrow \mathbf{U}(\widetilde{\mathfrak{gl}}_n).$$
\end{thm}
We remark that in the case of $m>d$ (resp. $n>d$), it is still valid to replace $\mathbf{U}(\widetilde{\mathfrak{gl}}_m)$ (resp, $\mathbf{U}(\widetilde{\mathfrak{gl}}_n)$) by $\mathbf{U}(\widetilde{\mathfrak{sl}}_m)$ (resp. $\mathbf{U}(\widetilde{\mathfrak{sl}}_m)$) in the above Howe duality. But it is not valid in the case of $n\leq d$ since the map $\mathbf{U}(\widetilde{\mathfrak{sl}}_n)\rightarrow \widetilde{\Sc}_{n,d}$ is no longer surjective. 

%--------------------------------
\subsubsection{Howe duality for affine $\imath$quantum groups}
%------------------------------ 
Now we consider the special case of type B/C by taking $Q_\ff=Q_m^{\mathfrak{c}}$ and $Q_\g=Q_m^{\mathfrak{c}'}$ ($\mathfrak{c,c'}=\imath,\jmath$). We rewrite the Fork space 
$\widetilde{\mathbf{T}}_{\ff\g}$ by $\widetilde{\mathbf{T}}_{m,n}^{d;\mathfrak{c,c'}}$. Assume $m,n\geq d$ so that we can apply Theorem~\ref{thm:howe} to obtain an $(\mathbf{S}^{\imath\mathfrak{c}}_{\mathfrak{c}(m),d}, \mathbf{S}^{\imath\mathfrak{c'}}_{\mathfrak{c'}(n),d})$-duality on $\widetilde{\mathbf{T}}_{m,n}^{d;\mathfrak{c,c'}}$ ($\mathfrak{c,c'}=\imath,\jmath$), where we denote $\jmath(n)=2n+1$ and $\imath(n)=2n$.

By the surjectivity of $\mathbf{K}_{2n+1}^{\imath\jmath}\twoheadrightarrow {_{\mathbb{Q}}\Sc}^{\imath\jmath}_{2n+1,d}$ and $\mathbf{K}_{2n}^{\imath\imath}\twoheadrightarrow {_{\mathbb{Q}}\Sc}^{\imath\imath}_{2n,d}$, we get a Howe duality between affine $\imath$quantum groups as follows.
\begin{thm}
 For any integers $m,n\geq d$, the affine $\imath$quantum groups $\mathbf{K}^{\imath\mathfrak{c}}_{\mathfrak{c}(m)}$ and $\mathbf{K}^{\imath\mathfrak{c'}}_{\mathfrak{c'}(n)}$ admit a double centralizer property on ${_{\mathbb{Q}}\widetilde{\mathbf{T}}}_{m,n}^{d;\mathfrak{c,c'}}(:=\mathbb{Q}(q)\otimes_{\mathbb{Z}[q,q^{-1}]}\widetilde{\mathbf{T}}_{m,n}^{d;\mathfrak{c,c'}})$:
$$\mathbf{K}^{\imath\mathfrak{c}}_{\mathfrak{c}(m)}\twoheadrightarrow {_{\mathbb{Q}}\Sc}^{\imath\mathfrak{c}}_{\mathfrak{c}(m),d}\curvearrowright {_{\mathbb{Q}}\widetilde{\mathbf{T}}}_{m,n}^{d;\mathfrak{c,c'}} \curvearrowleft{_{\mathbb{Q}}\Sc}^{\imath\mathfrak{c'}}_{\mathfrak{c'}(n),d}\twoheadleftarrow \mathbf{K}^{\imath\mathfrak{c'}}_{\mathfrak{c'}(n)},\quad (\mathfrak{c,c'}=\imath,\jmath).$$
\end{thm}

%====================================================================== 

\end{document}